\documentclass[12pt]{article}
\usepackage{amsmath,amssymb,graphicx} %load extra symbols and environments

\usepackage[margin=1in]{geometry} %set margins
\usepackage{amsthm}
\usepackage[backend=bibtex,style=numeric]{biblatex}
\usepackage[toc,page]{appendix}
\usepackage{comment}
\usepackage[colorlinks=true,linkcolor=blue]{hyperref}
\usepackage[usenames,dvipsnames]{xcolor}
\usepackage{mathtools}
\usepackage{enumerate}
\usepackage{enumitem}
\usepackage{tikz,}
\usepackage{float}
\usetikzlibrary{shapes.geometric}
\tikzstyle{circ}=[circle,draw=black,fill=white,inner sep=0,minimum size=0.708cm,text=black,font=\footnotesize]
\tikzstyle{squa}=[rectangle,draw=black,fill=white,inner sep=0,minimum size=0.5cm,text=black,font=\footnotesize]
\tikzstyle{diam}=[diamond,draw=black,fill=white,inner sep=0,minimum size=0.708cm,text=black,font=\footnotesize]
\tikzstyle{semi}=[semicircle,draw=black,fill=white,inner sep=0,minimum size=0.36cm,text=black,rotate = 90,font=\footnotesize]
\tikzstyle{pent}=[regular polygon,draw=black,fill=white,inner sep=0,minimum size=0.708cm,text=black,font=\footnotesize]
\tikzstyle{connection}=[inner sep=0,outer sep=0]

\usepackage{graphics}
\usepackage{bbm}

\newtheorem{theorem}{Theorem}
\newtheorem{question}{Question}
\newtheorem{lemma}{Lemma}
\newtheorem{remark}{Remark}
\newtheorem{proposition}{Proposition}
\newtheorem{definition}{Definition}
\newtheorem{conjecture}{Conjecture}

\numberwithin{equation}{section}
\numberwithin{proposition}{section}
\numberwithin{lemma}{section}
\numberwithin{definition}{section}
\numberwithin{conjecture}{section}

\newcommand{\black}[1]{\color{black}}
\newcommand*{\rom}[1]{\expandafter\@slowromancap\romannumeral #1@}

\graphicspath{ {./figures/} }

\newcommand{\footremember}[2]{%
	\footnote{#2}
	\newcounter{#1}
	\setcounter{#1}{\value{footnote}}%
}

\setlist[enumerate]{font=\bfseries}
\author{Illya Koval \footremember{email}{Illya.Koval@ist.ac.at}}
\title{Billiard tables with analytic Birkhoff normal form are generically Gevrey divergent}
\date{%
	Institute of Science and Technology Austria\\ %
	\today
}

\begin{document}
	\maketitle
	\abstract{The problem of the existence of an analytic normal form near an equilibrium point of an area-preserving map and analyticity of the associated coordinate change is a classical problem in dynamical systems going back to Poincar\'e and Siegel. One important class of examples of area-preserving maps consists of the collision maps for planar billiards. Recently, Treschev discovered a formal $\mathbb{Z}_2 \times \mathbb{Z}_2$ symmetric billiard with locally linearizable dynamics and conjectured its convergence. Since then, a Gevrey regularity for such a billiard was proven in \cite{zhang}, but the original problem about analyticity still remains open. 
		
		We extend the class of billiards by relaxing the symmetry condition and allowing conjugacies to non-linear analytic integrable normal forms. To keep the formal solution unique, odd table derivatives and the normal form are treated as parameters of the problem. We show that for the new problem, the series of the billiard table diverge for general parameters by proving the optimality of Gevrey bounds. The general parameter set is prevalent (in a certain sense has full measure) and it contains an open set.
		
		Instead of considering the problem in a functional sense and iterating approximation procedures, we employ formal power series methods and one-by-one directly reconstruct all the Taylor coefficients of the table. In order to prove that on an open set Taylor series diverges we define a Taylor recurrence operator and prove that it has a cone property. All solutions in that cone are only Gevrey regular and not analytic.}
	\section{Introduction}
	\label{sec1}
	
	The question that we are going to consider in this paper appears in various classes of dynamical systems, from standard maps to geodesic flows. The methods of tackling this question are usually formal and don't use the geometry of the system. So, results in several of those classes will share the common essence, but they can be slightly different in technical details. So, we shall focus on one type of dynamical systems and then try to extend the result to other types. We will comment on those possible extensions below.     
	
	A mathematical billiard is a dynamical system, first proposed by G. D. Birkhoff in \cite{birk} as a playground, where the “the formal side, usually so formidable in dynamics, almost completely disappears and only the interesting qualitative questions need to be considered”. This makes it great for our purposes, so we will focus on billiards. However, we stress that in the proof we don't use any special properties of the billiard dynamics, as opposed to, for example, a standard map.   
	
	Let $\Omega$ be a strictly convex $C^r$ domain in $\mathbb{R}^2$ with $r>3$. Let $x$ be a point in 
	the boundary $\partial \Omega$ and $\varphi$ is angle of a direction $V$ with 
	the clockwise tangent to $\partial \Omega$ at $x$. Let 
	$M := \left\lbrace \right (x, \varphi): x\in \partial \Omega, \varphi \in (0, \pi)\rbrace $. Then, one can consider 
	a billiard map $f : M \rightarrow M$, where $M$ consists of unit vectors with foot $x$ 
	on $\partial \Omega$ and with inward direction $v$. The map reflects the ray from the boundary of 
	the domain elastically, i.e. the angle of the incidence equals the angle of reflection. $f$ is a $2$-dimensional symplectic map. 
	
	One can study the properties of the billiard in a specific table, investigating things like integrability, ergodicity or the number of periodic points. Alternatively, once can look for unique tables with fascinating qualities within the whole space of billiard domains. Since this space is so large, one can expect there to be plenty of such interesting domains. However, their equations are hard to guess, provided they even exist. Hence, when there is a procedure to potentially reconstruct one of such domains, one should check whether it really converges. 
	
	Among those qualities, it is reasonable to consider the local behavior of the billiard around an elliptic periodic point. The simplest $2$-dimensional symplectic map with an elliptic fixed point is a rotation. So, one can look for domains whose dynamics around a periodic point are conjugate to a rotation. We note that we are interested in the conjugacy of the whole neighborhood: billiards often have invariant curves with restricted dynamics conjugate to a circle rotation. Our condition is much stronger. 
	
	If one finds a domain with a conjugacy to a rational rotation, this will give an open set of periodic points for the billiard and hence disprove Ivrii's conjecture \cite{ivrii}. However, since this condition seems too strong (see \cite{keagan} for some progress), we will only consider irrational rotations.
	
	Nevertheless, even a conjugacy to an irrational rotation will make billiard dynamics integrable in the neighborhood of the periodic point. If one finds such a table, it would provide a surprising counterexample to a version of the Birkhoff conjecture. This famous conjecture, posed in \cite{pori}, states that the only integrable domains are ellipses (in non-circular ellipses, there is no conjugacy to a rotation). Several questions about this version of the Birkhoff conjecture we recently stated in \cite{kaloshin2023birkhoff}.  
	
	Still, it is not clear how to navigate within the space of all tables to find such domains. It turns out that if we add more restrictions on the table, this will make the problem more approachable, since we would know exactly what are we looking for. This leads us to the question, studied by Treschev in \cite{tresh}, \cite{tresh2}, \cite{tresh3}: only analytic domains with $2$ orthogonal axes of symmetry are considered. This conditions guarantees the existence of a periodic orbit of period $2$ along both of the axes. 
	
	\begin{figure}
		\includegraphics[width=8cm]{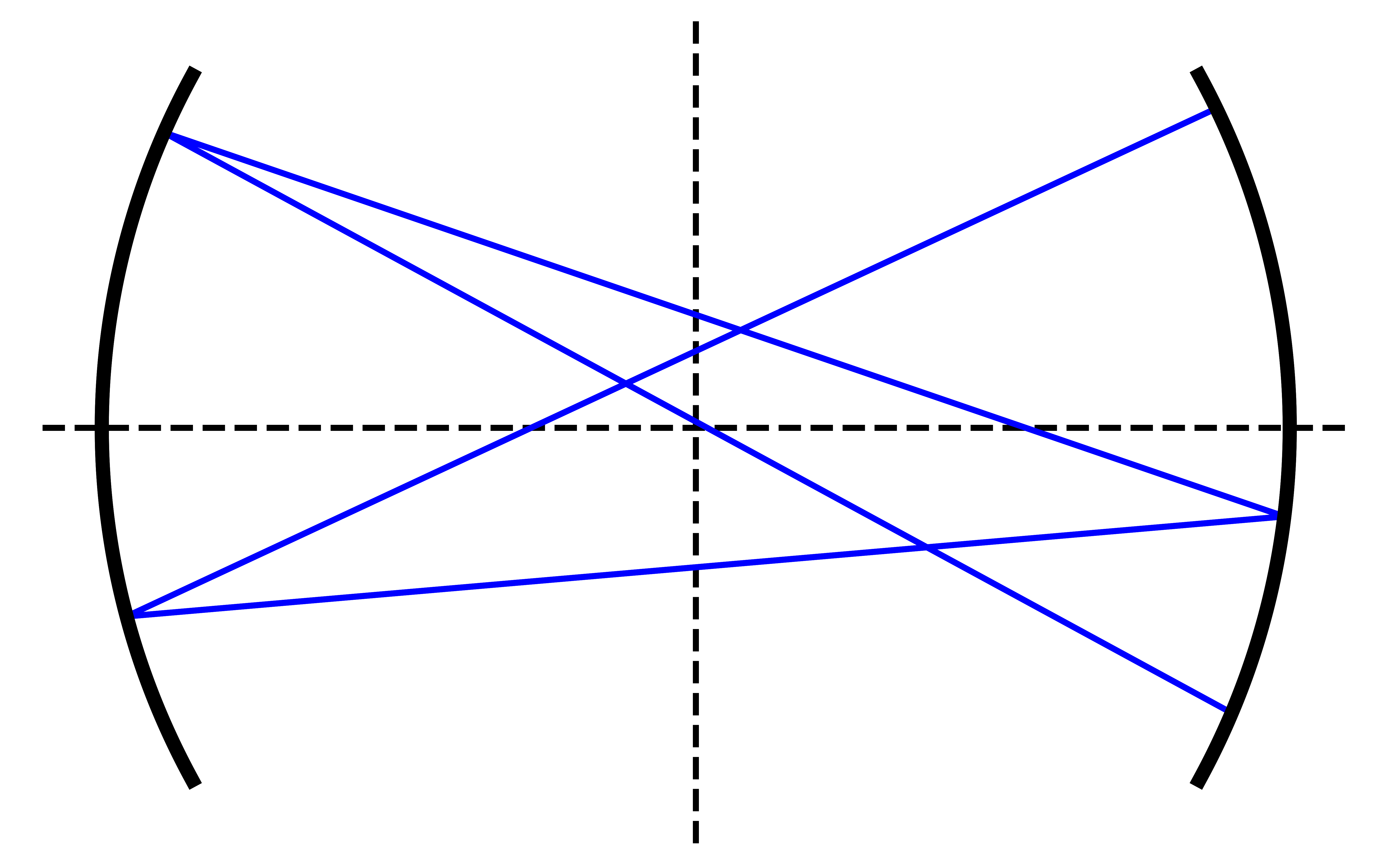}
		\centering
		\caption{Treschev's domain, plotted using its first $25$ power series terms. An orbit coming from the formal conjugacy is also plotted, its error is at most $10^{-37}$. Despite the shape looking like an ellipse, the fourth derivative of the pictured $q(t)$ at $0$ differs from that of the closest ellipse by more than $1$.}
	\end{figure}
	
	\begin{question}[\cite{tresh}]
		Under these assumptions, do there exist tables, whose billiard dynamics are locally analytically conjugate to the irrational rotation around one of those orbits?
		\label{q1} 
	\end{question}

	These restrictions are added to make the problem solvable in formal power series (so, we need analytic conditions). One axis (more accurately, the central symmetry) creates the symmetry between reflection points, so we may only consider one power series for the domain. By \cite{cdv}, the conjugacy condition ensures (more on that later) that for any $n > 1$ the expression that depends linearly on the $2n$-th derivative and polynomially on the lower derivatives (all the odd derivatives are $0$ due to the other axis of symmetry) is zero. This creates the upper-triangular system in even derivatives that can be uniquely solved (\cite{cdv}, \cite{tresh}). Hence, for every rotation number there is exactly one formal power series for the billiard table. So, instead of dealing with the whole domain space, we can check the convergence of some implicit series.   
	
	The first major result on Question \ref{q1} was recently done by \cite{zhang}. In that project, the authors used KAM-theory to investigate the regularity of the formal power series for the table. Under mild assumptions on the rotation number, they have managed to prove $(1+\alpha)$-Gevrey growth of these series. So, if the domain is given by $q(t) = \sum_{n = 0}^{+\infty} q_{2n}t^{2n}$, they have obtained that
	
	\begin{equation}
		|q_{2n}| = O\left(e^{\alpha (2n) \log (2n)}\right), \; n\rightarrow +\infty
	\end{equation}
	
	for $\alpha > 5/4$. To get convergence, one needs at most exponential growth. It was unclear whether this bound could be improved or was it optimal.
	
	Before stating the main theorems, we need to generalize the problem a fair bit. Firstly, being conjugate to a rotation is impressive, since it shows that the local dynamics is integrable, connecting the problem to the Birkhoff conjecture. However, integrable maps are richer than just rotations and they can be conjugated to the maps of type $\phi \rightarrow \phi + b(r^2)$ and $r\rightarrow r$ in polar coordinates, where $b(r^2)$ is a Birkhoff normal form (if $b(r^2) = \sum_{n = 0}^{\infty} b_{2n} r^{2n}$ is constant, we get a rotation by angle $b_0$). Any symplectic map with an irrational rotation number can be formally conjugated to its normal form. In fact, the aforementioned expression is the $n-1$-st derivative of $b(r^2)$, also called the twist coefficient. Thus, we may extend Question \ref{q1} from irrational rotations to being conjugate to dynamics of a given analytic $b(r^2)$.
	
	Secondly, we may consider just centrally symmetric domains instead of considering two axes. This preserves the correspondence between reflection points, but allows for odd derivatives of the table. Then, that coefficient of $b$ will also linearly depend on the $2n-1$-st table derivative. Sadly, this breaks the upper-triangular method, since we doubled the number of parameters, so there is no more uniqueness. To combat this, we assume that the odd power series of the table are given, and the goal is to find the even series. We demand that the odd series are convergent.
	
	Lastly, we don't have to consider only orbits of period $2$, so we make the period $\tilde{q}$ arbitrary. Since we still need the symmetry between reflection points, we will demand that the domain (and the orbit) is rotationally symmetric with angle $2\pi/\tilde{q}$ around the origin. Then, we may consider elliptic orbits of rotation number $\omega = \tilde{p}/\tilde{q}$ for some $\tilde{p}$. We include this regime, since recently there was some interest in rotationally symmetric billiards, see \cite{bialytsodi} and \cite{Ferreira_2024}. We also note that due to the symmetry we identify all the reflection points and consider the normal form of one-iterate, instead of $\tilde{q}$-iterate of the map. The usual normal form can be obtained by scaling all the coefficients of $b(r^2)$ by $\tilde{q}$.  
	
	\subsection{Main results}
	We will refer to the formal even power series of $q(t)$ ($q_{even}(t)$) as the generalized Treschev problem solution (GTPS). The parameters of the problem are the odd series of $q(t)$ ($q_{odd}(t)$), series of $b(r^2)$, and the period $\tilde{q}$. In these statements, $q(t)$ is the support function of the domain, defined later. Also, we remind that we say that $x \in \mathbb{R}\setminus\mathbb{Q}$ has restricted partial quotients, if there exists $C>0$, such that:
	
	\begin{equation}
		|x - m/n| > Cn^{-2}, \; \; \; \forall m/n \in \mathbb{Q}.
	\end{equation}
	
	 We first state the upper bound result for GTP:
	
	\begin{theorem}
		If $q_{odd}(t)$ and $b(r^2)$ are real analytic functions and $b_0/\pi$ has restricted partial quotients, then GTPS is $1+1$-Gevrey:
		\begin{equation}
			|q_{2n}| \le e^{(2n)\log(2n)} e^{Cn}
		\end{equation}
		for some $C \in \mathbb{R}$.  
		\label{th1}
	\end{theorem}  
	
	In specific cases a better bound can be achieved, see \eqref{eq14}. By paying with Gevrey order, one can also relax the Diophantine condition.
	
	However, it turns out that for a large set of parameters, this bound cannot be improved. This follows from the next statement, concerning the lower bound.
	
	\begin{theorem}
		Assume that $b_0/\pi$ is fixed and has restricted partial quotients. Then, for almost every real analytic functions $q_{odd}(t)$ and $b(r^2)$ with given $b_0$ there is a lower $1+1$-Gevrey bound of GTPS. 
		\label{th3} 
	\end{theorem}
	
	\begin{remark}
		Here, the notion of “almost every” is expanded upon in the next theorem. However, one can justify it by considering various measures on the space of analytic functions, induced by the shifts of the Hilbert brick. The details of this construction are presented in \cite{kaloshinhunt}.
	\end{remark}
	
	Since almost every analytic function is a rather implicit condition, we specify this set more accurately in the next theorem. Particularly, we may only look at $2$-dimensional parameter space and assume that $q_3 \ne 0$.
	
	\begin{theorem}
		Assume that $q_{odd}(t)$ and $b(r^2)$ are analytic functions, that $b_0/\pi$ has restricted partial quotients and that $q_3 \ne 0$. Let $n_0$ be large enough. Then, if we consider a family of GTP-s obtained by changing $b_2$ and $q_{2n_0+1}$, than for a full measure set in $\mathbb{R}^2$ of these parameters there is a lower $1+1$-Gevrey bound of GTPS.
		\label{th2}
	\end{theorem} 
	
	\begin{remark}
		The properties of this full measure set come from \cite{perez}. Particularly, the generic point on the plane is in this set.
	\end{remark}

	\begin{remark}
		In Theorem \ref{th3} condition on $b_0/\pi$ can be improved from partial quotients to all Diophantine numbers. Similar improvements can also be done in Theorem \ref{th2} under specific conditions. See \eqref{eq15} for details. 
		\label{rema2}
	\end{remark}
	Theorem \ref{th3} and \ref{th2} show that for a large set of parameters GTPS is Gevrey. However, they give no regularity of this set and don't even provide a single explicitly given element of it. We give open and uniform conditions for it in the next theorem:
	
	\begin{theorem}
		For any $C, D$, for any large enough $n_0$ there exists $C_8$, such that the following holds. Assume we are given $q_{odd}(t)$ and $b(r^2)$ that satisfy the following:
		\begin{itemize}
			\item $|q_{2n+1}| < e^{Dn}$ for any $n > 0$.
			\item $|b_{2n}| < e^{Dn}$ for any $n > 0$.
			\item $e^{C-0.01} < |q_3| < e^{C}$.
			\item $|b_0/\pi - m/n| > \frac{1}{Dn^2}$ for any $m/n \in \mathbb{Q}$.
			\item $\omega > e^{-D}$.
			\item $|b_2| < e^{2C}\Phi(b_0)$ with $\Phi$ explicitly defined later, but independent of $C$ and $D$. 
		\end{itemize}
		Then, if $q_{2n_0+1}$ gets increased by $C_8$, there is a $1+1$-Gevrey lower bound on GTPS. Particularly, we have:
		\begin{equation}
			|q_{2n}| \ge c^n e^{2n\log(2n)} 
		\end{equation}
		for some $c(C, D)$ and large enough $n$. 
		\label{th4}
	\end{theorem} 

	\subsubsection{Cone property for Taylor recurrence operator}
	Our upper-triangular system has many terms with various signs. So, one of the main difficulties of proving Theorem \ref{th4} is dealing with potential cancellations that may arise. We tackle this problem in the following way. 
	
	In our case, the upper-triangular system can be well approximated by a linear recurrence with constant coefficients. To study the properties of the solution, we consider the corresponding linear operator (as one considers $\begin{pmatrix}
		1 & 1 \\ 1 & 0
	\end{pmatrix}$ for Fibonacci sequence). Since the operator tracks the behavior of Taylor coefficients of $q$, we call it the Taylor recurrence operator. Unfortunately, the order of the recurrence is infinite, so we have to deal with the Hilbert space $\ell^2$, where the operator becomes compact after rescaling.

	One would assume that the successive iterates of a vector under this operator will come closer and closer to the largest eigenvalue eigenvector and this will control the cancellations. Unfortunately, this won't always happen (for example if the starting vector was another eigenvector). To deal with this, we introduce an invariant cone around the largest eigenvector. We roughly define the cone as:
	
	\begin{equation}
		\mathcal{C}_\theta=\left\{\mathbf{x} \in \ell^2 | \mathbf{x} = \mathbf{x}_1 + \mathbf{x}_2, \mathbf{x}_1 \in L_1, \mathbf{x}_2 \in L_2, ||\mathbf{x}_1|| > \theta ||\mathbf{x}_2|| \right\},
	\end{equation}

	where $L_1$ is the span of the eigenvector and $L_2$ is the invariant hyperplane, such that $L_1\oplus L_2 = \ell^2$. Then, as long as the solution is inside the cone, it remains there and behaves regularly, allowing us to control the errors and to prove Theorem \ref{th4}. To get the solution into the cone to begin with, we need to change $q_{2n_0+1}$ by $C_8$.  
	\usetikzlibrary {patterns} 
	\begin{figure}
		\centering
		\begin{tikzpicture}[scale=1.0]
			\node [connection] (0) at (-2, -0.3) {};
			\node [connection] (1) at (6, 0.9) {};
			\node [connection] (2) at (-0.2, -0.7) {};
			\node [connection] (3) at (0.4, 1.4) {};
			\node [connection] (4) at (0.3, -0.9) {};
			\node [connection] (5) at (-0.6, 1.8) {};
			\node [connection] (6) at (0.0, -0.6) {};
			\node [connection] (7) at (0.0, 1.2) {};
			\node [connection] (8) at (-1.8, 0.6) {};
			\node [connection] (9) at (3.6, -1.2) {};
			\node [connection] (10) at (-1.8, -1.0) {};
			\node [connection] (11) at (3.6, 2.0) {};
			\node [connection] (-1) at (0, 0) {};
			\node [connection] (30) at (2.4, -0.8) {};
			\node [connection] (31) at (2.7, 1.5) {};
			\node [connection] (40) at (-2, 0) {};
			\node [connection] (41) at (6, 0) {};
			\node [connection] (42) at (-1.8, -0.6) {};
			\node [connection] (43) at (5.85, 1.95) {};
			\fill[color = red!20] (-1.2, 0.4) -- (0, 0) -- (-0.9, -0.5) to[out=120,in=-120] (-1.2, 0.4);
			\fill[color = purple!30] (-1.5, 0) -- (0, 0) -- (-1.5, -0.5) to[out=120,in=-120] (-1.5, 0);
			\draw [red, thick] (8) -- (-1);
			\draw [red, thick] (10) -- (-1);
			\draw [purple, thick](40)--(-1);
			\draw [purple, thick](42)--(-1);
			\draw [<<-, thick] (0) -- (-1);
			\filldraw[blue!30, thick,  rotate=100] (0.5,0) ellipse (2 and 0.75);
			\fill[color = red!20] (2.4, -0.8) -- (0, 0) -- (2.7, 1.5) to[out=-60,in=60] (2.4, -0.8);
			
			\fill[color = purple!30] (4.5, 0) -- (0, 0) -- (4.5, 1.5) to[out=-60,in=60] (4.5, 0);
			
			\draw [red, thick] (-1) -- (9);
			\draw [red, thick] (-1) -- (11);
			\draw [purple, thick](-1)--(41);
			\draw [purple, thick](-1)--(43);
			
			\draw [->>, thick] (-1) -- (1);
			\draw [<->, thick] (2) -- (3);
			\draw [<->, thick] (4) -- (5);
			\draw [<->, thick] (6) -- (7);
			\node [circle, fill, inner sep = 1] (50) at (1.5, -0.3) {};
			\node [circle, fill, inner sep = 1] (51) at (2.5, 0.15) {};
			\node [circle, fill, inner sep = 1] (52) at (2.6, -0.2) {};
			\node [circle, fill, inner sep = 1] (53) at (3.5, 0.3) {};
			\node [circle, fill, inner sep = 1] (54) at (3.6, 0) {};
			\node [circle, fill, inner sep = 1] (55) at (5, 0.5) {};
			\draw [->, thick, darkgray] (50)--(51);
			\draw [->, thick, darkgray] (52)--(53);
			\draw [->, thick, darkgray] (54)--(55);
			\draw [->, thick, violet] (51)--(52);
			\draw [->, thick, violet] (53)--(54);
			\node [connection, text = red] at (4.0, 1.8) {$\mathcal C_{\theta}$};
			\node [connection, text = purple] at (5.5, 1.3) {$T \mathcal C_{\theta}$};
			\node [connection] at (6, 0.5) {$L_1$};
			\node [connection, text = blue] at (0.5, 2.3) {$L_2$};
		\end{tikzpicture}
		
		\caption{Cone property of operator $T$. The dynamics of the point under $T$ is shown with gray arrows and purple arrows denote the errors of approximation.}
	\end{figure}
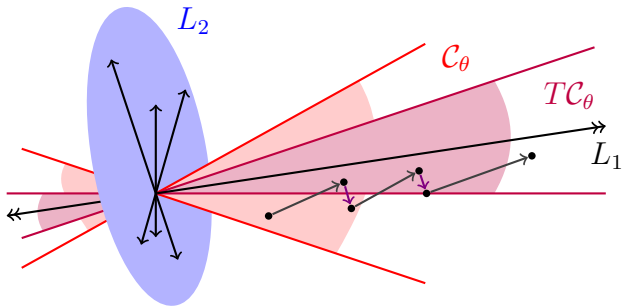

	\subsubsection{Formal conjugacy problem}
	
	Formal series methods were used in dynamics for a long time. One of the most famous results in this field is that by Siegel \cite{siegel}, where the author have proven that the dynamics around a periodic point of wide variety of $1$-dimensional complex dynamical systems is conjugate to the irrational rotation. In the original paper, the formal power series and their convergence are investigated, although since then several other proofs were developed.
	
	Without the complex structure, however, the problem becomes more complicated. In the symplectic case the formal conjugacy between the map and its normal form was shown by Birkhoff in \cite{birk}. Still, the question of convergence was mostly open. It was recently answered by Krikorian in \cite{krikorian}, where it was shown that for a given map the normal form is generally divergent (the divergence of the conjugacy was known before). The proof uses sizes of invariant circles around the stationary point.
	
	A surprising result, that was used by \cite{krikorian}, is that of P\'erez-Marco \cite{perez}. It states that if the convergence of the type of recurrent relation we are studying is investigated, and the dependence of the recurrent elements on the parameters is given by some polynomials, whose degree is bounded, then either we get a convergence for all parameters, or only for a small set of them. The proof is an application of potential theory.
	
	\subsubsection{Relation to other billiard problems}
	
	So, in some sense, our result deals with the opposite problem to \cite{krikorian}. Instead of studying the convergence of the normal form with given $q(t)$, we study the convergence of $q(t)$, given the normal form. Despite this difference, our question offers applications to the inverse spectral problem and integrability questions. The question of determining the domain by its spectrum was asked by Kac in \cite{kac} and has since had many results published about it. Usually, the normal form can be deduced from the spectrum using asymptotic analysis, see, for example, \cite{hezzel}, \cite{leguil}, \cite{osterman} and \cite{z}. In \cite{leguil} and \cite{osterman} the authors manage to recover the normal form, but they cannot recover odd derivatives, so they assume symmetry and use \cite{cdv}. Potentially, our method could help to find the odd derivatives, since other choices would lead to the divergence of even series (although they deal with hyperbolic periodic points). Hence, our inverse problem can have many applications in this direction.    
	
	We have to mention that the periodic points' neighborhoods are not the only local regimes in billiards. One can ask the conjugacy question near the boundary, as well as near the invariant curves, including KAM curves. The boundary question has seen some progress in recent years. Mather's beta function and its formal expansion at $0$ play the role of the Birkhoff normal form in this regime. Instead of being polynomials of the Taylor expansion of $q$ at reflection points, these coefficients can be determined by integrating the differential polynomials of curvature over the boundary. The expressions for some of them were found by Sorrentino in \cite{sorrent}. Instead of linear, they have a quadratic main term, and this was exploited by Melrose \cite{melrose} and Vig \cite{vig} to prove the compactness of the isospectral set for Laplace and length spectrum respectively. We also note that Gevrey growth is conjectured to arise in this regime in a related problem, see \cite{ramirezros}.  
	
	\subsubsection{Standard maps and beyond}
	However, the methods of this paper can be applied not only to billiards. The core of the proof works for other two-dimensional area preserving maps. These maps should have a one-dimensional potential $q(t)$ for us to determine. One of the most studied maps of this type is a standard map, given by:
	\begin{equation}
		\begin{cases}
			x' = x + y' \\
			y' = y + q(x)
		\end{cases}.
	\label{stmap}
	\end{equation}

	If $q(\pi) = 0$ and $-4 < q'(\pi) < 0$, the map will have an elliptic point at $(\pi, 0)$. One can ask the same question for this map:
	\begin{question}
		Assume $q_{even}(x)$ and $b(r^2)$ are analytic parameters of the problem. For which parameters does there exist an analytic $q_{odd}(x)$, realizing the given normal form at $(\pi, 0)$?
	\end{question}

	Here, we count even and odd w.r.t $\pi$, since this is where we are expanding $q(x)$ into series. Note that the role of odd and even are switched in this problem, since the generating function has an antiderivative of $q(x)$, whereas for billiards it has $q(t)$ itself, and parity changes when we differentiate. 
	
	We expect the same results with Gevrey growth to hold in this problem (accounting for parity switch). The plan of the proof should be similar as well -- one difference is that the main equation will have a new form, so one would have to estimate different error terms. But the problem should be formally uniquely solvable and the principal terms should also remain the same, as well as the cones and the recurrence formula.
	
	There are other classes of maps, where it may be harder to apply these methods. An example of such a class would be some types of geodesic flow. The main difficulty of that class is that it is not a discrete map, and so one would have to consider a first return map and try to find a function to expand at one point. Alternatively, one can consider symplectic maps in more dimensions, even including billiards. 
	
	\subsubsection{Trees and formal methods}
	
	One of the main difficulties of studying the Treschev problem lies in the number of terms one has to estimate. To better categorize all those terms, we use the concept of trees. To get a needed coefficient, one just has to sum the weights of all the trees that follow some specific rule. The vertices will represent the expansion coefficients that are featured in the term.
	
	This type of Feynman tree method is widely used in similar problems. For example, the authors of \cite{corsi1} use trees to study quasi-periodic solutions of some Hamiltonian flow around an elliptic point. In their case, they also have a complicated formal equation and they introduce trees to control the terms, add labels and rules to the trees and sum their weights to get the contribution, thus solving the equation formally. Trees also appear in studies of the NLS, see \cite{corsi2} and \cite{denghani}.
	
	\subsubsection{Comments on the theorem assumptions}

	In our main results we have made several assumptions on the parameters of GTP. Now, we are going to explain their importance and what will happen, if we remove them. 
	
	Perhaps the greatest restriction of Theorem \ref{th2} is the demand that $q_3 \ne 0$. This sadly means that we are unable to state the lower bound result for the original problem, and that the question of its convergence remains open. The reason for this is discussed in Appendix \ref{apb} and \eqref{eq521} in detail, along with some numerical data. In short, some recurrent relation degenerates in that case. Everything revolves around $q_3$, since it the first odd derivative ($q_1$ is zero, otherwise there is no orbit). It seems that the odd derivatives in general give growth to the system, and without them there are large cancellations. Our numerical analysis leads us to the following:
	
	\begin{conjecture}
		In the original Treschev problem, if the rotation number is Diophantine, then
		\begin{equation}
			|q_{2n}| \le e^{(n \log n) / k} e^{Cn}
		\end{equation}
		for any given $k > 0$ for some $C(k)$. 
		\label{conj1}
	\end{conjecture}   
	
	In general, even if only $q_3$ degenerates, and assuming some conditions on $b$ (see below), it should be possible to get a better upper bound on $q_{2n}$:
	
	\begin{equation}
		|q_{2n}| \le e^{n \log n} e^{Cn}.
		\label{eq14}
	\end{equation}
	
	Next, we discuss the conditions on $b(r^2)$, including the restricted partial quotients and the choice of $b_2$ in Theorem \ref{th2}. These deal with the small denominator problem, that usually arises in this type of questions. Specifically, while we are solving the equations, denominators of type $\frac{1}{\lambda^k - 1}$ (where $\lambda = e^{ib_0}$ is a map eigenvalue) may appear and contribute to the large growth. To restrict and bound it, $b_0/\pi$ must be Diophantine:
	
	\begin{equation}
		\left|\frac{b_0}{\pi} - \frac{m}{n}\right| > C_{dio} n^{-\tau},
		\label{eq15}
	\end{equation} 

	where $m/n$ is rational and $\tau \ge 2$.
	
	This property gives heavy restrictions on the rotation number, when analytic convergence is studied. However, we are investigating Gevrey growth, so it can be relaxed in the original problem. In fact, there is a correspondence between growth rate and regularity of $\lambda$, and for Gevrey it is enough have an exponential function $e^{-Cn}$ instead of $n^{-\tau}$. This was also pointed in \cite{zhang}. They, however, use KAM methods, where it can be seen more clearly. We use formal direct computations, so this property manifests itself in various cancellations in the method and is not that clear. We discuss it in Appendix \ref{apc}. 
	
	In the proof, we deal with restricted partial quotients numbers, meaning $\tau = 2$. It is a rather strong condition, and many numbers don't satisfy it, but we don't want to complicate the proof by dealing with cancellations. Moreover, the situation changes when we include non-trivial normal form. Once we do this, the cancellations no longer exist, so large $\tau$ may contribute to the faster than $1+1$-Gevrey growth. In fact, the new procedure for growth appears, fueled by $\tau$ and $b_2$. While $\tau = 2$ and $b_2$ is small, it can be controlled, otherwise it will overpower the usual growth procedure. Since we use \cite{perez}, it is enough to prove the bounds for small $b_2$. 
	
	So, if $b(r^2)$ is trivial, then the condition on $b_0$ can be relaxed. Otherwise, if $\tau$ needs to be increased, we should demand that $b_2 = 0$ (and the same for $b_4$, etc., if we need to increase $\tau$ even more). Particularly, by setting several first twist coefficients to zero and to suit the Diophantine condition of $b_0/\pi$ and then applying the argument from \cite{perez} for all of them lets us justify Remark \ref{rema2} for Theorem \ref{th3}. 
	
	\subsubsection{Borel transform of Gevrey series}
	
	Now, we comment on how one can try to make sense of the obtained Gevrey series. First of all, every formal power series for $q(t)$ can be turned into a smooth function around $0$ by Borel's theorem. However, this function is not determined by the series uniquely, so one has to introduce errors and hence is likely to destroy the conjugacy. 
	
	Since we have proven $1+1$-Gevrey regularity, we can go a different path. In some cases, one can do a Borel transform of the series and obtain an analytic function in a sector at the origin, determined by the series. In \cite{sauzin} this procedure is described in detail. In our case, the leading term of the procedure allows for Borel transform, but we don't know how to deal with the errors.
	
	\subsection{Setting and plan of the proof}
	
	We assume that there exists some domain $\Omega$, in which there is a periodic orbit of rotation number $\omega = \frac{\tilde{p}}{\tilde{q}} \le 1/2$. We assume that $\Omega$ is invariant under rotation by $2\omega \pi$ around the origin and that $\Omega$ is analytic in the neighborhood of the reflection points. We introduce the support function of $\Omega$, using the notations in \cite{zhang}:

	\begin{equation}
		q(t) = \sup \left\{\left<z, e^{it}\right>: z \in \Omega \right\}.
	\end{equation}	
	
	Here, we associated the plane with $\mathbb{C}$ and put the symmetry center at the origin. This way, $q(t)$ is analytic and $2\pi/\tilde{q}$ periodic. Specifically, we assume that the support function of $\Omega$ is locally given by
	
	\begin{equation}
		q(t) = \sum_{j = 0}^{+\infty}q_j t^j,
	\end{equation}
	
	where $t = 0$ corresponds to the first reflection point. We choose the scaling constant $q_0 = 1$ and we demand $q_1 = 0$ for the periodic orbit to exist. Next, we assume that in the neighborhood of the orbit, the $\tilde{q}$-iterate of billiard dynamics is conjugated to an elliptic convergent Birkhoff normal form. 
	
	Specifically, there exist functions $\varphi(z, w)$ and $b(r)$, analytic around $0$, such that the map 
	
	\begin{equation}
		B(z, w) = \left(e^{ib(zw)} z, e^{-ib(zw)} w\right), \; \; B^{-1}(z, w) = \left(e^{-ib(zw)} z, e^{ib(zw)} w\right)
	\end{equation}

	generates a sequence $t_k = \varphi(B^k(z, w)) + (2k+1) \pi \omega$ for $k \in \mathbb{Z}$, when $w = \bar{z}$. We use the following notation:
	
	\begin{equation}
		\varphi(z, w) = \sum_{j, k = 0}^{+\infty} \varphi_{j, k} z^j w^k, \; \; b(r) = \sum_{j = 0}^{+\infty}b_{2j} r^j,
	\end{equation}

	where $\varphi_{0, 0} = 0$ and $b_{2j} \in \mathbb{R}, \; j \ge 0$. We also introduce an eigenvalue of the billiard map
	
	\begin{equation}
		\lambda = e^{ib_0}, \; \; |\lambda| = 1.
	\end{equation}
	
	As usual, we will require $\lambda$ to be a Diophantine number, since quantities like $\lambda^k - 1$ will arise in the denominators.

	However, not any choice of $b(r), \varphi(z, w)$ and $q(t)$ will lead to a conjugacy. Particularly we need that the sequence $\left\{t_k\right\}$ creates a billiard orbit. It should form a critical sequence of a generating function. The generating function, introduced in \cite{BIALY2017102} and studied in \cite{bialy}, is given by
	
	\begin{equation}
		S(t, t') = q\left(\frac{t + t'}{2}\right) \sin\left(\frac{t' - t}{2}\right)
	\end{equation}

	Since it is sufficient to check criticality only for $t_0$ for any $z$ and $w = \bar{z}$, we impose the following condition:
	
	\begin{equation}
		S_2(t_{-1}, t_0) + S_1(t_0, t_1) = 0.
	\end{equation}
	
	 First, we want to shift our objects a bit to avoid unnecessary constants. Since we consider the case $q = 2$ as the primary one, we will introduce new notation $\alpha = \frac{\pi}{2} -  \pi\omega$,  $\alpha \in \left[0, \pi / 2\right)$ so that it will be  $0$ for that case.	Next, we make $t_i$ be near $0$:
	 \begin{equation}
	 	t_- = t_{-1} + \pi \omega = \varphi(B^{-1}(z, w)), \; \; t = t_0 - \pi \omega = \varphi(z, w), \; \; t_+ = t_1 - 3\pi \omega = \varphi(B(z, w)).
	 \end{equation}
	 This mostly follows the ideas from \cite{zhang}. After substituting $S$ in the formula and using the symmetry of $q$, we get:
	
	\begin{align}
		q'\left(\frac{t_- + t}{2}\right)\cos\left(\frac{t - t_-}{2} - \alpha\right) + q\left(\frac{t_- + t}{2}\right)\cos'\left(\frac{t - t_-}{2} - \alpha\right) + \\ + q'\left(\frac{t_+ + t}{2}\right)\cos\left(\frac{t_+ - t}{2} - \alpha\right) - q\left(\frac{t_+ + t}{2}\right)\cos'\left(\frac{t_+ - t}{2} - \alpha\right) = 0.
		\label{maineq}
	\end{align} 
	
	This will be one of the main equations of this project (we modify it a bit in \eqref{newmain}). We will expand the left side into formal power series in $z$ and $w$ to find the expansion coefficients of $q_{even}(t)$ and $\varphi(z, w)$ step by step. For example, since $\varphi_{0, 0} = 0$, all $t$-s lack $z^0w^0$ term. So, we can substitute $0$ instead of them and verify that there is no error in $z^0w^0$ term.
	
	We should note that we can consider \eqref{maineq} as an equation in formal power series, since the constant term in series we are substituting is always $0$ (except when we are taking an exponent of $b$, but it isn't a problem). Otherwise, the composition wouldn't have to exist. When we are substituting into a cosine, we use formal power series of cosine at $-\alpha$.
	
	\begin{equation}
		\cos(x - \alpha) = \sum_{j = 0}^{+\infty} c_{j} x^j.
	\end{equation}

	The idea of the proof is to study this step-by-step process. On every step, after expanding \eqref{maineq} and considering the error in $z^jw^k$ term, we will be able to express the specific coefficient $q_{2n}$ or $\varphi_{j, k}$ as a polynomial in terms of lower order coefficients of $q_{even}(t)$ and $\varphi(z, w)$, as well as of parameters $c_j$, coefficients of $b(r^2)$ and $q_{odd}(t)$. So, essentially we get an upper-triangular system as we mentioned earlier, but now it also includes $\varphi(z, w)$. 
	
	The main difficulty with studying this system is that the number of terms is big, so it is hard to control all of them. To deal with this problem, we introduce the concept of trees, so that there is exactly one tree for every term of the system, as stated previously.
	
	It turns out that one doesn't have to focus on all the terms (or trees) for our goals. In that system only a few terms with a rather simple structure can give us $1$-Gevrey growth. All the other terms are at least by an order smaller. So, we consider and solve the system, consisting of only the big (or principal) terms. Then we claim that by adding all the other terms back in, we may treat them as errors that won't greatly affect the growth rate.  
	
	Sadly, this distinction between principal and other terms is asymptotic, so it only works for high order terms of $q$ and $\varphi$ and it is challenging to get control of lower order terms. This creates an issue for the lower bound case and Theorem \ref{th2}: there may be some cancellations in lower orders, so that once we study principal terms for high orders we won't get the needed growth. To avoid this, we have the demand about changing the $q_{2n_0+1}$ coefficient. If we make it relatively large (it would be of order $e^{2n_0 \log n_0}$), then it will overwhelm the contributions of low order terms, letting us to freely consider the reduced system for high orders. This is not a problem for an upper bound, since the cancellations are good for it. 
	
	Finally, let us consider the simple example of formal power series equation to show how Gevrey regularity may arise. Assume we want to invert a function $z+z^2$. There are multiple better ways of doing it, but we consider the following:
	\begin{equation}
		q(z+z^2) = z \Rightarrow \sum_{j = 1}^{+\infty}q_j (z+z^2)^j = z.
	\end{equation}
	The left hand side of this equation is a formal power series in $z$, so we may equate its coefficients to the right side. Looking at $z$ error, we get $q_1 = 1$. Next, looking at $z^2$ error we see that $q_2$ contributes to it and is the highest order that does, so we may express $q_2$ in terms of $q_1$. Looking at $z^3$, we can express $q_3$ in $q_1$ and $q_2$, and so forth. This way, we can find all the coefficients.
	
	In this case we also get an upper triangular system of equations with many terms. We want to focus on the term of dependence of $z^n$ error on $q_{n-1}$ for some $n$. We have to expand $(z+z^2)^{n-1}$, and find the $z^n$ coefficient of this expansion. It will be $n-1$, since we can choose where to put $z^2$. So, we would have that
	
	\begin{equation}
		q_n = -(n-1) q_{n-1} + \ldots.
	\end{equation}

	One can see that this can potentially lead to the factorial growth of $q_n$. Factorial and Gevrey growth are the same, because of the Stirling formula. One also has to consider the dependence of $q_n$ on $q_{n-2}$ and so on, but the Gevrey upper bound will persist. However, in this specific example, there would be many cancellations in the system, so the actual growth rate will be just exponential (since the inverse function is analytic). These Gevrey cancellations are similar in nature to the ones we will see in the original Treschev problem.
	
	\subsection{Plan of the proof}
	
	In Section \ref{sec2} we describe the formal power method to solve \eqref{maineq}. We also modify this equation a bit, giving us \eqref{newmain}. We discuss some inherent symmetries in the equation, allowing us to apply the method. 
	
	Next, in Section \ref{sec3} we define the notion of a tree. It lets us control the terms in formal power expansion easier. We place some restrictions on the structure of the tree and we associate every term of expansion to its tree. Moreover, we define the weight of a tree (that corresponds to the coefficient in front of the term) and we prove several lemmas. For example, we estimate the total number of trees.
	
	In Section \ref{sec4} we study principal trees, ignoring everything else. That reduces our system to a recurrent relation. After a change of coordinates, it turns out that the recurrent has an explicit solution, independent of the parameters of GTP (of $b(r^2)$, including $b_0$, $q_{odd}(t)$ and even of the structure of the equation \eqref{maineq} to some extent). The asymptotic behavior of the solution is studied.
	
	In Section \ref{sec5} we focus on proving Theorem \ref{th1}. In fact, we prove a stronger inductive statement: Lemma \ref{lema51}. The main part of the proof is to estimate the influence of aforementioned error terms on the system. These error terms are split into two classes: ones that feature a high-order element and the ones that don't. The former class is more structured and there are less elements in this class, however they demand better bounds. The latter class is more messy, but the bounds can be worse. So, we do the proof for these two classes separately, by slowly reducing the structure of the tree. The proof is rather technical, but we get that the principal recurrence (Taylor recurrence operator) dominates the system. To justify this, some operator theory is used.
	
	Finally, in Section \ref{sec6} we prove Theorems \ref{th3}, \ref{th2} and  \ref{th4} and a stronger Lemma \ref{lema61}. Once again, the proof involves estimating error terms and the two classes. To prove that there are no cancellations, a stable cone of an operator is introduced.  
	
	\subsection{Acknowledgments}
	The author acknowledges the partial support of the ERC Grant \#885707. He also thanks Vadim Kaloshin for greatly aiding the implementation. He thanks Michael Drmota for helping to study Question \ref{qu2}. The author is also grateful to Ke Zhang, Abed Bounemoura and Matthew Kwan for useful discussions. 
	
	\section{Main equation and symmetries}
	\label{sec2}
	We continue our study of \eqref{maineq}.
	
	First, we claim that $q_2$ can be found by looking at the error in the $z^1w^0$ term. Since we are looking at the linear error term, we may only consider the linear approximations of $\varphi$ and $b$. 
	\begin{remark}
		More generally, since we only take compositions, sums and products of formal power series, we should not consider terms in expansions of $q(t)$, $q'(t)$, $\cos(t)$, $\cos'(t)$, $\varphi(z, w)$ or $b(r)$ of higher order in $z$, $w$ or total than the error we are studying.
		\label{rem1}
	\end{remark}
	 In this case, we can use
	
	\begin{equation}
		q(t) \approx 1, \; q'(t) \approx 2q_2 t, \; \varphi \approx z,\; b(r) \approx b_0.
	\end{equation}

	After substituting it, we get:
	
		\begin{align}
		q_2\left(\lambda^{-1}z + z\right)\cos\left(\frac{z - \lambda^{-1}z}{2} - \alpha\right) + \cos'\left(\frac{z - \lambda^{-1}z}{2} - \alpha\right) + \\ + q_2\left(\lambda z + z\right)\cos\left(\frac{\lambda z - z}{2} - \alpha\right) -\cos'\left(\frac{\lambda z - z}{2} - \alpha\right) \approx 0.
	\end{align} 

	After expanding $\cos$:
	
	\begin{align}
		q_2\cos \alpha\left(\lambda^{-1} + 2 + \lambda \right)  z + \cos \alpha \left(\frac{\lambda - 2 + \lambda^{-1}}{2}\right)z  \approx 0.
	\end{align} 

	Hence, 
	
	\begin{equation}
		q_2 = -\frac{\lambda - 2 + \lambda^{-1}}{2(\lambda^{-1} + 2 + \lambda)}.
		\label{q2}
	\end{equation}
	
	As we will see later, studying $z^0w^1$ error term results in the same condition. Next, we look at the error terms of order $2$ to find $\varphi_{2, 0}$, $\varphi_{1, 1}$ and $\varphi_{0, 2}$. Since the calculation is technical, we put it into Appendix \ref{ap1}.
	
	After computing several expansion coefficients and hence verifying we have no error up to order $3$, we can start to generalize the method for higher coefficients. We will consider \eqref{maineq} as an equation in formal power series, forgetting about geometry of the system. Specifically, given a formal power series for Birkhoff normal form and  $q_{odd}(t)$, the method will determine even coefficients of $q_{even}(t)$, as well as a series of $\varphi(z, w)$ (the latter is not unique, more on that later).

	The method for finding all the coefficients of $q$ will be the following. There will be infinitely many steps: on the $m$-th step ($m \ge 1$), we will consider all the possible monomials $z^jw^k$ for $2m \le j + k \le 2m+1$ on the left side of \eqref{maineq}. Before doing the step, we assume that we know all the $q_j$ for $j \le 2m$ and all the $\varphi_{j, k}$ for $j + k < 2m$. We assume we don't know other coefficients of $\varphi(z, w)$ or $q_{even}(t)$.  Each step consists of three parts:
	
	\begin{itemize}
		\item Consider the error in $z^jw^k$ for $j + k = 2m$. By Remark \ref{rem1}, the only unknown coefficient, potentially influencing this term is $\varphi_{j, k}$. We will show that the dependence is linear and non-trivial, hence we will find this coefficient after solving a linear equation.
		\item Consider the error in $z^{m+1}w^m$. By Remark \ref{rem1}, it potentially can be influenced by $q_{2m+2}$ (as the $2m+1$-st term of the derivative), as well as by $\varphi_{m+1, m}$. It will turn out that the dependency on $\varphi_{m+1, m}$ will be trivial, thus we will find $q_{2m+2}$ after solving a linear equation.
		\item Consider the error in $z^jw^k$ for $j + k = 2m+1$ and $|j - k| > 1$. Among unknowns, only $\varphi_{j, k}$ can influence it, so we will find $\varphi_{j, k}$.
	\end{itemize}
	
	\begin{remark}
		Due to the symmetries of the system, described later, the error in $z^m w^{m+1}$ term will vanish once we nullify the error in $z^{m+1}w^m$. Thus, we are essentially solving a diagonal system of equations, modulo $\varphi_{j, k}$ with $|j - k| = 1$. This motivates the setting of our problem: for example, if we had tried to determine all of $q_i$-s, knowing just $b_i$-s, we would have had a plethora of solutions in formal power series.
		\label{rem2}
	\end{remark}

	\begin{remark}
		Since the choice of $\varphi_{j, k}$ for $|j - k| = 1$ and $j + k > 1$ doesn't affect $q_i$-s, we will set them all to $0$ in a further discussion.
	\end{remark}
	 
	 The described method requires several auxiliary lemmas to work, so we will prove them now.
	 
	 \begin{lemma}
	 	For $j + k \ge 2$, the dependence of the $z^jw^k$ coefficient of the left part of \eqref{maineq} on $\varphi_{j, k}$ is given by
	 	\begin{align}
	 		\varphi_{j, k}\cos \alpha \left(q_2\left(\lambda^{j - k} + 2 + \lambda^{k - j} \right) + \frac{\lambda^{j - k} - 2 + \lambda^{k - j}}{2}\right).
	 		\label{fden}
	 	\end{align} 
 		Particularly, it vanishes for $|j - k| = 1$ due to \eqref{q2}.
 		
 		For $j \ge 2$, the dependence of $z^jw^{j - 1}$ coefficient of the left part of \eqref{maineq} on $q_{2j}$ is given by
 		\begin{equation}
 			2^{-2j + 2}jq_{2j}\binom{2j-1}{j} (\lambda^{-1} + 1)^j (\lambda + 1)^j \cos \alpha.
 			\label{qden}
 		\end{equation}
 	\label{lema21}
	 \end{lemma} 
	\begin{proof}
		We should note that similarly to Remark \ref{rem1}, that if we study a term of order $j$ of a composition and we already take a term of order $j$ in one of the functions in a composition, then we can approximate all other functions of the composition linearly.
		
		Particularly, in the first part we approximate $b(r)$ with $b_0$, and $q(t)$ and $\cos(t)$ linearly. We get: 
		\begin{align}
			q_2\left(t_- + t\right)\cos \alpha - \cos \alpha \frac{t - t_-}{2} + q_2\left(t_+ + t\right)\cos \alpha +\cos \alpha \frac{t_+ - t}{2}.
		\end{align} 
		Now, we substitute $\varphi_{j, k}$:
		\begin{align}
			\varphi_{j, k}z^jw^kq_2\left(\lambda^{j - k} + 2 + \lambda^{k - j} \right)\cos \alpha + \varphi_{j, k} z^jw^k\frac{\lambda^{j - k} - 2 + \lambda^{k - j}}{2}\cos \alpha.
		\end{align} 
	Next, we come to $q_{2j}$. We may only use it in $q'$, and we once again approximate $\varphi$ and $b$ linearly:
	\begin{align}
		2jq_{2j}\left(\frac{\lambda^{-1}z + z + \lambda w + w}{2}\right)^{2j-1}\cos \alpha  + 2jq_{2j}\left(\frac{\lambda z + z + \lambda^{-1}w + w}{2}\right)^{2j - 1}\cos\alpha.
	\end{align} 

	We expand the power:
	\begin{align}
		2^{-2j + 2}jq_{2j}z^jw^{j-1}\binom{2j-1}{j} (\lambda^{-1} + 1)^{j-1} (\lambda + 1)^{j - 1} (\lambda + 2 + \lambda^{-1}) \cos \alpha,
	\end{align}
	and use $(1 + \lambda)(1 + \lambda^{-1}) = \lambda + 2 + \lambda^{-1}$ to finish the proof.
	\end{proof}
	
	We should note that the dependence on $q_{2j}$ in \eqref{qden} is always non-trivial (we will not consider $\lambda = -1$). The dependence of \eqref{fden} on $\varphi_{j, k}$ is also non-trivial, unless
	\begin{equation}
		\lambda + \lambda^{-1} = \lambda^{j - k} + \lambda^{k - j}.
	\end{equation}
	This reduces to $\lambda^{j - k - 1} = 1$ or $\lambda^{j - k + 1} = 1$. As stated, if $\lambda$ is irrational, this will only happen for $|j - k| = 1$. However, we can technically consider our method when $\lambda$ is rational. In that case, we will at some step run into a problem that the there will be no dependence on $\varphi_{j, k}$. Then, if the error at $z^j w^k$ term was non-zero, we will have to end the method and we will have conjugacy up to finite order. Otherwise, we will be able to pick any value for $\varphi_{j, k}$ we desire and continue.
	
	So, when $\lambda$ is irrational, we have verified that the system is indeed upper triangular. Now, we will do some simplifications, involving elements of \eqref{maineq}. We notice that
	
	\begin{align}
		\frac{t_\pm + t}{2} = \sum_{j, k}\frac{\varphi_{j, k}}{2}\left(\left(e^{\pm i b(zw)}z\right)^j \left(e^{\mp i b(zw)}w\right)^k + z^j w^k \right) = \\ = \sum_{j, k}\frac{\varphi_{j, k}z^jw^k}{2} \left(e^{\pm(j - k)ib(zw)} + 1\right) = \sum_{j, k}\varphi_{j, k}z^jw^k e^{\pm \frac{j - k}{2}i b(z w)} \cos \left( \frac{j - k}{2} b(z w)\right).
	\end{align}
	
	Similarly, 
	
	\begin{align}
		\frac{t_\pm - t}{2} = \sum_{j, k}\frac{\varphi_{j, k}}{2}\left(\left(e^{\pm i b(zw)}z\right)^j \left(e^{\mp i b(zw)}w\right)^k - z^j w^k \right) = \\ = \sum_{j, k}\frac{\varphi_{j, k}z^jw^k}{2} \left(e^{\pm(j - k)ib(z w)} - 1\right) = \pm \sum_{j, k}\varphi_{j, k}z^jw^k e^{\pm \frac{j - k}{2}i b(z w)} i \sin \left(\frac{j - k}{2} b(z w)\right).
	\end{align}
	
	Now we note that if we look at an error term of $z^{j_0}w^{k_0}$ and we expand $q$ and $\cos$ into power series and expand $t$ using above formulas, then each term contributing $z^{j_0}w^{k_0}$ will have to choose $j$-s and $k$-s in this sum such that the total sum of $j - k$ picked will be equal to $j_0 - k_0$. Hence, we are able to get $e^{\pm \frac{j-k}{2}ib(zw)}$ out of the composition as $e^{\pm \frac{j_0-k_0}{2}ib(zw)}$ while preserving the error in $z^{j_0}w^{k_0}$ term. Then, terms involving $t_+$ and $t_-$ would group together. If we denote $d_0 = j_0 - k_0$, \eqref{maineq} will turn into the following:
	
	\begin{align}
	\cos\left(\frac{d_0}{2}b(zw)\right) q'\left(\sum_{j, k}\varphi_{j, k}z^jw^k \cos \left( \frac{j - k}{2} b(z w)\right)\right)\cos\left(\sum_{j, k}\varphi_{j, k}z^jw^k i \sin \left(\frac{j - k}{2} b(z w)\right) - \alpha\right) + \\ + i\sin\left(\frac{d_0}{2}b(zw)\right)q\left(\sum_{j, k}\varphi_{j, k}z^jw^k  \cos \left( \frac{j - k}{2} b(z w)\right)\right)\sin\left(\sum_{j, k}\varphi_{j, k}z^jw^k  i \sin \left(\frac{j - k}{2} b(z w)\right) - \alpha\right) =_{d_0} 0.
		\label{newmain}
	\end{align}

	We will use this equation from now on, since it provides some symmetries.
	
	Our next step would be to observe several symmetric properties of determined objects:
	
	\begin{lemma}
		If we follow the method with irrational $\lambda$, we will satisfy the following:
		\begin{itemize}
			\item $\varphi_{j, k} = \overline{\varphi_{k, j}}$ for any $j$, $k$.
			\item $q_{2j}$ are real for any $j$.
		\end{itemize}
	\end{lemma}
	Particularly, we have justified Remark \ref{rem2}.
	\section{Tree method}
	
	\label{sec3}
		
	\subsection{Motivation}
	
	Now, we have presented the method and have shown we can iterate it. Now, we want to obtain some bounds on the resulting sequence of coefficients. Since explicitly substituting series in one another as we did earlier can get rather messy, we introduce a combinatorial tree method to describe all the contributions to the error term.
	
	We start with the following motivation. Let us assume we are trying to find $\varphi_{j, k}$ or some even coefficient of $q$ and hence are studying the $z^jw^k$ coefficient of the left side of \eqref{newmain}. We can substitute our power series into the equation and then expand everything and only leave terms with $z^jw^k$. We do not substitute known series coefficients, except $b_0$ and trigonometric functions that don't involve $\varphi(z, w)$. Then we will get a finite sum. 
	
	\begin{lemma}
	Each of the terms will be of form 
		\begin{equation}
			f(\lambda) q_\beta c_\gamma  \prod_{m = 1}^{m_{\max}} \varphi_{\eta_{m}, \theta_{m}}\prod_{n = 1}^{n_{\max}} b_{\zeta_n}, 
			\label{lemaeq}
		\end{equation}  
	where $f(\lambda)$ is some rational function of $\lambda$ (depending on the term), $\zeta_n$ are positive even and $\eta_m + \theta_m > 0$ for any $m$. Moreover, the following conditions are satisfied:
	\begin{itemize}
		\item $\sum_{m = 1}^{m_{\max}} \eta_{m} + \sum_{n = 1}^{n_{\max}} \zeta_n / 2 = j$ (Order in $z$ must coincide).
		\item $\sum_{m = 1}^{m_{\max}} \theta_{m} + \sum_{n = 1}^{n_{\max}} \zeta_n / 2 = k$ (Order in $w$ must coincide).
		\item $\beta + \gamma - 1 = m_{\max}$ (Substitution into $q$ must be valid).
	\end{itemize}
	\label{lema31}
\end{lemma}

We will visualize this with a following construction. We will place the name of the coefficient we are trying to find as a root of the tree. Then, we will connect this root of the tree with vertices labeled $q_\beta$, $c_\gamma$, $\varphi_{\eta_m, \theta_{m}}$ for any $m$ and $b_{\zeta_n}$ for any $n$. Thus, we will obtain a rooted tree of depth $1$, and $f(\lambda)$ will be proportional to the weight of this tree. We interpret this as saying that the leaves of the tree combine together and influence the value of the root. If we normalize $f(\lambda)$ by dividing it by the minus linear coefficient of Lemma \ref{lema21}, thus obtaining the true weight, we get the following formula for $\varphi_{j, k}$ (or $q$ if $|j-k| = 1$):

\begin{equation}
	\varphi_{j, k} = \sum_T w(T) q_\beta c_\gamma  \prod_{m = 1}^{m_{\max}} \varphi_{\eta_{m}, \theta_{m}}\prod_{n = 1}^{n_{\max}} b_{\zeta_n},
\end{equation}

where we sum over all possible trees. 

\begin{remark}
	While summing, we shall not consider trees that already involve the coefficient we are trying to find. For example, a term $q_2\varphi_{j, k}$ will appear in $z^jw^k$ order. We should not sum over this tree, since it already involves $\varphi_{j, k}$.
	\label{rem4} 
\end{remark}

We can further generalize this construction by introducing trees of arbitrary depth. Since the values of some of the leaves of the trees were found using \eqref{newmain}, they have their own trees, so we will consider trees of depth $2$ by turning the leaves of the main tree into these subtrees.

 We will continue to expand the depth until we have no more leaves to expand. Then, in order to compute the value at the root, we would have to sum over all possible trees of a given order. 
 
 \subsection{Formal tree definition}

The tree definition would involve the notion of an ordered rooted tree -- a tree  with all of the edges oriented away from a specific vertex (the root), with a specified ordering of the children for every vertex. A vertex is a child (descendant) of another vertex, if there is an edge (path) going from the latter to the former. When we draw a tree, we orient the edges from left to right (so the leftmost vertex is a root) and we order the children by their height: the first child is the top one. A subtree of a vertex is subgraph with the vertex and all of its descendants. A leaf is a vertex with no children.  

\begin{definition}
	A potential tree is an ordered rooted tree that can have vertices of $5$ types: circles (corresponding to $\varphi_{j, k}$), boxes ($q_{2j}$), diamonds ($q_{2j-1}$), semicircles ($b_{2j}$) and pentagons ($c_j$) together with integer labels for every non-box vertex, such that the following holds:
	\begin{itemize}
		\item  All the diamonds, semicircles and pentagons are leaves and their labels are positive integers ($i$ for $c_i$, etc). The labels of diamonds are odd and at least $3$, while semicircles' labels are even.
		\item Squares and circles cannot be leaves.
		\item  Circles' labels are integers with absolute value distinct from $1$ ($j - k$ for $\varphi_{j, k}$).
		\item The root of the tree is either a circle or a box.
		\item Every non-leaf vertex has at most one square(box or diamond) and at most one pentagon child. The first non-semicircle child is either a square or a pentagon (denotes whether we differentiate $q$ or $\cos$).
	\end{itemize}

\end{definition}

\begin{figure}
	\centering
	\begin{tikzpicture}[scale=0.6]
		\node [squa] (0) at (-4, 0) {};
		\node [semi,label=center:$4$] (1) at (-2, 2) {};
		\node [pent,label=center:$8$] (2) at (-2, 0) {};
		\node [circ,label=center:$3$] (3) at (-2, -2) {};
		\node [squa] (4) at (0, 0) {};
		\node [squa] (5) at (2, 2) {};
		\node [semi,label=center:$2$] (6) at (2, 0) {};
		\node [circ,label=center:$-3$] (7) at (2, -2) {};
		\node [pent,label=center:$4$] (8) at (4, 2) {};
		\node [diam,label=center:$3$] (9) at (4, 0) {};
		\node [pent,label=center:$1$] (10) at (4, -2) {};
		\draw (0) -- (1);
		\draw (0) -- (2);
		\draw (0) -- (3) -- (4) -- (7) -- (9);
		\draw (4) -- (5) -- (8);
		\draw (4) -- (6);
		\draw (7) -- (10);
	\end{tikzpicture}
	
	\caption{\label{fig1}An example of a formal tree.}
\end{figure}
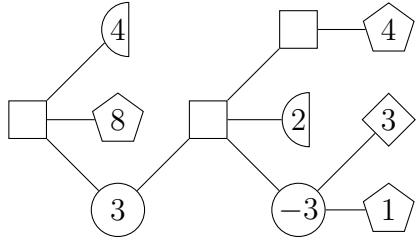 

	The first point tells that we know $q_{2j+1}$, $b_{2j}$ and $c_j$ at the start of the procedure, we don't need \eqref{newmain} to find them. We don't allow $q_0$, $c_0$, $\varphi_{1, 0}$, $\varphi_{0, 1}$ or $b_0$ into the tree to make the tree less complicated. The existence of them in the true term can always be determined from the context. For example, by Lemma \ref{lema31}, if we don't have a pentagon child it means we use $c_0$ (similarly for $q_0$). Since we cannot differentiate constant term, there is no competition for the first non-semicircle child in that case. 
	
	Similarly, we don't indicate the order of box vertices or $j+k$ for the circle vertices, since they can be determined from the context. One can inductively determine them all, together with the number of hidden $\varphi_{1, 0}$ and $\varphi_{0, 1}$ vertices by going from leaves to the root. We know order of leaves and for non-leaf vertices we can find these quantities from their subtree in the following way.
	
	We essentially treat equations in Lemma \ref{lema31} as a system of equations. Since we also know $j - k$ (if the subtree root is circle, then it is the label, and it is $1$ if the root is a box), we have $4$ variables (number of $\varphi_{1, 0}$ and $\varphi_{0, 1}$ we have to add, $j$ and $k$) and $4$ equations. It's easy to see that there is a unique solution, so we can determine the number of hidden vertices and all the orders.
	
	The next three points of definition just say that we are determining all the $q_{2j}$ and $\varphi_{j, k}$ from \eqref{newmain} and that we ignore $\varphi_{j, k}$ with $|j - k| = 1$. The fifth point states that there can only be one coefficient of $q$ and one of $\cos$ in \eqref{lemaeq}.

	This explains the definition of the potential tree, but not every potential tree makes sense and will give a nonzero contribution. For example, no one guarantees that after solving the aforementioned systems one finds that the number of needed $\varphi_{1, 0}$ or $\varphi_{0, 1}$ is a non-negative integer. Alternatively, one can have a tree with $q_3$ and $c_1$ as children while having $10$ circle children. This will not contribute, since the last point of Lemma \ref{lema31} will be violated. Another possible problem is that our tree will try to compute $\varphi_{j, k}$ using $\varphi_{j, k}$, for example. Hence, we introduce some definitions to determine contributing trees.
	
	\begin{definition}
		A non-polygonal child of some vertex in the potential tree is called of square type, if its previous polygonal sibling is a square. If that sibling is a pentagon, we call the vertex pentagonal.
	\end{definition}

	When we are expanding \eqref{newmain}, there will be coefficients of $\varphi$ coming from the $q$ (square) side of the product and ones coming from the $\cos$ (pentagonal) side. We indicate the part where each circle vertex belongs using their position relative to the polygonal vertices. In both groups, we demand that the order of vertices coincides with the order in the expansion, meaning that if in $q_3 \varphi(z, w)^3$ we expanded the first $\varphi$ using $\varphi_{3, 0}$ and the third with $\varphi_{1, 1}$, then $\varphi_{3, 0}$ circle vertex should come before $\varphi_{1, 1}$ in the graph.
	
	We should also comment on the placement of semicircular vertices of our tree. They represent $b_{2j}$, and since there are many parts of \eqref{newmain} where they arise, we should use the order of semicircle children with respect to other children to indicate where they are used in \eqref{newmain}.
	
	 The first usage of $b$-s in the formula is in the cosine or sine in front of $q$. Thus, the respective semicircles should be placed in front of all the other children -- before the polygonal child we differentiate. Once we expand the trigonometric function into a series, we once again demand that an order of semicircles must correspond to the order we placed $b_{2j}$-s in the expansion. We also note that we expand this function not at $0$, but at $d_0b_0/2$ and that we don't need to indicate the order of expansion, since we can just sum up the labels on these semicircles.
	 
	 The second type of $b$-s we use, are those arising in the factor of $\varphi_{j, k}$ we use, when $j+k >1$. We take a similar approach: we place the semicircles right after the respective circle vertex, and we order them as they are in the expansion of the trigonometric function, that we expand at $(j - k)b_0/2$. 
	 
	 The last type of semicircles are the ones, attached to $\varphi_{1, 0}$ and $\varphi_{0, 1}$. Since these vertices are hidden, we cannot place semicircles as we did in the previous group. Hence, we will take unconventional means. We expand:
	 
	 \begin{equation}
	 	\cos \left(\frac{b(zw)}{2}\right) / \cos \frac{b_0}{2} = \sum_{j  = 0}^\infty b_{2j}^c(zw)^j, \; \;  \sin \left(\frac{b(zw)}{2}\right) / \sin \frac{b_0}{2} = \sum_{j  = 0}^\infty b_{2j}^s(zw)^j.
	 \end{equation} 
	
	Since $b$ was analytic, $b^c$ and $b^s$ also are. Now, this type of semicircles will represent coefficients of $b^c$ or of $b^s$. Since we are not allowed to place semicircles with label $0$ in the tree, we normalized, so that $b_0^c = b_0^s = 1$. Now, we place semicircles of $b^c$ between the square child and the next circle child and semicircles of $b^s$ between the pentagon child and the next circle child. Once again, the order of semicircles should coincide with the order we picked in the expansion.
	
	Note that just by looking at the structure of the tree, one can determine what type every semicircle is. If the previous non-semicircle child was a circle -- then it is of second type. If it is a square -- $b^c$, a pentagon -- $b^s$. If all the previous are semicircles -- first type. We note that in Lemma \ref{lema31} $b_j^c$ and $b_j^s$ play the same role as $b_j$.
	
	\begin{remark}
		Now, if we formally expand \eqref{newmain}, then every term of expansion has exactly one tree of depth $1$, representing it. Many terms may have the same trees (we are not indicating the order of hidden vertices), but we are avoiding double-counting. 
	\end{remark}
	
	\subsection{Contributing trees}
	
	As we previously mentioned, the order of the vertex ($j + k$ for $\varphi_{j, k}$ and $2j$ for $q_{2j}$) is not indicated, but can be determined from the tree. Now, we make this determination explicit by defining order using tree combinatorics. 
	
	\begin{definition}
		The order of the vertex is the sum of the labels of leaves in the respective subtree minus twice the number of circle vertices in that subtree, where its root counts as half a vertex.
		\label{orderdef}
	\end{definition}
 
 	\begin{lemma}
 		This definition agrees with an inductive method of computing order, described above.
 		\label{lema32}
 	\end{lemma}
 
 	\begin{proof}
 		We prove by induction going from the leaves to the root. For the leaves, this definition agrees with the sub-index ($j$ for $c_j$, etc). Hence, we only need to prove the statement for non-leaf vertices. Let us assume that there are $x$ of $\varphi_{1, 0}$ and $y$ of $\varphi_{0, 1}$ hidden leaves attached to our vertex. By the first $2$ equations of Lemma \ref{lema31} and by induction, we know that $j + k$ is given by the sum of orders of non-polygonal children and of hidden children. Hence,
 		\begin{equation}
 			j+k-x-y
 		\end{equation}
 		is the total order of non-polygonal children. To compute the combinatorial order of our vertex, we should add all the orders of the children and then subtract the number of circle vertices among our vertex and its children. Since the total order of polygonal vertices is $\beta + \gamma$ in terms of Lemma \ref{lema31}, and since the number of circle children is $m_{\max} - x - y$, the total order is
 		\begin{equation}
 			j + k + \beta + \gamma - m_{\max} - \delta_{\circ} = j + k + 1 - \delta_{\circ},
 		\end{equation} 
 		where $\delta_{\circ}$ indicates if our vertex is a circle. Hence we get $j+k$ for $\varphi_{j, k}$ and $j + (j-1) + 1$ for $q_{2j}$, thus proving the lemma.
 	\end{proof}
 
 	Next, we introduce square and pentagonal sizes 
 	
 	\begin{definition}
 		For a non-leaf vertex, its square and pentagonal size ($s_q$ and $s_c$ respectively) is the number of its circle children of square type and pentagonal circle children respectively.
 	\end{definition}
 
 	\begin{definition}
 		For a non-leaf vertex, its square/pentagonal capacity ($C_q$ and $C_c$ respectively) is the order of its square/pentagon child minus $s_q$ or $s_c$ respectively. If the respective child is the first one, the capacity is reduced by $1$. If the respective child doesn't exist, the capacity is $0$.
 	\end{definition}
 
 	The capacity indicates how many $\varphi_{1, 0}$ or $\varphi_{0, 1}$ children are hidden in square or pentagon part. Finally, we introduce $\delta$, a quantity that indicates the balance between $z$ and $w$.
 	
 	\begin{definition}
 		For a non-leaf vertex, its delta ($\delta$) is the label on the vertex ($1$ for boxes) minus the sum of labels of the circle children.
 	\end{definition} 
 
 	The point of $\delta$ can be seen from subtracting the second equation of Lemma \ref{lema31} from the first one:
	\begin{equation}
		\sum_{m = 1}^{m_{\max}} (\eta_{m} - \theta_{m}) = j - k.
	\end{equation}
	On the right we have a label of the root, and on the left there are the sum of circle children labels, including the hidden children. Since the hidden children are not included in the definition of $\delta$, we have
	\begin{equation}
		\delta = x - y
	\end{equation}
	in the notation of Lemma \ref{lema32}. 
	
	Now we can introduce contributing trees:
	
	\begin{definition}
		A potential tree is called a contributing tree, if the following is satisfied for every non-leaf vertex:
		\begin{itemize}
			\item A box vertex cannot have a box vertex of the same order or a circle vertex of order one less as a child.
			\item A circle vertex cannot have a circle vertex of the same order as a child.
			\item Orders of boxes and circles should be at least $2$.
			\item Square and pentagonal capacities must me non-negative.
			\item $|\delta| \le C_q + C_c$, and $\delta$ must share parity with $C_q+C_c$. 
			\item The number of semicircle children directly after square child should not exceed $C_q$ and the number of semicircle children directly after pentagon child should not exceed $C_c$.
		\end{itemize}
	\label{contrtree}
	\end{definition} 

	\begin{figure}
		\centering
		\begin{tikzpicture}[scale=0.6]
			\node [squa,label=below:$q_{18}$] (0) at (-4, 0) {};
			\node [semi,label=center:$4$,label=left:$b_4$] (1) at (-2, 3) {};
			\node [pent,label=center:$8$,label=below:$c_8$] (2) at (-2, 0) {};
			\node [circ,label=center:$3$,label=below:$\varphi_{5, 2}$] (3) at (-2, -3) {};
			\node [squa,label=below:$q_8$] (4) at (0, 0) {};
			\node [squa,label=below:$q_4$] (5) at (2, 3) {};
			\node [semi,label=center:$2$,label=left:$b_2^c$] (6) at (2, 0) {};
			\node [circ,label=center:$-3$,label=below:$\varphi_{0, 3}$] (7) at (2, -3) {};
			\node [pent,label=center:$4$,label=below:$c_4$] (8) at (4, 3) {};
			\node [diam,label=center:$3$,label=below:$q_3$] (9) at (4, 0) {};
			\node [pent,label=center:$1$,label=below:$c_1$] (10) at (4, -3) {};
			\draw (0) -- (1);
			\draw (0) -- (2);
			\draw (0) -- (3) -- (4) -- (7) -- (9);
			\draw (4) -- (5) -- (8);
			\draw (4) -- (6);
			\draw (7) -- (10);
			\node[draw,align=left] (11) at (-9,0) {$C_q = 0$, $C_c = 6$, $\delta = -2$,\\ $s_q = 0$, $s_c = 1$};
			\node[draw,align=left] (12) at (9,0) {$C_q = 2$, $C_c = 1$, $\delta = -3$,\\ $s_q = 0$, $s_c = 0$};
			\draw[dashed] (0)--(11);
			\draw[dashed] (7)--(12);
		\end{tikzpicture}
		
		\caption{\label{fig2}A unique reconstruction of elements of a contributing tree from Figure \ref{fig1}. For $2$ vertices, their parameters are stated.}
	\end{figure}
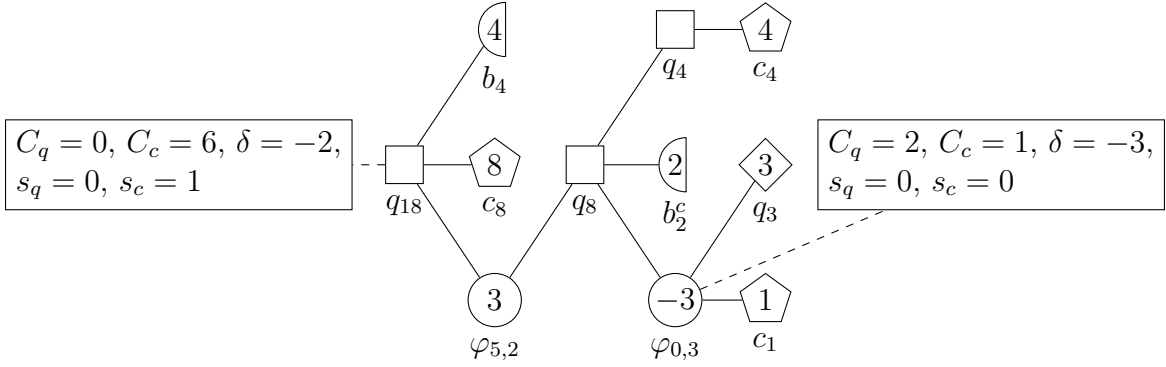
	First two points of the definition reflect our triangular system of equations and address Remark \ref{rem4}: we shouldn't try to find $q_{2j}$ using $q_{2j}$ itself or $\varphi$ of the same step, nor should we find $\varphi_{j, k}$ while using $\varphi_{j, k}$. The third points states that we start our method at order $2$. The next point tells that we shouldn't substitute more $\varphi_{j, k}$ into $q$ or $\cos$, than the order of it. The fourth point guarantees that the solution to a system of equations $x+y = C_q+C_c$, $x - y = \delta$ allows for a solution in non-negative integers. The last point states since every $b^c$ or $b^s$ semicircle corresponds to a hidden vertex, the capacity shouldn't be less than this number.
	
	Now, once we have defined the trees that are contributing, we can finally define a contribution or weight of every such a tree.

	\begin{definition}
		The weight of a leaf the value of the respective coefficient with order given by the label (for semicircles it can be either $b$, $b^c$ or $b^s$ depending on a position within a tree). For non-leaf vertices, the weight is the product of children's weight times the weight multiplier. 
	\end{definition}
	\begin{definition}
		The weight multiplier is the product of base multiplier, combinatorial multiplier and of linear multiplier.
	\end{definition}

	The weight multiplier essentially corresponds to $f(\lambda)$ in Lemma \ref{lema31}. We will define these $3$ types of multiplier in a moment, but we remark that since multiple terms can correspond to the tree, base multiplier consists of factors that are common for every such term. Combinatorial multiplier deals with term-dependent factors, and linear multiplier corresponds to the denominator, similar to Lemma \ref{lema21}.
	
	\begin{definition}
		Base multiplier is the product of the following factors:
		\begin{itemize}
			\item $-1$, if the first non-semicircle child is a pentagon (B1).
			\item The order of the first non-semicircle child (B2).
			\item $i^p$, where $p$ is the order of a pentagon child (B3).
			\item The $m$-th coefficient of an expansion of $\cos$ ($\sin$) at $l b_0 / 2$ times $(l/2)^m$, if the first non-semicircle child is a square (pentagon). Here, $l$ is a label on our vertex ($1$ for boxes) and $m \ge 0$ is the total order of semicircles at the front (B4).
			\item For every circle child of square (pentagon) type, the $n$-th coefficient of an expansion of $\cos$ ($\sin$) at $l b_0 / 2$ times $(l/2)^m$. Here, $l$ is a label on the circle child and $m \ge 0$ is the total order of semicircles right after this child (B5). 
			\item $\cos \alpha$, if there is no pentagon child (B6). 
		\end{itemize}
	\label{basemult}
	\end{definition}
	Next, we define the combinatorial multiplier. Since it will be helpful to know how many hidden vertices of both types have square or pentagonal type, we introduce $x$ as a number of hidden $\varphi_{1, 0}$ of the square. Note that we cannot determine $x$ from the tree itself, so we will have to sum over all values of it. Then, we know that there are $C_q - x$ of square $\varphi_{0, 1}$, and 
	
	\begin{equation}
		\frac{C_c + C_q + \delta}{2} - x \; \; \text{and} \; \; \frac{C_c - C_q - \delta}{2} + x
	\end{equation} 
	of pentagonal $\varphi_{1, 0}$ and $\varphi_{0, 1}$ hidden vertices respectively. Then,
	
	\begin{definition}
		The combinatorial multiplier is the sum over all $x$ of the product of the following factors:
		\begin{itemize}
			\item $\binom{C_q + s_q}{s_q}$ (the number of ways to choose places for hidden square type vertices among all square type vertices) (C1).
			\item $\binom{C_c + s_c}{s_c}$ (the same for pentagonal) (C2).
			\item $\binom{C_q}{x}\binom{C_c}{\frac{C_c + C_q + \delta}{2} - x}$ (the number of ways to choose a place for $\varphi_{1, 0}$ vertices among hidden vertices) (C3).
			\item $\cos(b_0/2)^{C_q} \sin(b_0/2)^{C_c}$ (This comes from the trigonometric factor after $\varphi$ in \eqref{newmain}. For real circle children we have included it into the base multiplier. We are taking zero-order expansion, because we have already considered $b^c$ and $b^s$) (C4).
			\item $(-1)^{\frac{C_c - C_q - \delta}{2} + x}$ (This is also related to this trigonometric factor. When we have pentagonal $\varphi_{0, 1}$, we should expand the sine in $-b_0/2$. Since unlike cosine, sine is odd, we accumulate $-1$ factor) (C5).
			\item $\binom{C_q}{y_q} \binom{C_c}{y_c}$, where $y_q$ ($y_c$) is the number of semicircles directly after a square (pentagon) child. (We can choose which hidden children correspond to these semicircles) (C6).
		\end{itemize}
	\label{combmult}
	\end{definition}
	Finally, to define a definition of linear multiplier, we present an analogue of Lemma \ref{lema21} for \eqref{newmain}:
	\begin{lemma}
		For $j_0 + k_0 \ge 2$, if $j_0 = k_0+1$, then $q_{2j_0}$ arises in \eqref{newmain} with coefficient
		\begin{equation}
			\cos(b_0/2)^{2j_0} 2j_0 \cos \alpha \binom{2j_0 - 1}{j_0}.
		\end{equation}
		When $|j_0 - k_0| \ne 1$, $\varphi_{j_0, k_0}$ comes with coefficient
		\begin{equation}
			2\cos(b_0d_0/2)^2q_2\cos \alpha - \sin (b_0 d_0 / 2)^2 \cos \alpha = \frac{\cos \alpha}{\cos^2(b_0/2)}(\cos^2 (b_0d_0/2) - \cos^2 (b_0/2)).
		\end{equation}
	\end{lemma}
	\begin{definition}
		If our vertex is a box of order $2j_0$, the linear multiplier (L) is
		\begin{equation}
			-\frac{1}{\cos(b_0/2)^{2j_0} 2j_0 \cos \alpha \binom{2j_0 - 1}{j_0}}.
		\end{equation}
	If it a circle with label $d_0$, then this multiplier is
		\begin{equation}
			-\frac{\cos^2(b_0/2)}{ \cos \alpha(\cos^2 (b_0d_0/2) - \cos^2 (b_0/2))}.
		\end{equation}
	\end{definition}
	We also want to prove several lemmas, dealing with structural properties of a tree:
	
	\begin{lemma}
		The order of circle vertices shares parity with their label and the orders of box vertices are even in contributing trees.
		\label{lema34}
	\end{lemma}
	\begin{proof}
		We will prove it using induction. Assume that the root vertex is a circle, the proof is similar if it is a box. The order of the root is the total order of its children minus the number of circles among the root and its children, so minus $s_q + s_c + 1$. Semicircle children have even orders, and the total order of polygonal children is $s_q + C_q + s_c+ C_c - 1$. By induction, the parity of order of the root coincides with the parity of $C_q+C_c$ plus the sum of labels of circle children, itself having the same parity as $C_q+C_c+\delta+l_r$, where $l_r$ is the label of the root. Since $C_q+C_c+\delta$ is even, we have proven the lemma.
	\end{proof}
	
	\begin{lemma}
		In every contributing tree, the sum of the leaves' labels is at least $3$ times the number of circle vertices.
	\end{lemma}
	\begin{proof}
		To every circle vertex, we will associate a set of leaves with a total label of at least $3$. This set will consist of all leaf descendants of the circle vertex $A$, for whom $A$ is the youngest circle ancestor. These sets are disjoint. If we start from $A$ and we will travel in the tree by going into the highest order polygonal child of our vertex until we get into a leaf, this leaf will lie in $A$'s  set. If the order of this leaf is at least $3$, we have proven the lemma. If the order was $1$, then some vertex has polygonal children of only order $1$. Since there is no square of order $1$ by Lemma \ref{lema34}, this vertex's only polygonal child was a pentagon with label $1$. We immediately obtain $s_q = s_c = C_q = C_c = 0$, so it didn't have circle children. If it had a semicircle child, then the order of that child was at least $2$ and we already have total order at least $3$ (pentagon plus semicircle). If it didn't have such a child, its only child would have been a pentagon of order $1$, so the vertex would have order $1$, leading to contradiction.
		
		If the order of that leaf was $2$, then consider the last vertex in our chain of order at least $3$, or the starting vertex itself, if there was no such vertex. It has a polygonal child of order $2$. If it has any other non-circle child, we have finished the proof. Otherwise, $s_q+s_c+C_q+C_c = 1$, so it either had $0$ or $1$ circle children. If it had $1$, then the orders of our vertex and its circle child would've contradicted Definition \ref{contrtree}. If there is no such child, then the order of our vertex is $1$, if it is a circle or $2$ if it is a box, leading to a contradiction.
	\end{proof}
	
	The next statement follows immediately:
	
	\begin{proposition}
		In the tree of order $j$, the total order of leaves doesn't exceed $3j$. Particularly, there can be at most $3j$ leaves. 
		\label{prop31}
	\end{proposition}

	We bound the total size of the tree, using its order:
	
	\begin{lemma}
		The tree of order $j$ can have at most $15j$ vertices.
		\label{lema36}
	\end{lemma}
	\begin{proof}
		The number of leaves in a tree is equal to $1$ plus the sum of degrees minus $1$ over every non-leaf vertex. By the degree of the vertex here we mean its number of children. We claim that many non-leaf vertices should have a degree at least $2$. By Definition \ref{contrtree}, every box vertex of degree $1$ has a diamond child of the same order. Similarly, every circle of degree $1$ has a polygonal child (either a box or a leaf). Since the number of leaves is at most $3j$, the number of vertices with a leaf child should also no exceed $3j$. The number of non-leaf vertices of degree at least $2$ doesn't exceed $3j-1$.
		
		The only vertices yet unaccounted for are the circle vertices of degree $1$ with a box child. Since the number of boxes doesn't exceed $6j$, the are at most $6j$ of them.  
	\end{proof}
	
	\begin{definition}
		We call two contributing trees structurally equivalent, if the only difference between them is the labeling of vertices.
	\end{definition}
	
	\begin{lemma}
		There are at most $e^{75j}$ structural equivalence classes of trees of order $j$.
		\label{lema37}
	\end{lemma}
	\begin{proof}
		First, we forget about the shapes of the vertices. There are at most $2^{4\times15j}$ ordered rooted trees with at most $15j$ vertices. Then, we should pick a shape for every vertex. There are $5$ shapes in total, so the number of classes is bounded by
		\begin{equation}
			2^{60j}5^{15j} \le e^{75j}.
		\end{equation}
	\end{proof}

	\section{Primary trees and their recurrence}
	\label{sec4}
		Before considering all trees and bounding their contributions, we would want to focus on a specific family of trees that we claim give the most important contribution. This means that the weight of such trees would be much greater than of others. We first give some ideas why we are considering these specific trees. 
		
		 Lemma \ref{lema37} indicates that the number of trees is at most exponential, so to get faster growth there should be some heavy trees with large weight. Since the weight of every tree is a product of the weights of leaves and of multipliers at non-leaf vertices, we note that it is impossible to create a large contribution from leaf weights. Since the leaf weights are expansions of analytic functions and because of Proposition \ref{prop31}, the product of leaf weights can be bounded by $C^{3j}$ for some $C > 0$.
		
		Hence we will try to increase the weight multiplier, so that it grows more than exponentially. In the base multiplier one thing that can make it big is the order of the first non-semicircle child, coming from differentiation. Theoretically, we can have a lot of boxes in our tree, so we take a product of all these orders (it doesn't make sense to put many other polygons of high order, since they are leaves and will reduce the size of the tree greatly. One can see this from the proof of Lemma \ref{lema34}: having high-order leaves will improve the coefficient $3$ to a larger one). However, these boxes all come with their linear multipliers, that have the orders of boxes in the denominator (this also comes from differentiation), so B2 in base multipliers wouldn't give the strong growth. We should be careful about $(l/2)^m$ in B4 and B5, since this can generate growth (see Section \ref{sec5} and Appendix \ref{apc}). That is why we placed the bound on $b_2$ in Theorem \ref{th4}.  
		
		Another possibility we can use is the trigonometric factors in the denominator of linear multiplier, making them close to $0$. This will not give strong contributions in the long run, since $\lambda$ is Diophantine. One can also notice that $\cos(b_0/2)$ and $\sin(b_0/2)$ also wouldn't provide real growth. 
		
		Hence strong growth can only come out of binomial coefficients in the combinatorial and linear multipliers. As stated, putting leaves of high order is not helpful, so we may assume that $s_c$ and $C_c$ are small. Next, if we make $x \approx C_q/2$ (and hence $\delta$ near $0$), then the binomial coefficient $\binom{C_q}{x}$ will be of the same order as $\binom{2j_0-1}{j_0}$ in the linear multiplier for small $s_q$ (this follows from the fact that diagonal binomials grow just exponentially).
		
		Then, only $\binom{C_q + s_q}{s_q}$ is left to give growth. And it can give growth, since if for example we will always have $s_q = 1$, then we will essentially multiply the weight by the order of the square. If we put many boxes in a long chain and do this for all of them, then we will get the weight multiplied by the product of their orders. If we wouldn't waste order on other vertices, we can make the orders of boxes to decay linearly, and this will give us the factorial bound. 
		
		In this case it wouldn't make sense to include any other vertices in our graph, since it would speed up the decay. However, since $s_q$ is created by circle children, they should be in our tree, but once again, with the lowest order possible (so $\varphi_{2, 0}$, $\varphi_{1, 1}$ and $\varphi_{0, 2}$).  This motivates the following definition:
		
		\begin{definition}
			A contributing tree is principal at a box vertex $A$, if $A$ has $2\sigma+1$ children: one box child and $2\sigma$ circle children of order $2$, and $2\sigma$ is called a recurrent order.
		\end{definition}
	
		\begin{remark}
			Since the order of $A$ is $2\sigma$ plus the order of its box child, we note that $2\sigma$ is even and positive, since box vertices have even orders by Lemma \ref{lema34}.
		\end{remark}

	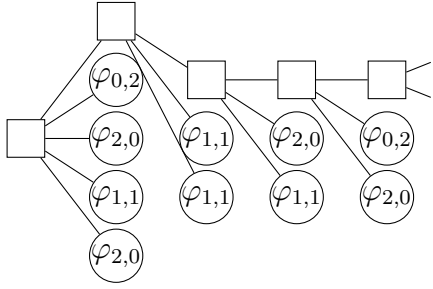
\begin{figure}
		\centering
		\begin{tikzpicture}[scale=0.6]
			\node [squa] (0) at (-4, 0) {};
			\node [squa] (1) at (-2, 2.6) {};
			\node [circ,label=center:$\varphi_{0, 2}$] (21) at (-2, 1.3) {};
			\node [circ,label=center:$\varphi_{2, 0}$] (2) at (-2, 0) {};
			\node [circ,label=center:$\varphi_{1, 1}$] (3) at (-2, -1.3) {};
			\node [circ,label=center:$\varphi_{2, 0}$] (31) at (-2, -2.6) {};
			\node [squa] (4) at (0, 1.3) {};
			\node [circ,label=center:$\varphi_{1, 1}$] (32) at (0, 0) {};
			\node [circ,label=center:$\varphi_{1, 1}$] (33) at (0, -1.3) {};
			\node [squa] (5) at (2, 1.3) {};
			\node [circ,label=center:$\varphi_{2, 0}$] (6) at (2, 0) {};
			\node [circ,label=center:$\varphi_{1, 1}$] (7) at (2, -1.3) {};
			\node [squa] (8) at (4, 1.3) {};
			\node [circ,label=center:$\varphi_{0, 2}$] (9) at (4, 0) {};
			\node [circ,label=center:$\varphi_{2, 0}$] (10) at (4, -1.3) {};
			\node [connection] (41) at (5, 1.7) {};
			\node [connection] (42) at (5, 0.9) {};
			\draw (0) -- (1) -- (4) -- (5) -- (8) -- (41);
			\draw (0) -- (21);
			\draw (0) -- (2);
			\draw (0) -- (3);
			\draw (0) -- (31);
			\draw (1) -- (32);
			\draw (1) -- (33);
			\draw (4) -- (6);
			\draw (4) -- (7);
			\draw (5) -- (9);
			\draw (5) -- (10);
			\draw (8) -- (42);
		\end{tikzpicture}
		
		\caption{\label{fig3}An example of a tree, principal at several vertices.}
	\end{figure} 
	
	Since principal vertices multiply the weight by their orders, the difference between principal and other vertices is greater if the order is larger. Hence, for low orders trees principal at many vertices may not dominate. In this section, we will only study trees that are principal at their box vertices, when the order of the vertex exceeds certain number $j_0$ (we also require the root to be a box).
	
	These trees consist of a long chain of box vertices, going down to order $j_0$, with some circle children of order $2$. We don't care what happens in the subtrees of these circle vertices: we can just consider them as leaves with weights $\varphi_{2, 0}$, $\varphi_{1, 1}$ and $\varphi_{0, 2}$ respectively. Of course, this doesn't work with a tree definition and there are multiple trees with circle root and order $2$, but if identify them and add up these trees' weights, we will get $\varphi$ of order $2$. 
	
	Instead of computing weights of these chain principal trees explicitly, we can compute the total weight of these trees of order $2j > j_0$ (we denote this total weight $\hat{q}_{2j}$) recursively: every root of such a tree has a child, whose subtree can be any chain principal tree of order $2j - 2\sigma$, so if sum over all possible subtrees, we will get $\hat{q}_{2j - 2\sigma}$. The weight of the rest of the tree is explicitly computable and is independent of the choice of aforementioned subtree.
	
	This gives rise to the following linear recurrent relation for $2j > j_0$:
	
	\begin{equation}
		\hat{q}_{2j} = \sum_{\sigma = 1}^{\infty} R_{\sigma, j} \hat{q}_{2j - 2\sigma}.
		\label{recrel}
	\end{equation}
	
	 Our next goal would be to compute $R_{\sigma, j}$. We should consider all the possible configurations of $2\sigma$ circle children of the root, and and up the resulting weights. Assume that the root has $u$, $v$ and $w$ circle children, with labels $2$, $0$ and $-2$ respectively (corresponding to $\varphi_{2, 0}$, $\varphi_{1, 1}$ and $\varphi_{0, 2}$) with $u + v + w = 2\sigma$. Then, we have $\varphi_{2, 0}^u\varphi_{1, 1}^v \varphi_{0, 2}^w$ factor coming from the weights of the circle leaves, and there are $\frac{(2\sigma)!}{u!v!w!}$ ways to shuffle these leaves around, preserving the weight of the tree.
	 
	 Next, we are only left to consider the weight multiplier at the root. We find the related quantities:
	 
	 \begin{equation}
	 	s_q = 2\sigma, \; \; s_c = 0, \; \; C_q = 2j - 4\sigma - 1, \; \; C_c = 0, \; \; \delta = 1 - 2u + 2w. 
	 \end{equation}
 
 	This allows us to find the value of $x$:
 	
 	\begin{equation}
 		x = \frac{C_c+C_q +\delta}{2} = j-2\sigma - u +w,
 	\end{equation}
 
 	otherwise the contribution is automatically zero from combinatorial multiplier. After substituting this into the multiplier definition, we get:
 	
 	\begin{align}
 		R_{\sigma, j} = \sum_{u+v+w = 2\sigma} \varphi_{2, 0}^u\varphi_{1, 1}^v \varphi_{0, 2}^w\frac{(2\sigma)!}{u!v!w!} (2j - 2\sigma) \cos(b_0/2) \cos^{u+w}(b_0) \cos (\alpha) \times \\ \times \binom{2j - 2\sigma - 1}{2\sigma}\binom{2j - 4\sigma - 1}{ j-2\sigma - u +w}\cos(b_0/2)^{2j - 4\sigma - 1} \frac{-1}{\cos(b_0/2)^{2j} 2j \cos \alpha \binom{2j - 1}{j}}.
 	\end{align}
 
 	After simplification, this reduces to
 	
 	\begin{align}
 		R_{\sigma, j} = - \frac{(2\sigma)!(j - \sigma)\binom{2j - 2\sigma - 1}{2\sigma}}{j\binom{2j - 1}{j}}\cos(b_0/2)^{- 4\sigma}\times \\ \times \sum_{u+v+w = 2\sigma} ( \varphi_{2, 0}\cos b_0)^u\varphi_{1, 1}^v (\varphi_{0, 2}\cos b_0)^w\frac{1}{u!v!w!}  \binom{2j - 4\sigma - 1}{ j-2\sigma - u +w} .
 		\label{eq46}
 	\end{align}
	
	Right now, the last binomial coefficient prevents us from evaluating the sum in $u$, $v$ and $w$. As we will see later, we are primarily interested in the case, when $\sigma$ is much lower than $j$. Under these circumstances, we note that this binomial coefficient is very close to the diagonal one. Namely,
	
	\begin{align}
		\binom{2j-4\sigma- 1}{j - 2\sigma - u + w} = \binom{2j - 4\sigma - 1}{j - 2\sigma}\frac{(j-2\sigma)!(j - 2\sigma - 1)!}{(j - 2\sigma - u + w)!(j - 2\sigma + u - w-1)!} = \\ = \binom{2j - 4\sigma - 1}{j - 2\sigma} \prod_{\tau = 0}^{|u - w - 1/2| - 1/2} \frac{j - 2\sigma - \tau}{j - 2\sigma + \tau} =  \binom{2j - 4\sigma - 1}{j - 2\sigma} \left(1 + O\left(\frac{2|u-w|^2}{j - 2\sigma}\right)\right),
		\label{eq49}
	\end{align}

	where the bound of $O$ is uniform for all parameters, provided $j - 2\sigma > 0$. If we ignore the error term and just use the main term, we have:
	
	\begin{align}
	 - \frac{(2j - 2\sigma - 1)!(j-1)!^2(j - \sigma)}{(2j-1)!(j - 2\sigma)!(j - 2\sigma - 1)!}\cos(b_0/2)^{- 4\sigma} \sum_{u+v+w = 2\sigma} ( \varphi_{2, 0}\cos b_0)^u\varphi_{1, 1}^v (\varphi_{0, 2}\cos b_0)^w\frac{1}{u!v!w!}   .
	\end{align}

	The quantity inside the sum is generated by
	
	\begin{equation}
		\exp\left(\varphi_{2, 0} \cos b_0 + \varphi_{1, 1} + \varphi_{0, 2}\cos b_0\right).
	\end{equation}

	Since we are computing terms of total order $2\sigma$, our expression reduces to
	
	\begin{align}
		- \frac{(2j - 2\sigma - 1)!(j-1)!^2(j - \sigma)}{(2j-1)!(j - 2\sigma)!(j - 2\sigma - 1)!}\cos(b_0/2)^{- 4\sigma}(\varphi_{2, 0} \cos b_0 + \varphi_{1, 1} + \varphi_{0, 2}\cos b_0)^{2\sigma}\frac{1}{(2\sigma)!}.
		\label{twosigma}
	\end{align}
	
	The first fraction is equivalent to $\frac{(2j)!}{(2j - 2\sigma)!2^{4\sigma}}$ as $j \rightarrow \infty$ while $\sigma$ is fixed. So, for now we substitute it instead of the fraction. We will deal with the error later:
		
	\begin{align}
		\hat{R}_{\sigma, j} = - \frac{(2j)!}{(2j - 2\sigma)!2^{4\sigma}}\cos(b_0/2)^{- 4\sigma}(\varphi_{2, 0} \cos b_0 + \varphi_{1, 1} + \varphi_{0, 2}\cos b_0)^{2\sigma}\frac{1}{(2\sigma)!}.
		\label{eq413}
	\end{align}		

	Now we will study \eqref{recrel}, where we substitute $\hat{R}$ for $R$. It turns out that the resulting relation can be solved using the generating function. First of all, we perform a linear change of coordinates:
	
	\begin{equation}
		\hat{q}_{2j} = \frac{(2j)!(\varphi_{2, 0} \cos b_0 + \varphi_{1, 1} + \varphi_{0, 2}\cos b_0)^{2j}}{\cos^{4j} (b_0/2) 2^{4j}} \breve{q}_{2j}.
		\label{coorch}
	\end{equation}

	To check if the coordinate change is valid, we have to show that $\varphi_{2, 0} \cos b_0 + \varphi_{1, 1} + \varphi_{0, 2}\cos b_0 \ne 0$. We do that in Appendix \ref{ap1}. It turns out that this value is nonzero, if and only if $q_3$ is nonzero. Hence our method only works if $q_3$ is non-trivial. We denote the scaling factor:
	
	\begin{equation}
		S_{j} =  \frac{j!(\varphi_{2, 0} \cos b_0 + \varphi_{1, 1} + \varphi_{0, 2}\cos b_0)^{j}}{\cos^{2j} (b_0/2) 2^{2j}}
	\end{equation}
	
	The coordinate change reduces the recurrent to
	
	\begin{equation}
		\breve{q}_{2j} = -\sum_{\sigma = 1}^{\infty}\frac{\breve{q}_{2j - 2\sigma}}{(2\sigma)!}.
		\label{newrec}
	\end{equation}

	This equation is linear and has constant coefficients. Hence to solve this equation in general, we can just solve the equation:
	
	\begin{equation}
		\breve{q}_0 = 1; \; \; \; \sum_{\sigma = 0}^{j}\frac{\breve{q}_{2j - 2\sigma}}{(2\sigma)!} = 0, \; \; j > 0.
		\label{recur}
	\end{equation}

	If we go back to formal power series with $\breve{q}(t) = \sum_{j = 0}^{\infty} \breve{q}_{2j}t^{2j}$, then \eqref{recur} states that
	
	\begin{equation}
		\breve{q}(t) \cosh t = 1 \Rightarrow \breve{q}(t) = \frac{1}{\cosh t} = \text{sech} \; t.
		\label{eq418}
		\end{equation}
		
	By the definition of Euler numbers, we have $\breve{q}_{2j} = \frac{E_{2j}}{(2j)!}$. Particularly, since $\cosh t$ has zeros on the complex plane, and its radius of convergence around zero is $\frac{\pi}{2}$, so we get:	
	
	\begin{equation}
		\limsup_{j \rightarrow \infty} \sqrt[2j]{|\breve{q}_{2j}|} = \frac{2}{\pi}.
	\end{equation}

	We are yet to deal with the starting conditions for $\hat{q}$, but if there are good enough, this will correspond to $\breve{q}_{2j}$ depending exponentially on $j$. Then, when we will substitute this into \eqref{coorch}, we will get that the absolute value of some subsequence of $\hat{q}_{2j}$ grows faster than
	
	\begin{equation}
		 \frac{(2j)!(\varphi_{2, 0} \cos b_0 + \varphi_{1, 1} + \varphi_{0, 2}\cos b_0)^{2j}(1 - \varepsilon)^j}{\cos^{4j} (b_0/2) 2^{2j} \pi^{2j}}.
	\end{equation}

	So, this reduced system will produce growth of speed similar to $(2j)!$ with zero radius of convergence.

	\subsection{Problems without the third derivative}
	Here, we will explain why the condition that the third derivative is non-zero is so crucial for us. This explanation will be done for the case $\tilde{p}/\tilde{q} = 1/2$, and when $q_{odd}$ is zero. In that case, we will have $\varphi_{2, 0} = \varphi_{1, 1} = \varphi_{0, 2} = 0$, as can be seen from Appendix \ref{ap1}. 
	
	The same intuition about principal trees will apply in that case, but instead of circle vertices of order $2$, one would have to use circle vertices of order $3$. Particularly, we will have $u$ of $\varphi_{3, 0}$ vertices and $w$ of $\varphi_{0, 3}$ with $2u + 2w = 2\sigma$. 
	
	This will lead to two major changes. First, the factor $\frac{1}{(2\sigma)!}$ in \eqref{twosigma} will be replaced by $\frac{1}{\sigma !}$, since this corresponds to $(u+w)!$. Secondly, instead of having $(2j)!$ in \eqref{coorch}, we will just have $j!$, so the expected growth rate will be much slower compared to our case.
	
		\begin{figure}
		\centering
		\begin{tikzpicture}[scale=0.6]
			\node [squa] (0) at (-4, 0) {};
			\node [squa] (1) at (-2, 1.95) {};
			\node [circ,label=center:$\varphi_{0, 3}$] (21) at (-2, 0.65) {};
			\node [circ,label=center:$\varphi_{3, 0}$] (2) at (-2, -0.65) {};
			\node [circ,label=center:$\varphi_{3, 0}$] (3) at (-2, -1.95) {};
			\node [squa] (4) at (0, 1.3) {};
			\node [circ,label=center:$\varphi_{0, 3}$] (32) at (0, 0) {};
			\node [circ,label=center:$\varphi_{0, 3}$] (33) at (0, -1.3) {};
			\node [squa] (5) at (2, 0.65) {};
			\node [circ,label=center:$\varphi_{3, 0}$] (6) at (2, -0.65) {};
			\node [squa] (8) at (4, 1.3) {};
			\node [circ,label=center:$\varphi_{3, 0}$] (9) at (4, 0) {};
			\node [circ,label=center:$\varphi_{0, 3}$] (10) at (4, -1.3) {};
			\node [connection] (41) at (5, 1.7) {};
			\node [connection] (42) at (5, 0.9) {};
			\draw (0) -- (1) -- (4) -- (5) -- (8) -- (41);
			\draw (0) -- (21);
			\draw (0) -- (2);
			\draw (0) -- (31);
			\draw (1) -- (32);
			\draw (1) -- (33);
			\draw (4) -- (6);
			\draw (5) -- (9);
			\draw (5) -- (10);
			\draw (8) -- (42);
		\end{tikzpicture}
		
		\caption{An analogue of a principal tree, when $q_{odd}$ is trivial.}
	\end{figure} 
	
	The first change, however, will make proving a lower bound much more difficult. Indeed, \eqref{newrec} in this degenerate case will take the form of
	
	 \begin{equation}
	 	\breve{q}_{2j} = -\sum_{\sigma = 1}^{\infty}\frac{\breve{q}_{2j - 2\sigma}}{\sigma!}.
	 	\label{eq521}
	 \end{equation}
	
	This differs from \eqref{newrec} by a little, but the behavior of the solution is really different. We have:
	
	\begin{equation}
		\breve{q}(t)e^{t^2} = 1 \Rightarrow \breve{q} (t) = e^{-t^2}.
	\end{equation} 

	Since $\breve{q}(t)$ has an infinite radius of convergence, $\breve{q}_{2j}$ will decay as $\frac{1}{j!}$, so $\hat{q}_{2j}$ won't have any factorial powers, so they will have a positive radius of convergence. This doesn't give us any information about the convergence of the original problem, since the error terms we removed may kill the convergence once we add them back.

	\section{Tree contribution bounds}
	\label{sec5}
	
	To justify the study of principal trees we should show that the contributions of other trees can be bounded. This will also give us an upper bound for the problem. We note that the bounds in the following statement are motivated by the study of principal trees in the previous section.
	
	\begin{lemma}
		For any $C$, $D$ and $\delta > 0$ there exists $C_1 > 0$ and $C_7>0$, such that the following holds.
		Lets assume we are given $\left\{q_{2n+1}\right\}$ and $\left\{b_{2n}\right\}$ that satisfy the following conditions:
		\begin{itemize}
			\item $|q_{2n + 1}| < e^{Dn}$ for any $n>0$.
			\item $|b_{2n}| < e^{Dn}$ for any $n>0$.
			\item $|q_3| \le e^{C}$.
			\item $|\lambda - e^{2i\pi p /q}| >  \frac{1}{Dq^2}$ for any $p/q \in \mathbb{Q}$.
			\item $|\cos(b_0/2)| > e^{-D}$.
			\item $|\cos \alpha| > e^{-D}$.
			\item $\left| b_2 e^{-2(C+F(\lambda))} C_{dio}^{-1} \right| < 1$, with $C_{dio}$ defined in \eqref{eq15} with $\tau = 2$.
		\end{itemize}
		
		Then, the solution satisfies the following upper bounds:
		
		\begin{equation}
			|q_{2n}| \le e^{W(2n-2, C_1)} (2n)! e^{2n(C + F(\lambda) + \delta)}, \; n > 1,
			\label{qbound}
		\end{equation}
	
		and 
		
		\begin{equation}
			|\varphi_{j, k}| \le |G_{j, k}|  e^{W(j+k-1, C_1)}  (j + k + 1)! e^{(j+k + 1)(C + F(\lambda) + \delta )}, \; j+k>1.
			\label{phibound}
		\end{equation}
		
	Here, $W$ is a function, constant for large $n$,
	
	\begin{equation}
		F(\lambda) = \log \left| \frac{\varphi_{2, 0} \cos b_0 + \varphi_{1, 1} + \varphi_{0, 2} \cos b_0}{q_3}\right| - 2 \log |\cos b_0/2| -  \log (2\pi),
		\label{eq53}
	\end{equation}
		
	and 
	
	\begin{equation}
		G_{j, k} = C_7(j + k + 1)  \binom{j +k}{j} \frac{\cos^{j+k+2}(b_0/2)}{ \cos^2 ((j-k)b_0/2) - \cos^2 (b_0/2)}.
	\end{equation}
	\label{lema51}
	\end{lemma}

	\begin{remark}
		We note that $F(\lambda)$ doesn't depend on $q_3$, as can be seen from Appendix \ref{ap1}, so $q_3 = 0$ also works.
	\end{remark}

	To prove this bound, we will use an inductive method. We prove the statement by the induction in the order of the vertex. The exact values of $\delta$ and $C_1$ will become apparent during the proof.
	
	The first condition that we need to place is that $c_j$ and $q_{2j+1}$ should satisfy the respective version of \eqref{qbound}.
	
	In the first step, we will prove the bound on $q_{2n}$, knowing that the bound applies for $q_{2m}$ with $m < n$ and for $\varphi_{j, k}$ with $j + k < 2n - 1$. To achieve the bound, we need to go over all possible trees of depth $1$ for $q_{2n}$ and to estimate their weight. It will be beneficial to split such trees into several classes and to consider each class separately. 
	
	Before delving into the proof, we have to some remarks about the variables we are using. The proof will work only when $n$ (or $j+k$) is large enough. For small orders the statement is proven by making $W(n, C_1)$ grow fast enough for these small $n$, with $C_1$ being a measure of the speed of growth, chosen after this asymptotic $n$. Since the root of every tree has higher argument of $W$, than every other element of the tree (unless the only child of the circle root is box, but it is exactly accounted for with $G_{j, k}$), we can choose $W$ and $C_1$, so that the statement will hold for small $n$. More requirements on $W$, that focus on its behavior on large $n$ is given after \eqref{eq514}.
	
	 So, $C$, $D$ and $\delta$ are given to us. Various constants in the proof (for example $C_4$ and $C_7$) are chosen after $C$, $D$ and $\delta$. The asymptotic $n$ is chosen after these constants, and $C_1$ is chosen after the asymptotic $n$. During the proof, we will change some of these constants. During each change, we won't use values that are chosen after the changed variable (so for example we can say that we assume $D>1000$ but we cannot say assume $D > C_1$). Moreover, we have to avoid the loops of those changes (we cannot say to double $D$ on every inductive step). 
	
	One can note that we lack the statement for $\varphi_{1, 0}$, $\varphi_{0, 1}$ and $q_2$ in the proof. But these $\varphi$-s cannot be tree elements, so we don't need the bound. $q_2$ can only be bigger than its fake bound by a constant, and all the bounds will feature powers of $n$, so it doesn't affect the proof. 
	
	\subsection{Low children trees}
	
	We will call the first class low children trees. Their main property is that they will not feature any children of high order. Specifically, we call a tree of depth $1$ a low children tree if it has no children of order more than $H(2n)$, where $2n$ is the order of the root. 
	
	Let's fix a low children tree of depth $1$ of order $2n$. Then, we can estimate the product of children's weight using an induction hypothesis. If we apply the bounds, the exponential of $C + F(\lambda) + \delta$ will come with scaling $2n + 2s_q + 2s_c$. Thus, to prove the recurrent formula, we can cancel this exponential and to leave $2s_q+2s_c$ scaling.
	
	Next, there are a few factors in the multiplicative factor that can be bounded by a constant $C_2$: these factors are B1, B3, B6, C5 from Definitions \ref{basemult} and \ref{combmult}, as well as factor $\cos \alpha$ from L. Next, B2 is bounded by $2n$ in L, so we can remove both of them. Since we are bounding coefficients we should only use bounds on absolute value from now on.
	
	Now, we collect $\cos(b_0/2)$ factors. We can use that cosine is analytic and improve the bounds to turn $\sin(b_0/2)$ into $\cos(b_0/2)$ in C4. A $C_q+C_c$ power comes from C4, $-2n$ power comes from L and $3s_q+3s_c$ plus the sum of $j+k - 1$ over all circle children, coming from $G_{j, k}$. Recalling the definition of the order of the vertex and that $C_q+C_c+s_q+s_c$ is the total order of polygonal children minus one, we bound the contribution by $\cos(b_0/2)^{-\omega_b + 2s_q+2s_c}$, where $\omega_b$ is the total order of semicircle children. Since the normal form is analytic, we accommodate the cosine factor into it and forget about it, increasing $D$. 
	
	Next, we bound the trigonometric expansion in B4 and B5 by $1/m!$. We also remove the power of $2$ from there. 
	
	To finish off algebraic quantities, we use a Diophantine property of $\lambda$ to bound the denominator in $G_{j, k}$ by $\frac{1}{D(|j - k| + 1)}$.
	
	Now, after we have removed all the algebraic factors, we come to the combinatorial ones. Since we have removed C5, we can sum over $x$ in C3. We have:
	
	\begin{equation}
		\sum_{x} \binom{C_q}{x}\binom{C_c}{\frac{C_c + C_q + \delta}{2} - x} = \binom{C_q+C_c}{\frac{C_c + C_q + \delta}{2} } \le 2^{C_q+C_c}.
		\label{eq55}
	\end{equation}
	
	Next, we will slowly simplify the structure of the tree. Since we have bounded many factors and hence removed them from considerations, different trees will now have very similar contributions. For example, if we move a circle vertex (together with all semicircles after it) or  one of the semicircles directly after the polygon from the pentagonal part to the square one and vice versa, the contribution will not change much. Specifically, the only terms that change when we move  are C1, C2 and C6:
	
	\begin{equation}
		\frac{(C_q+s_q)!(C_c + s_c)!}{s_q!s_c!(C_q - y_q)!(C_c - y_c)! y_q!y_c!}.
		\label{c1c2c6}
	\end{equation}
	
	We can sum over all those possible trees in the following way. We have an ordered family of $s_q + s_c$ circle children with semicircles behind them and another ordered family of $y_q+y_c$ semicircle vertices. We can draw a barrier in any position in both of those families and bring everything before the line to the square part and everything after -- to the pentagonal part (alternatively first and second polygonal children's parts), while maintaining the original order. This will split all trees in classes, where in every class we maintain the order of vertices, but we place them in different polygonal parts. Essentially in every class the relative order of circle children is the same.
	
	For each class, we sum over \eqref{c1c2c6}:
	
	\begin{equation}
		\sum_{\gamma_1, \gamma_2} \frac{(C_q + s_q)!(C_c+s_c)!}{\gamma_1!(s_q+s_c - \gamma_1)!\gamma_2!(y_q+y_c - \gamma_2)!(C_q + s_q - \gamma_1 - \gamma_2)!(C_c - s_q - y_q - y_c + \gamma_1 + \gamma_2)!}.
	\end{equation}

	Here, $\gamma_1$ and $\gamma_2$ are the positions of the barriers in the respective sets. This sum is computable due to the binomial coefficients' properties and gives
	
	\begin{equation}
		\binom{s_q+s_c+C_q+C_c}{s_q +s_c} \binom{C_q+C_c}{y_q+y_c}.
		\label{eq58}
	\end{equation} 
	
	From now on, we denote $s_0 = s_q+s_c$, $C_0 = C_q+C_c$, $y_0 = y_q+y_c$. We also denote $\omega_q$ and $\omega_c$ the orders of square and pentagonal children and $\omega_0 = \omega_q + \omega_c$.
	
	Now we can do a similar trick for the semicircle children that go after the circle children. We can shuffle these semicircles around to go after different circles, but be should maintain their relative order. The only part of the contribution, that is changing under these actions, is what is left of B5, namely the product of $|l|^m/m!$ over all circle vertices. Note that we took the absolute value since $l$ can be negative.
	
	Denote the total number of after-circle semicircles as $m_0$. Then, if for every circle vertex (with number $i$) we choose a non-negative number $m_i$, such that $\sum_i m_i = m_0$, there would be a unique shuffling of those semicircles, preserving order, such that the number of semicircles after every circle is $m_i$ (where $i$ is the circle's index). Thus we just need to compute
	
	\begin{equation}
		\sum_{m_1 + m_2 + \ldots + m_{s_0} = m_0} \prod_i \frac{|l_i|^{m_i}}{m_i!}. 
	\end{equation}

	This formula can be simplified, since it gives the formal power expansion coefficient of the product of exponentials. Hence, it is equal to 
	
	\begin{equation}
		\frac{\left(\sum_i |l_i|\right)^{m_0}}{m_0!}.
	\end{equation}

	Thus, we can change B5 into this and consider these semicircles, separated from the circles. We can also convolve the positioning of semicircles in front of the first polygon child with this, giving us $+2$ in the power, and we denote the total number of them and $m_0$ as $\mu$. We keep the semicircles, going directly after polygons separate. 
	
	We also bound the weights of all the semicircle leaves themselves. There are $3$ versions of them: $b$, $b^c$ and $b^s$, but they are all analytic and hence their product can be bounded by $D^{2\omega_b}$, where $2\omega_{b}$ is their total order. This gives us:
	\begin{equation}
		\frac{\left(\sum_i |l_i| + 2 \right)^\mu D^{2\omega_b}}{\mu !}, 
		\label{semialmostfinal}
	\end{equation}
	
	To finish off the semicircles, we only need to choose the original ordered sequence of semicircles. Since every order is even, we essentially need to multiply \eqref{semialmostfinal} by the number of ordered lists of positive integers of length $\mu+y_0$ with total sum $\omega_b$. This number is equal to $\binom{\omega_b - 1}{\mu + y_0 - 1} \le 2^{2\omega_b}$.  Hence, the total semicircle contribution is bounded by
	
	\begin{equation}
		\frac{\left(\sum_i |l_i| + 2 \right)^\mu (2D)^{2\omega_b}}{\mu !}
		\label{eq512}
	\end{equation}
	
	Next, we bound the binomial coefficient in (L) by $2^{2n}/(C_3n)$ for some fixed $C_3$. We also note that now we can sum over orders of pentagon and square, such that their total order is $\omega_0$. We just have to sum over the factorials, coming from \eqref{qbound}. We also sum over which child is first. In total, we get $\omega_0!$ with some constant, that goes into $C_3$.
	
	Now, our next goal would be to sum over all possible labels of circle children. For this, we note that $\sum_i |l_i| + 2 \le 2n + 1$, and we also use the last bound in \eqref{eq55}. 
	
	This allows us to sum over labels of each particular circle child. We have to compute the sum 
	
	\begin{equation}
		\sum_{j+k = \sigma} \binom{j+k}{j} D (|j - k| + 1) \le C_4\sqrt{\sigma} 2^{\sigma},
		\label{eq513}
	\end{equation} 

	for $C_4$ that depends only on $D$. Now for every circle vertex we get both $C_4$ and $C_7$, so we redefine $C_4$ as their product. 
	
	We note that now we have several powers of $2$ in our terms:
	
	\begin{equation}
		\frac{2^{C_0}\prod_i 2^{\sigma_i}}{2^{2n}} \le \frac{1}{2}.
		\label{eq514}
	\end{equation}

	This $1/2$ goes to $C_3$. Thus we have canceled all of powers of $2$ yet. 
	
	Our next goal in simplifying the tree structure is to sum over all possible orders of circles, while keeping the total circle order, as well as the number of circles fixed. The terms, that will change under such operations are the following. First, we have the factorial and $W$ from \eqref{phibound}. Next, there is $(\sigma+1)$ factor, coming from $G_{j, k}$ and $\sqrt{\sigma}$ from \eqref{eq513}. We bound the product of factorial, $(\sigma + 1)$ and $\sqrt{\sigma}$ by $(\sigma + 3)!$. However, to sum over all orders, we will have to demand several conditions on $W(n, C_1)$:
	
	\begin{itemize}
		\item $W(n, C_1)$ is non-decreasing in $n$.
		\item $W(n, C_1)$ is constant for large $n$.
		\item $W(0, C_1) = 0$.
		\item $W(n, C_1) + \frac{1}{100000} \log \Gamma(n + 5)$ is a convex function. 
		\item $W(n, C_1) + \frac{1}{100000} \log \Gamma(n + 1)$ is a convex function.
	\end{itemize}
	
	This allows us to prove several lemmas, that would be helpful when summing the aforementioned products of factorials.
	
	\begin{lemma}
		For any $m \ge 2$, there is the following bound:
		\begin{equation}
			\sum_{a, b \ge 1}^{a+b=m} e^{W(a, C_1)+W(b, C_1)}(a+4)!(b+4)! \le 4800 e^{W(m - 1, C_1)+W(1, C_1)}(m + 3)!
		\end{equation}
		\label{lema52}
	\end{lemma}
	
	\begin{proof}
		First, we note that due to the properties of $W$ we can use Karamata's inequality to get:
		\begin{equation}
			e^{W(a, C_1)+W(b, C_1)}((a+4)!(b+4)!)^{1/100000} \le e^{W(m - 1, C_1)+W(1, C_1)}(5!(m+3)!)^{1/100000},
		\end{equation}
	
		when $a+b = m$. Hence, we only need to prove that
		
		\begin{equation}
			\sum_{a, b \ge 1}^{a+b=m} ((a+4)!(b+4)!)^{99999/100000} \le 40 (120)^{99999/100000}(m + 3)!^{99999/100000}.
		\end{equation}
		
		We can only consider terms, for which $a \le b$. Since we only consider half of the terms, we should add a factor of $2$ to the right hand side. Next, we bound every term of the sum in the following way:
		\begin{equation}
			\frac{(a+4)!(b+4)!}{5!(m+3)!} = \prod_{j = 1}^{a-1} \frac{5+j}{b + 4 + j} \le \prod_{j = 1}^{a - 1} \frac{5+j}{m/2 + 4 + j} \le \left( \frac{5+m/2}{m/2 + 4 + m/2}\right)^{a-1}\le \left(\frac{13}{14}\right)^{a-1}.
		\end{equation}
	
		The last assertion follows when $m \ge 3$, and when $m = 2$, the statement is trivial. Substituting this to our inequality, we are only left to prove that
		\begin{equation}
			\sum_{a \ge 0} \left(\frac{13}{14}\right)^{99999a/100000} \le 20,
		\end{equation}
		
		which is trivial.
	
	\end{proof}

	In our situation, there can be many circle vertices, so there will be more than $2$ factorials. Moreover, we are considering low children trees, so the maximal order (and hence the maximal factorial is bounded). Hence, we introduce the following lemma:
	
	\begin{lemma}
		Let $k > \frac{2m}{3}$ and let $2 \le s \le m$. Then,
		\begin{equation}
			\sum_{1 \le a_1, \ldots, a_s \le k}^{a_1+\ldots+a_s = m} \prod_{i = 1}^{s}e^{W(a_i, C_1)}(a_i+4)! \le (9600)^s e^{W(k, C_1) + W(m - k - s + 2, C_1) + (s - 2)W(1, C_1)}(k + 4)!(m - k - s+ 6)!,
 		\end{equation} 
 		if $k < m - s + 2$ and 
 		\begin{equation}
 			\sum_{1 \le a_1, \ldots, a_s \le k}^{a_1+\ldots+a_s = m} \prod_{i = 1}^{s}e^{W(a_i, C_1)}(a_i+4)! \le (9600)^s e^{W(m - s + 1, C_1)+(s-1)W(1,C_1)}(m - s +5)!,
 		\end{equation}
 		if $k \ge m - s + 2$. 
 		\label{lema53}
	\end{lemma}  
	\begin{proof}
		In the latter case the proof just involves applying Lemma \ref{lema52} $m - 1$ times. First, one sums over $a_1$ and $a_2$, then over their sum and $a_3$ and so on. The bound will be even better than stated.
		
		In the former case, the situation is a bit more involved. We claim that since $k > \frac{2m}{3}$, it is always possible to split up every summed tuple of $a_i$-s into two, such that the sum in each sub-tuple would not exceed $k$. To show this, we first form two tuples (each of size $1$), consisting of the largest two elements. Then, go over other elements and add the to any tuple, as long at it will fit. If some element $a_i$ doesn't fit in both, it means that $a_i$ plus the sum of $a$-s over each tuple exceeds $\frac{2m}{3}$. So, $2a_i$ plus the sum of $a$-s over both tuples exceeds $\frac{4m}{3}$. Since the sum of all $a$-s is $m$, we have $a_i > m/3$. But that would mean that $a_i$ is one of two largest elements, leading to contradiction. 
		
		So, every tuple has at least one splitting. Next, we fix a splitting and sum over all possible tuples, for which the splitting is valid. There are at most $2^s$ possible splittings, so we add this factor to the right side.
		
		Inside of both sub-tuples we once again apply Lemma \ref{lema52} several times. At the end, for every splitting we bound the sum with
		
		\begin{equation}
			(4800)^{s-2} e^{W(A, C_1) + W(B, C_1)+(s - 2)W(1, C_1)}(A+4)!(B+4)!,
		\end{equation}  
		where $A$ and $B$ are the sum of $a$-s in the respective sub-tuple minus the size of sub-tuple plus $1$. ($A+B = m - s + 2$). Due to the same convexity argument we see that we can bound this by substituting $A = m - k - s + 2$ and $B = k$, thus proving the lemma.  
	\end{proof}

	In our case, we have $s_0$ circle vertices. Instead of treating their orders as $a_j$, we will treat their orders as $a_j+1$, since order $1$ corresponds to a "default" vertex. We will denote the sum of orders of circle vertices minus $s_0$ as $\omega_{\varphi}$. We can also set $k$ in Lemma \ref{lema53} to be equal to $H(2n) - 1$. Thus we can apply this lemma to sum products of $(\sigma+3)!$ and $W$ over all orders of the circle vertices with a given sum. 
	
	We get a factor of $(9600)^{s_0}$, but there is also $C_4^{s_0}$, coming from \eqref{eq513}, so we can just redefine $C_4$ and get rid of $9600$. Moreover, we can once again use the convexity property of $W$ to only consider $W(\omega_\varphi)$. Thus, our sum over all orders (with $C_4$) is bounded by
	
	\begin{equation}
		C_4^{s_0} e^{W(\omega_{\varphi}, C_1)} \omega_{\varphi}!^{1/100000} \left((H(2n) + 3)!(\omega_{\varphi} - H(2n) - s_0 + 7)!\right)^{99999/100000},
		\label{eq523}
	\end{equation}

	when $H(2n) \le \omega_{\varphi} - s_0 + 2$ and by
	
	\begin{equation}
		C_4^{s_0} e^{W(\omega_{\varphi}, C_1)} \omega_{\varphi}!^{1/100000}(\omega_{\varphi} - s_0 + 5)!^{99999/100000},
		\label{eq524}
	\end{equation}

	otherwise.
	
	Since we have summed over the orders, there is only a finite number of parameters left, namely: $\omega_0$ -- the total polygonal order, $s_0$ -- the number of circle vertices, $\omega_{\varphi}$ -- the order, coming from circles, $y_0$ -- the number of polygonal semicircles, $\mu$ -- the number of other semicircles, $2\omega_b$ -- the total semicircle order. Not including \eqref{eq523} and \eqref{eq524}, the rest of the contribution divided by the needed bound is given by
	
	\begin{equation}
		\frac{e^{-W(2n-2, C_1) + W(\omega_0 - 2, C_1)}\omega_0!}{(2n)!}e^{2s_0(C + F(\lambda) + \delta)}C_3n\frac{(\omega_0 - 1)!}{s_0!(C_0 - y_0)!y_0!}\frac{(2n+1)^\mu(2D)^{2\omega_b}}{\mu!}.
		\label{eq525}
	\end{equation} 

	There is also an equality on our parameters:
	
	\begin{equation}
		\omega_0 + \omega_{\varphi} + 2\omega_b = 2n.
		\label{eq526}
	\end{equation}

	We continue simplifying. We note that $\mu \le \omega_b < n$ (every semicircle's order is at least $2$), so increasing $\mu$ by one will divide the contribution by something, not exceeding $n$, at will multiply it by $2n+1$, so the contribution will increase at least in $2$ times. However, $\mu$ cannot be extended indefinitely, since $\mu+y_0 \le \omega_b$. Hence, we may assume that $\mu = \omega_b - y_0$ and just multiply $C_3$ by $2$. 
	
	Next, we use the bound
	
	\begin{equation}
		\frac{1}{y_0!\mu!} = \frac{1}{y_0!(2\omega_b - y_0)!} \le \frac{1}{\omega_b!^2} \le \frac{2^{2\omega_b}}{(2\omega_b)!}. 
	\end{equation}

	Then, the only places in \eqref{eq525}, where $y_0$ remains is in $(C_0 - y_0)!$ and in $(2n+1)^{-y_0}$, since we have substituted $\mu$. We know that $2n+1 > C_0$, hence decreasing $y_0$ always increases the contribution. So $y_0 = 0$ has the biggest contribution, and there are at most $n$ other possibilities for $y_0$, so we add an $n$ factor and set $y_0 = 0$. Then, \eqref{eq525} reduces to 
	
	\begin{equation}
		\frac{e^{-W(2n-2, C_1) + W(\omega_0 - 2, C_1)}\omega_0!}{(2n)!}e^{2s_0(C + F(\lambda) + \delta)}C_3n^2\frac{(\omega_0 - 1)!}{s_0!(\omega_0 - s_0 - 1)!}\frac{(2n+1)^{\omega_b}(4D)^{2\omega_b}}{\omega_b!}.
		\label{eq528}
	\end{equation} 

	Now, we remove all the $W$ factors. Once again using convexity, we get:
	
	\begin{align}
		e^{W(\omega_0 - 2, C_1) + W(\omega_\varphi)}(\omega_0 - 1)!^{1/100000}\omega_{\varphi}!^{1/100000} \le\\ \le e^{W(\omega_0+\omega_{\varphi} - 2, C_1)}(\omega_0 + \omega_{\varphi} - 1)!^{1/100000} \le e^{W(2n - 2, C_1)}(2n - 1)!^{1/100000}.
	\end{align}
	
	Thus, after using this, the total contribution becomes
	
	\begin{equation}
		\frac{(\omega_0-1)!^{99999/100000}}{(2n)!^{99999/100000}}e^{2s_0(C + F(\lambda) + \delta)}C_3C_4^{s_0}n^2\frac{\omega_0!}{s_0!(\omega_0 - s_0 - 1)!}\frac{(2n+1)^{\omega_b}(4D)^{2\omega_b}}{\omega_b!}.
		\label{eq530}
	\end{equation}

	multiplied by
	
	\begin{equation}
		\left((H(2n) + 3)!(\omega_{\varphi} - H(2n) - s_0 + 7)!\right)^{99999/100000} \; \; \text{or} \; \; (\omega_{\varphi} - s_0 + 5)!^{99999/100000}.
		\label{eq531}
	\end{equation}

	We can guarantee that $n$ is large enough and that $2n - H(2n) < \sqrt[5]{n}$. This way, we can bound the contribution, when $\omega_{\varphi} \ge H(2n) + s_0 - 2$. Then,
	
	\begin{equation}
		\frac{\omega_0!}{s_0!(\omega_0 - s_0 - 1)!} \le \omega_0 2^{\omega_0}.
		\label{eq533}
	\end{equation} 

	When we use this bound, $s_0$ only remains present in \eqref{eq531} and in $C_4^{s_0}$, so we can multiply everything by $n$ and set $s_0 = 1$ (since $\omega_{\varphi}>0$ we cannot have $s_0 = 0$), and add a $C_4^{2n - H(2n)}$ factor.  Next, we claim that it is always beneficial to increase $\omega_b$ by $1$ and decrease $\omega_0$ or $\omega_\varphi$ by $2$ (provided $\omega_0 \ge 4$ and $\omega_{\varphi}$ doesn't fall below the threshold).  
	
	So, we add another $n$ factor for the choices of $\omega_0$ and $\omega_{\varphi}$ and set them to the minimal values possible under such procedure. Then, $\omega_0$ factorial and the last factorial in \eqref{eq531} are bounded by a constant, going to $C_3$, and we forget about $\omega_b$ factorial. Then, everything is bounded by
	
	\begin{equation}
		\frac{1}{(2n)!^{99999/100000}}C_3n^3(2n+1)^{n - H(2n)/2}(4DC_4)^{2n - H(2n)} (H(2n) + 3)!^{99999/100000}.
		\label{eq534}
	\end{equation}

	We set $H(2n) \le 2n - \log n$. Then, this is smaller than $n^{-\frac{\log n}{4}}$, provided $n$ is large enough.
	
	We can proceed similarly in the latter case of \eqref{eq531}. Note that if $\omega_0 < n/10$ and $\omega_\varphi < n/5$, then the contribution can be easily bounded, because a coefficient $2$ in \eqref{eq526} makes $\omega_b$ not an efficient way of increasing the contribution. Specifically, all $\omega_0$ and $\omega_\varphi$ factorials in the numerators of \eqref{eq530} and \eqref{eq531} can be bounded by $(2n)!^{3/10}$, next $s_0!(\omega_0 - s_0-1)!$ can be bounded by $1$ and $(2n+1)^\omega_b(4D)^{4\omega_b}$  by $(2n)!^{6/10}$. So, the total contribution will be smaller than $\frac{1}{(2n)!^{1/10}}$.
	
	If $\omega_0 \ge n/10$, then for large $n$ increasing $\omega_0$ by $2$ and decreasing $\omega_b$ by $1$ increases contribution geometrically. So, we add some constant to $C_3$ and apply this transformation as many times as we can. At the end, we will have either $\omega_b = 0$ or $\omega_0 > H(2n) - 2$. Similarly, if $\omega_{\varphi} \ge n/5$ and $\omega_0 < n/10$, then $s_0 < n/10$, and hence the argument of factorial in \eqref{eq531} is greater than $n/10$. So, increasing $\omega_\varphi$ by $2$ and decreasing $\omega_b$ by $1$ increases contribution geometrically. Once again, we either get $\omega_b = 0$ or $\omega_{\varphi} > H(2n) + s_0 - 4$. The latter case is already bounded, so now we either have $\omega_b = 0$ or $H(2n) - 2 < \omega_0 \le H(2n)$.
	
	We start with the latter case. Then, since $s_0 \le \omega_{\varphi} \le 2n - H(2n) \le \sqrt[5]{n}$, we see that increasing $s_0$ by $1$ at least doubles the contribution ($(\omega_0 - s_0 - 1)!$ is the main force). So, we may assume that $s_0 = \omega_{\varphi}$. Then, $5!$ from \eqref{eq531} goes to $C_3$. Next, because of the same factorial, we can once again start decreasing $\omega_b$ by $1$ and increasing $\omega_\varphi$ by $2$, until the former becomes $0$. Then, the total contribution will be of form
	
	\begin{equation}
		\frac{(\omega_0 - 1)!^{99999/100000}}{(2n)!^{99999/100000}}C_3C_4^{\omega_\varphi}n^2\frac{\omega_0!}{\omega_\varphi! (\omega_0 - \omega_\varphi - 1)!}.
	\end{equation} 

	Now, if we substitute $\omega_\varphi = 2n - H(2n)$ or $2n - H(2n) + 1$ and $\omega_0 = H(2n)$ or $H(2n - 1)$, we can easily bound this with
	
	\begin{equation}
		C_3C_4^{2n - H(2n)}n^4 \frac{(2n)^{\frac{2n - H(2n)}{100000}}}{(2n - H(2n))!}
	\end{equation}

	If we let $H(2n) = 2n - n^{1/50000}$, this should be less than 
	
	\begin{equation}
		n^{-\frac{n^{1/50000}}{200000}}
	\end{equation}

	for large $n$.

	Now we study the case of $\omega_b = 0$. We only need to show that the contribution will be maximized when either $\omega_0$ or $\omega_\varphi$ will be near $H(2n)$. Then, if $\omega_0 > \omega_{\varphi} - s_0 +10$, increasing $\omega_0$ by $1$ and decreasing $\omega_{\varphi}$ by $1$ up until $\omega_{\varphi} = s_0$ increases the contribution. Then, we can bound everything like in the previous case.

	Next, we can try to decrease $s_0$ by $1$. If we do this, the expression will get multiplied by 
	
	\begin{equation}
		\frac{(\omega_{\varphi} - s_0 + 6)^{99999/100000}s_0}{C_4(\omega_0 - s_0)}.
		\label{eq537}
	\end{equation} 
	
	Note that if $s_0 = 1$, then the expression is trivially maximized, when $\omega_0$ or $\omega_{\varphi}$ is $H(2n)$, so we can always decrease $s_0$ (we don't worry about turning it negative). Then, if $s_0 \ge n^{1/100000}$, the numerator of \eqref{eq537} will exceed $\omega_{\varphi} - s_0 + 6 > \omega_0 - s_0$, so decreasing $s_0$ will increase the contribution. Hence, we consider $s_0 < n^{1/100000}$. Then, we know that $\omega_0 > 2n - H(2n) \ge n^{1/50000}$. Hence, if we now try to increase $\omega_{\varphi}$ by $1$ and decrease $\omega_0$ by $1$, we will get that it will increase contribution, if
	
	\begin{equation}
		\frac{(\omega_{\varphi} - s_0 + 6)^{99999/100000}}{(\omega_0 - 1)^{99999/100000}} \times \frac{\omega_0 - s_0 - 1}{\omega_0} > 1.
	\end{equation}

	The last fraction is bounded from below by $1/2$. If $\omega_0$ was much smaller than $\omega_\varphi$, then the first fraction would be bigger than $2$ and so doing this procedure would increase the contribution. Since we want to maximize the contribution, the first fraction cannot exceed $2$ and hence $\omega_0 > \omega_\varphi / 3$. So, $n/2 < \omega_0 < \omega_\varphi + 10 < 3n/2$. Under these conditions, we can bound 
	
	\begin{equation}
		\frac{\omega_0!C_4^{s_0}}{s_0!(\omega_0 -s_0 - 1)!} \le (2n)^{n^{1/100000}}.
		\label{eq539}
	\end{equation}  

	But then the rest turn into:
	
	\begin{equation}
		\frac{(\omega_0 - 1)!^{99999/100000}(\omega_\varphi - s_0 + 5)!^{99999/100000}}{(2n)!^{99999/100000}}
	\end{equation}

	which decays exponentially in $n$, when both $\omega_{\varphi}$ and $\omega_0$ are proportional to $n$, and \eqref{eq539} grows slower than exponentially, so we get an exponential bound on these middle terms. 
	
	\subsection{Trivial tree estimate}
	
	Before studying the trees with large children, we should first prove some trivial bound (much weaker than Lemma \ref{lema51}). We do this, because in our current proof of Lemma \ref{lema51} we are only studying the case of large $n \ge n_0$ for some $n_0$. However, we don't have any estimate for elements with order less than $n_0$. The only reason why Lemma \ref{lema51} works for small $n$ is because for every $n_0$ we can pick $C_1$ to satisfy the bound for small orders.
	
	An example why this is a problem can be seen from the following: If we have a principal tree, then switching $k$ circles of order $2$ to $1$ of order $1+k$ decreases the contribution for large $n$, since we multiply the contribution by a constant, depending on $k$ and divide by a power of $n$. However, we don't currently know if this large $n$ is uniform over all $k$. We cannot use a $C_1$ bound, since $C_1$ itself depends on the choice of $n$.  
	
	Hence, we prove the following bound:
	
	 \begin{lemma}
	 	Under the conditions of Lemma \ref{lema51} and assuming $D>100$, the absolute values of $q_{2n}$ and $\varphi_{j, k}$ are bounded by:
	 	\begin{equation}
	 		m^{90m}e^{70Dm},
	 		\label{eq542}
	 	\end{equation}
 		where $m$ is $2n$ or $j+k+1$ respectively. 
 		\label{lema54}
	 \end{lemma}
 
 	\begin{proof}
 		We sum the contribution of every tree that contributes to $q_{2n}$ ($\varphi_{j,k}$). Denote $m$ to be the maximal possible order of the vertex of the tree ($m = 2n$ or $m = j+k+1$). By Lemma \ref{lema37}, there are at most $e^{75m}$ structural classes of trees. Within each class, we should add a label to the vertices. The labels of root vertices can go from $1$ to $m$ and the labels of circles can go from $-m$ to $m$. So, there are at most $2m+1$ ways to label each vertex. By Lemma \ref{lema36} the number of all labelings does not exceed $(2m+1)^{15m} \le (em)^{15m}$.
 		
 		Next, we estimate the total weight of all the leaves in our tree. By Proposition \ref{prop31}, the total order of leaf vertices doesn't exceed $3m$. Since every leaf vertex of order $j$ is bounded by $e^{Dj}$ (for semicircles and diamonds by the statement of Lemma \ref{lema51}, for pentagons -- since $D>0$), the product of leaves weights is bounded by $e^{3Dm}$.
 		
 		Finally, we have to estimate all the product of multipliers over every non-leaf vertex. B1, B3, B6, C5 are all bounded by $1$. B2 is bounded by $m$ for each vertex, so at most $m^{15m}$ in total. The product of all B4 and B5 can be bounded by $m^{15m}$, since labels are bounded by $m$, expansion coefficients of since or cosine -- by $1$, and the total order of semicircles in the tree cannot exceed $15m$. C3 can once again be bounded by $2^{C_0}$, and we bound C4 with $(\cos (b_0/2))^{C_q}$. 
 		
 		In L, we bound $2j_0$ by $1$ and we also add a $\cos^{-15m}\alpha$ to the result add remove $\cos \alpha$ factor. In L for the circle, we remove $\cos^2(b_0/2)$ factor and bound the cosine difference by $2Dm$. In total, they can at most give $(2Dm)^{15m}$ factor. In L for the box, we bound $\binom{2j_0 - 1}{j_0}$ with $2^{2j_0-1} / (2j_0)$. These $2j_0$ at most give $m^{15m}$ in total. 
 		
 		We bound C1, C2 and C6 with
 		
 		\begin{equation}
 			m^{s_q+s_c+y_q+y_c}.
 		\end{equation}
 		Since the sum of $s_q$, $s_c$, $y_q$ and $y_c$ over all tree vertices cannot exceed the size of the tree, we once again have $m^{15m}$ bound. 
 		
 		We can bound $2^{C_0}$, left from C3 the following way by $2^{\omega_q + \omega_c}$. The sum of $\omega$-s, corresponding to leaves is bounded by $3m$. For boxes, this power of $2$ cancels out with the power of $2$ in their L with only $2^1$ left. So, in total we bound with $2^{3m}\times2^{15m} = 2^{18m} \le e^{18m}$.
 		
 		Finally, we do a similar trick with $\cos(b_0/2)$, bounding $C_q$ with $\omega_q - s_q - 1$. These $-1$ give us at most $-15m$ power over the whole tree. If the square child is a diamond, we bound the rest with $1$. If it is a box, then $\omega_q$ cancels out with $\omega_q$ in its L, and $s_q$ gives us at most $-15m$. Lastly, we bound $\cos(b_0/2)$ in L of the root with $-m$.   
 	\end{proof}
	
	\subsection{Large child trees}
	
	Our next goal would be to show that trees with large children (trees of depth $1$ with one child of order at least $H(2n)$, where $2n$ is the order of the tree) have small bounded contribution, unless it is a principle tree -- one box child of order at least $H(2n)$ and all the other children are circles of order $2$. After that, we will separately compute the contribution of these principal trees.
	
	We note that since $H(2n) > n$, there could be only one large child in a tree. We also note that the cases of this child being a semicircle, pentagon or a diamond we discussed in the small children part, so we only have to deal with box or circle large child case (we have only considered the total polygonal order being small, but since diamonds and pentagons are analytic we may only consider the large box case). We will consider these at the same time, denoting separate bounds for box case as $\boxed{*}$ and for circle as $\circledast$.

	A lot of the bounds would be similar to the ones we did previously, but we will need more accurate bounds in some specific parts. 
	
	One important difference would be that instead of using inductive bounds from Lemma \ref{lema51} to bound small circle or box children, we will use bounds, coming from Lemma \ref{lema54}. Note that we can always increase $D$ to be greater than $100$. For the large child we will still use the inductive bounds.
	
	However, to have some similarity to the small children case, instead of using \eqref{eq542} for small children, we use
	
	\begin{equation}
		|\varphi_{j, k}| \le (j+k+1)^{90(j + k + 1) - 1}e^{70D(j+k+1)}e^{(j+k+1)(C+F(\lambda) + \delta)}\cos^{j+k+2}(b_0/2) 2^{j+k}
		\label{eq5444}
	\end{equation}

	and 
	
	\begin{equation}
		|q_{n}| \le (n)^{90n}e^{70Dn} e^{n(C+F(\lambda) + \delta)}.
		\label{eq545}
	\end{equation}

	We can always increase $D$ to make those estimates true for every order.
	
	First, identically to the low children case, we bound B1, B3, B6, C5 and a factor $\cos \alpha$ from L with $C_2$. We also cancel out the exponents with $F$, leaving only $2s_q+2s_c$ scaling. Next, $\boxed{*}$: B2 is bounded by $2n$ from L, so we cancel them. $\circledast$: $(j +k + 1)$ from the $G$ of the big circle is bounded by $2n$ form L, we cancel them, we also bound B2 by $n^{1/50000}$.
	
	We deal with $\cos(b_0/2)$ factors in the same way, as well as with B4 and B5. $\circledast$ We also bound the denominator from $G$ of the big circle child in the same way. We sum over $x$ with \eqref{eq55}. However, in $\boxed{*}$, instead of using the last inequality of \eqref{eq55} we use
	
	\begin{equation}
		\binom{C_q+C_c}{\frac{C_c + C_q + \delta}{2} } \le \frac{2^{C_q+ C_c}}{\sqrt{C_q+C_c}} \le  \frac{2^{C_q+ C_c+1}}{\sqrt{2n}}. 
		\label{eq541}
	\end{equation}

	The last inequality holds, since $s_q+s_c \le n^{1/40000}$ and hence $C_q+C_c > n$.
	
	Next, we sum over putting circles in different parts and semicircles into different circles, giving us \eqref{eq58} and \eqref{eq512}. Moreover, in our case the sum of orders of small circle children ($\omega_{s, \varphi}$) cannot exceed $2n - H(2n) + s_0$. Since $s_0$ is bounded by $\omega_{s, \varphi} / 2$, we have that 
	
	\begin{equation}
		\sum_{i} |l_i| \le \omega_{s, \varphi} \le n^{1/45000},
	\end{equation}  

	where the sum is taken only over small circle children. However, $\circledast$ the label on the large circle child is also bounded by $n^{1/45000}$, since $\delta$ is bounded by the order of polygonal children by Definition \ref{contrtree}, so $\delta \le n^{1/45000}$. So, by definition of $\delta$ and triangle inequality, we have that the label on the big circle is also bounded. 
	
	Hence, we can bound \eqref{eq512} by
	\begin{equation}
		n^{\mu /40000}(2D)^{2\omega_b}.
	\end{equation}
	
	Next, we can sum over the labels of circle vertices. For small circle vertices, all the labels have the same bounds, so we just multiply by $(j+k+1)$ for every circle. It cancels out with $(j+k+1)^{-1}$ from \eqref{eq5444}. For the big circle in $\circledast$ we need another bound. Particularly, $D(|j - k| + 1)$ in \eqref{eq513} can be bounded by $n^{1/40000}$, the number of possible labels also by $n^{1/40000}$ and the binomial by the maximal binomial. So, in total we have
	
	\begin{equation}
			n^{1/20000} \binom{j+k}{\left[\frac{j+k}{2}\right]} \le n^{1/20000} \frac{2}{\sqrt{2n}}.
			\label{eq544}
	\end{equation}
	
	We bound the binomial coefficient in L with $2^{2n - 1}/(\pi \sqrt{2n})$. Then, the powers of $2$ cancel out, like in \eqref{eq514}, and this square root of $2n$ cancels out with a square root in either \eqref{eq541} in $\boxed{*}$ or \eqref{eq544} in $\circledast$.  
	
	We have that $W$ factor for the root is no less than a $W$ factor for the large child, so we can cancel them. 
	
	At this point we want to sum over the orders of small circle children (We assume that they exist in our tree). As long as their number and their total order remains the same, the only thing that depends on their order is what's left of \eqref{eq5444}, namely:
	
	\begin{equation}
		(j+k+1)^{90(j+k+1)}e^{70D(j+k+1)}.
	\end{equation}
	 
	 The logarithm of this is a strictly convex function of $(j+k)$, and hence, by Karamata's inequality, the maximum over orders with a fixed sum will be achieved, when all orders, except one, will be equal to $2$. The maximum would be equal to
	 
	 \begin{equation}
	 	3^{180(s_{s, 0} - 1)}e^{70D(\omega_{s, \varphi} + s_{s, 0})} (\omega_{s, \varphi} - 2s_{s, 0} + 3)^{90(\omega_{s, \varphi} - 2s_{s, 0} + 3)},
	 \end{equation}
 
 	where $s_{s, 0}$ is the number of small circle children. Once again by increasing $D$, we can forget about a power of $3$ at the front. Similarly, we get rid of $2s_q+2s_c$ scaling of exponent with $F$. Now we need to multiply this by a total number of ways we can assign orders. Trivially, this number doesn't exceed $\binom{\omega_{s, \varphi}}{s_{s, 0} - 1} \le 2^{\omega_{s, \varphi}}$. We also get rid of this power of $2$. 
 	
 	Now we can sum over the orders of polygons. First, we add a factor $2$ because we may differentiate either one. Then, in $\circledast$ there are at most $n^{1/45000}$ ways of distributing orders among the two of them. The maximal contribution among those will correspond to $0$ order pentagon and $\omega_0$ order square, due to \eqref{eq545}. So, we will have 
 	
 	\begin{equation}
 		2n^{1/45000}\omega_0^{90\omega_0}e^{70D\omega_0}.
 		\label{eq552}
 	\end{equation}
 
 	In $\boxed{*}$, we have to be more careful. Once again, we have at most $n^{1/45000}$ ways of distributing, and the largest will be when the order of pentagon is zero (the only thing changing is $\omega_q!$ from the recurrent assumption). So, we bound these factorials by
 	
 	\begin{equation}
 		2n^{1/45000}\omega_0!.
 		\label{eq553}
 	\end{equation}
 
 	However, sometimes we shouldn't consider $\omega_c = 0$, particularly when the pentagon is the only thing that makes our tree different from the principal one. Particularly, when $\omega_{s, \varphi} = 2s_0$, $y_0 = 0$ and $\mu = 0$ we cannot have $\omega_c = 0$, so we should divide \eqref{eq553} by $\omega_0 > n/2$.
 	
 	Thus, in $\boxed{*}$ we have the following total bound:
 	
 	\begin{equation}
 		2n^{1/10000} \frac{\omega_0!}{(2n)!}\binom{s_0+C_0}{s_0}\binom{C_0}{y_0}e^{70D(\omega_{s, \varphi} + s_{0})} (\omega_{s, \varphi} - 2s_{0} + 3)^{90(\omega_{s, \varphi} - 2s_{0} + 3)}n^{\mu/40000}(2D)^{2\omega_b},
 		\label{eq554}
 	\end{equation}
 
 	possibly divided by $n/2$, while in $\circledast$ we have:
 	
 	\begin{equation}
 		2n^{1/10000} \omega_0^{90\omega_0} \frac{(\omega_{l, \varphi} + 1)!}{(2n)!}\binom{s_0+C_0}{s_0}\binom{C_0}{y_0}e^{70D(\omega_0 + \omega_{s, \varphi} + s_{0} - 1)} (\omega_{s, \varphi} - 2s_{0} + 5)^{90(\omega_{s, \varphi} - 2s_{0} + 5)}n^{\mu/40000}(2D)^{2\omega_b},
 		\label{eq555}
 	\end{equation}
 
 	where $\omega_{l, \varphi}$ is the order of the large child.
 	
 	Now, we deal with $y_0$, $\mu$ and $\omega_b$. In $\boxed{*}$ increasing $y_0$ by $1$ and decreasing $\mu$ by $1$ multiplies the contribution by $\frac{C_0 - y_0}{(y_0+1)n^{1/40000}} > 2$, since the numerator is of order $n$. Hence, we may assume $\mu = 0$ and just add a factor of $2$ to the front. Similarly, increasing $y_0$ by $1$ up until $y_0 = \omega_b$ adds another factor of $2$. In $\circledast$, $\binom{C_0}{y_0}n^{\mu/40000}$ is always bounded by $n^{(\mu + y_0)/40000} \le n^{\omega_b/40000}$. Hence, we remove $\mu$ and $y_0$ and change $2D$ to $4D$ to account for the choices of them. 
 	
 	We also bound $\omega_{s, \varphi - 2s_0 + 3}$, $\omega_{s, \varphi - 2s_0 + 5}$ and $\omega_0$ in the base of the powers in \eqref{eq554} and \eqref{eq555} by $n^{1/45000}$. In \eqref{eq555} we also bound $\binom{s_0+C_0}{s_0}$ by $n^{s_0/500}$. Then, they turn into the following:
 	
 	\begin{equation}
 		8n^{1/10000} \frac{\omega_0!}{(2n)!}\binom{s_0+C_0}{s_0}\binom{C_0}{\omega_b}e^{70D(\omega_{s, \varphi} + s_{0})} n^{(\omega_{s, \varphi} - 2s_{0} + 3)/500}(2D)^{2\omega_b},
 		\label{eq556}
 	\end{equation}
	
	and
	
	\begin{equation}
		2n^{1/10000}  \frac{(\omega_{l, \varphi} + 1)!}{(2n)!}e^{70D(\omega_0 + \omega_{s, \varphi} + s_{0} - 1)} n^{(\omega_0 + \omega_{s, \varphi} + 2\omega_b - s_{0} + 5)/500}(4D)^{2\omega_b}.
		\label{eq557}
	\end{equation}

	We note that these bounds also hold in the case of $s_{s, 0} = 0$.
	
	Now we focus on $\circledast$ and \eqref{eq557}. First, we bound $(4D)^{2\omega_b}$ with $e^{70D(2\omega_b)}$ and move it into the exponential. Then, we bound $\omega_{s, \varphi} + s_0 - 1 \le 3(\omega_{s, \varphi} - s_0 + 1)$, since the order of every circle is at least $2$. This factor of $3$ changes the coefficient to $210$. Finally, by changing $1/500$ to $1/400$ in the power of $n$ we can get rid of the exponential for large $n$. In the result, we get:
	
	\begin{equation}
		2n^{1/10000}  \frac{(\omega_{l, \varphi} + 1)!}{(2n)!} n^{(\omega_0 + \omega_{s, \varphi} + 2\omega_b - s_{0} + 5)/400}.
	\end{equation}

	Finally, we can bound this. Note that
	
	\begin{equation}
		\omega_0 + \omega_{s, \varphi} + 2\omega_b - s_0 = 2n - \omega_{l, \varphi},
	\end{equation}

	hence we in fact have
	
	\begin{equation}
		2n^{1/10000}  \frac{(\omega_{l, \varphi} + 1)!}{(2n)!} n^{(2n - \omega_{l, \varphi} + 5)/400}.
	\end{equation}

	When, $\omega_{l, \varphi}$ is of order $n$, this increases, when $\omega_{l, \varphi}$ is increased. However, we have the bound $\omega_{l, \varphi} \le 2n - 2$. Hence, we substitute $2n-2$ as $\omega_{l, \varphi}$ and bound the number of ways to choose parameters ($\omega_{l, \varphi}$, $s_0$, etc) by $n^{1/400}$. Then, the total contribution is bounded by 
	
	\begin{equation}
		2n^{1/10000}\frac{1}{2n}n^{7/400}n^{1/400} \le n^{-9/10}.
	\end{equation}
	
	Now we can solely focus on $\boxed{*}$ and \eqref{eq556}. We have the restriction
	
	\begin{equation}
		\omega_0 + \omega_{s, \varphi} + 2\omega_b - s_0 = 2n.
	\end{equation}
	
	We substitute $C_0 = \omega_0 - s_0 - 1$. Then, we note that increasing $\omega_0$ by $2$ and decreasing $\omega_b$ by $1$ always increases contribution. Hence we can assume $\omega_b = 0$ and add $n^{1/10000}$ factor, provided decreasing $\omega_b$ like this will not make a tree a principal one. In that case, the maximum will be achieved for $\omega_b = 1$, but we can still substitute $\omega_b = 0$ and divide by $n^{-9999/10000}$ factor in such case. Hence, we get:
	
	\begin{equation}
		8n^{1/5000} \frac{\omega_0!}{(2n)!}\binom{\omega_0-1}{s_0}e^{70D(\omega_{s, \varphi} + s_{0})} n^{(\omega_{s, \varphi} - 2s_{0} + 3)/500}.
		\label{eq563}
	\end{equation}

	Now increasing $s_0$ and $\omega_{s, \varphi}$ by $1$ up until $2s_0 = \omega_{s, \varphi}$ will increase the contribution. So, we set $s_0$ to be that and add another $n^{1/5000}$ factor. Then, the last power of $n$ can also be easily bounded. This makes the tree principal, so we cann add $n^{-99/100}$ factor. Moreover, we remember other things we did that made trees principal, so we add a factor $3$ in front. We have:
	
	\begin{equation}
		24n^{-49/50} \frac{\omega_0!}{(2n)!}\binom{\omega_0-1}{2n-\omega_0}e^{280D(2n - \omega_0)}.
	\end{equation}

	By binomial definition, this itself is bounded by
	
	\begin{equation}
		24n^{-49/50}\frac{e^{280D(2n - \omega_0)}}{(2n - \omega_0)!}.
		\label{eq565}
	\end{equation}

	If we pick $n$, such that
	
	\begin{equation}
		24\sum_{j = 0}^{+\infty} \frac{e^{280Dj}}{j!} = 24e^{e^{280D}} < n^{1/50},
		\label{eq566}
	\end{equation}

	we can sum \eqref{eq565} over $\omega_0$ and bound the sum with
	
	\begin{equation}
		n^{-24/25}.
	\end{equation}

	\subsection{Principal trees}
	
	Having bounded small children trees and large child trees, excluding principal trees, we are now left to bound principal trees themselves. Our goal would be to show that the contribution of these trees is given by some linear operator like \eqref{newrec} plus error terms. Hence, we will not bound contributions in absolute terms.
	
	Right now we are interested in the large child trees, which are principal. We will discuss small children principal trees later.
	
	First, we need to estimate coefficients of linear dependence between $q_{2n}$ and previous $q_{2n - 2\sigma}$, previously denoted $R_{\sigma, n}$ in  \eqref{eq46}. Specifically, we need to bound $R_{\sigma, n} - \hat{R}_{\sigma, n}$, with $\hat{R}$ being introduced in \eqref{eq413}.
	
	\begin{lemma}
		We can increase $D$, such that for large child principal trees, we have:
		\begin{equation}
			|R_{\sigma, n} - \hat{R}_{\sigma, n}| \le \frac{(2n)!}{(2n - 2\sigma)!} e^{2D\sigma} \frac{n^{-999/1000}}{(2\sigma)!}.
		\end{equation}
	\label{lema55}
	\end{lemma}

	\begin{proof}
		Since we are discussing large child trees (and hence $\sigma < n^{1/40000}$), there are at most $n^{1/20000}$ terms in \eqref{eq46}, so we may consider them separately. 
		
		In every term, we bound the binomial with \eqref{eq49}. We separate the main term and the error term from there. The error is bounded $Dn^{-19999/20000}$ (we may assume the uniform constant is smaller than $D$). Then, in both error and main term we can reduce the sum. In the main term, it gives us \eqref{twosigma}, while in the error term we get:
		
		\begin{equation}
			\frac{(2n - 2\sigma - 1)!(n-1)!^2(n - \sigma)}{(2n-1)!(n - 2\sigma)!(n - 2\sigma - 1)!}\cos(b_0/2)^{- 4\sigma}(|\varphi_{2, 0} \cos b_0| + |\varphi_{1, 1}| + |\varphi_{0, 2}\cos b_0|)^{2\sigma}\frac{Dn^{-19999/20000}}{(2\sigma)!}.
			\label{eq569}
		\end{equation}
		Here, we bound the first fraction with
		\begin{equation}
			\frac{(2n - 2\sigma - 1)!(n - 1)!^2(n - \sigma)}{(2n-1)!(n - 2\sigma)!(n - 2\sigma - 1)!} \le \frac{(2n)!}{(2n - 2\sigma)!}.
		\end{equation}

	In the main term we get \eqref{twosigma}, and we need to justify the asymptotic change we did after \eqref{twosigma}. Hence, we estimate:
	
	\begin{equation}
		\frac{(2n-2\sigma)!2^{4\sigma}(2n - 2\sigma-1)!(n-1)!^2(n - \sigma)}{(2n)!(2n-1)!(n-2\sigma)!(n - 2\sigma-1)!} = \left(\frac{(2n-2\sigma)!2^{2\sigma}}{(2n)!}\times\frac{n!}{(n-2\sigma)!}\right)^2 (1 + O(n^{-999/1000})).
	\end{equation}
	The squared term is the following:
	
	\begin{equation}
		1 \ge \prod_{j = 0}^{2\sigma - 1}\frac{2(n - j)}{2n - j} = \prod_{j = 0}^{2\sigma - 1} \frac{1-j/n}{1-j/(2n)} \ge \left(1 - 2\sigma/n\right)^{2\sigma} \ge 1 - \frac{4\sigma^2}{n} \ge 1 - n^{-999/1000}. 
	\end{equation}

	Note that we have almost finished the proof of the lemma, since in both terms we got the $n^{-999/1000}$ term. We also note that we can estimate the absolute values in \eqref{eq569} by $D$ uniformly (by increasing $D$), because of Lemma \ref{lema54}.  

	\end{proof}

	We can immediately estimate the contribution of these error terms to Lemma \ref{lema51}. Here, contribution means $R_{\sigma_n} - \hat{R}_{\sigma, n}$ times $q_{2n - 2\sigma}$, summed $\sigma$ in the large child range.
	
	\begin{proposition}
		For any large enough $n$, the total contribution of error terms, associated to $R_{\sigma, n} - \hat{R}_{\sigma, n}$ for large child principal trees, divided by the RHS of \eqref{qbound} is bounded by
		
		\begin{equation}
			n^{-9/10}.	
		\end{equation}
	\end{proposition}
	\begin{proof}
		We use an inductive bound on $q_{2n - 2\sigma}$ and cancel $W$ due to its monotonicity. We use the previous lemma and cancel out $(2n)!$ and $(2n - 2\sigma)!$. Then, we sum over $\sigma$ and we use the trick from \eqref{eq566} to get rid of the exponential growth in $\sigma$. 
	\end{proof}

	Finally, we have to mention small children principal trees. We have already estimated all small trees. But we want to use $\hat{R}_{\sigma, n}$ for every $\sigma$, not just small ones. So, we add a sum of $\hat{R}_{\sigma, n} q_{2n - 2\sigma}$ to $q_{2n}$ and subtract it. Then, we estimate the subtracted terms and claim that they are small:
	
	\begin{proposition}
		The sum of $\hat{R}_{\sigma, n} q_{2n - 2\sigma}$ over $\sigma$ in small children tree range, divided by the RHS of \eqref{qbound}, is bounded by $n^{-9/10}$.
	\end{proposition}

	\begin{proof}
		We once again use the inductive bound for $q_{2n - 2\sigma}$, cancel out $(2n)!$ and $(2n - 2\sigma)!$ and $W$. Then, the result will be the following:
		
		\begin{equation}
			\sum_{\sigma} \frac{C_5^{\sigma}}{(2\sigma)!},
		\end{equation}
		
		where $C_5$ is something, dependent only on $C$, $D$ and $\delta$. We bound $(2\sigma)!$ with $\sigma!^2$. One of the factorials will bound $C_5$, like in \eqref{eq566}, and another one will be bounded by $n$, since $\sigma > n^{1/60000}$.  
	\end{proof}

	Finally, we can finish the proof for the boxes. We now know that
	
	\begin{equation}
		|q_{2n} - \sum_{\sigma = 1}^{n} \hat{R}_{\sigma, n}q_{2n - 2\sigma}| < n^{-4/5} e^{W(2n-2, C_1)} (2n)! e^{2n(C + F(\lambda) + \delta)}.
		\label{eq575}
	\end{equation}

	Hence, we just need to bound 
	
	\begin{equation}
		\frac{\sum_{\sigma = 1}^{n} \hat{R}_{\sigma, n}q_{2n - 2\sigma}}{ e^{W(2n-2, C_1)} (2n)! e^{2n(C + F(\lambda) + \delta)}}.
	\end{equation}

	Recalling the definitions of $\hat{R}_{\sigma, n}$, $\breve{q}_{2n}$ and $F(\lambda)$ from \eqref{eq413}, \eqref{coorch} and \eqref{eq53}, this reduces to:

	\begin{equation}
		\frac{q_3^{2n}}{e^{2nC}} \times e^{-2n\delta} (\pi/2)^{2n} \times \frac{\sum_{\sigma = 1}^{n} - \frac{1}{(2\sigma)!} \breve{q}_{2n - 2\sigma}}{ e^{W(2n-2, C_1)}}.
	\end{equation}

	 We also need to write the inductive assumption and \eqref{eq575} in terms of $\breve{q}$:
	
	\begin{equation}
		|\breve{q}_{2n}| \le e^{W(2n-2, C_1)} \times \frac{e^{2nC}}{q_3^{2n}} \times e^{2n\delta}(\pi/2)^{-2n},
		\label{eq578}
	\end{equation}

	and 
	
	\begin{equation}
		|\breve{q}_{2n} - \sum \frac{-1}{(2\sigma)!}\breve{q}_{2n  -2\sigma}| <  n^{-4/5}e^{W(2n-2, C_1)} \times \frac{e^{2nC}}{q_3^{2n}} \times e^{2n\delta}(\pi/2)^{-2n}.
	\end{equation}

	Next, we define a linear evolution operator, that will help us to prove the bounds. If we define a usual operator for a recurrent relation \eqref{newrec}, it will not be compact, so we have to renormalize $\breve{q}$ once again:
	
	\begin{equation}
		\mathbf{q}_{n, j} = \breve{q}_{2n - 2j} / j!,
	\end{equation} 

	where $\breve{q}_{2n}$ is $0$ if $n < 0$. We define an operator:
	
	\begin{equation}
		T: \ell^2 \rightarrow \ell^2, \; \; \; T\mathbf{x} = \mathbf{y} \; \; \text{if} \; \;  \mathbf{y}_0 = -\sum_j \frac{j!}{(2j + 2)!}\mathbf{x}_j, \; \; \mathbf{y}_k = \frac{\mathbf{x}_{k - 1}}{k}, \; k > 0.
	\end{equation}
	
	\begin{lemma}
		T is a bounded compact operator and its spectrum is given by
		\begin{equation}
			\left\{0\right\}\cup \left\{\frac{-4}{(2m+1)^2\pi^2}\; | \; m \in \mathbb{N}\right\}.
		\end{equation}
	Moreover, the dimension of every eigenspace for non-zero eigenvalue is $1$. 
	\label{lema56}
	\end{lemma}
	\begin{proof}
		$T$ is bounded because it is a sum of a shift and scale operator and of an operator, that places a product of $x$ with another $\ell^2$ vector tot he first position. Moreover, the Hilbert-Schmidt norm of $T$ is given by
		\begin{equation}
			\sum_{k = 1}^{+\infty} \frac{1}{k^2} + \sum_{j = 0}^{+\infty}\frac{j!^2}{(2j+2)!^2} < +\infty,
			\end{equation}
		so $T$ is compact. 
		
		If $\lambda$ is a non-zero eigenvalue of $T$ with eigenvector $\mathbf{x}$, then $\mathbf{x}_{j} = \frac{1}{\lambda^j j!}\mathbf{x}_0$, and if we write condition for $\mathbf{x}_0$, we get:
		\begin{equation}
			-\sum_{j = 0}^{+\infty} \frac{1}{\lambda^j(2j+2)!} = \lambda \; \; \Rightarrow \; \; \cosh(1/\sqrt{\lambda}) = 0.
		\end{equation}
	
		Since $\cosh(1/\sqrt{\lambda})$ has only simple roots, the eigenspace dimension is $1$.
	
	\end{proof}

	Now we can apply $T$ to our problem. We have:
	
	\begin{equation}
		|(T\mathbf{q}_{n-1})_0 - \mathbf{q}_{n, 0}| <  n^{-4/5}e^{W(2n-2, C_1)} \times \frac{e^{2nC}}{q_3^{2n}} \times e^{2n\delta}(\pi/2)^{-2n}, \; \; (T\mathbf{q}_{n - 1})_j = \mathbf{q}_{n, j},
		\label{eq585}
	\end{equation}
	
	where $j$ and $n$ are positive. We can bound $T\mathbf{q}_{n-1}$ using the norm of $T$. However, to get a $\pi/2$ factor in the inductive formula, we need to use the spectral radius of $T$. Hence, we take a number $N$, such that
	
	\begin{equation}
		\left\| T^N \right\| \le e^{-\pi^2/4}\left(\frac{4e^{\delta}}{\pi^2} \right)^N.
	\end{equation}

	Then, since $N$ is only dependent on $\delta$, we can use \eqref{eq585} for $n - N$. Namely, for $j = n - N + 1$ up to $n$, we have:
	
	\begin{equation}
		\left\| T\mathbf{q}_{j - 1} - \mathbf{q}_j\right\| < 2n^{-4/5} e^{W(2n-2, C_1)} \times \frac{e^{2nC}}{q_3^{2n}} \times e^{2j\delta}(\pi/2)^{-2j}.
		\label{eq587}
	\end{equation} 

	Note we have also bounded some terms using $n$ instead of $j$. This allows us to write a similar statement for $T^N$:
	
	\begin{align}
		\left\| T^N\mathbf{q}_{n - N} - \mathbf{q}_n\right\| \le \sum_{j = n - N + 1}^{n} 	\left\| T^{n-j} (T\mathbf{q}_{j - 1} - \mathbf{q}_j)\right\| \le \sum_{j = n - N + 1}^{n}  \left\|T\right\| ^{n-j}	\left\| (T\mathbf{q}_{j - 1} - \mathbf{q}_j)\right\|.
	\end{align}

	We use \eqref{eq587} to bound every term. We get:
	
	\begin{equation}
		|(T^N\mathbf{q}_{n-N})_0 - \mathbf{q}_{n, 0}| \le \left\| T^N\mathbf{q}_{n - N} - \mathbf{q}_n\right\| \le  n^{-3/4}e^{W(2n-2, C_1)} \times \frac{e^{2nC}}{q_3^{2n}} \times e^{2n\delta}(\pi/2)^{-2n}.
		\label{eq589}
	\end{equation}	

	We removed the sum over $j$ and exponentials of $n - j$ by paying with a power of $n$, since $n$ can be chosen to be large enough and these terms depended only on $D$, $\delta$ and the operator itself. 
	
	Finally, we bound $T^N\mathbf{q}_{n - N}$:
	
	\begin{align}
		|(T^N\mathbf{q}_{n-N})_0| \le \left\| T^N\mathbf{q}_{n - N}\right\| \le e^{-\pi^2/4}\left(\frac{4e^{\delta}}{\pi^2} \right)^N \left\|\mathbf{q}_{n - N}\right\| \le  e^{-\pi^2/4}\left(\frac{4e^{\delta}}{\pi^2} \right)^N \left\|\mathbf{q}_{n - N}\right\|_1 = \\ = e^{-\pi^2/4}\left(\frac{4e^{\delta}}{\pi^2} \right)^N \sum_{j = 0}^{+\infty} \frac{|\breve{q}_{2n-2N-2j}|}{j!}.
	\end{align}
	
	We use \eqref{eq578}, as well as monotonicity bounds to get:
	
	\begin{align}
		|(T^N\mathbf{q}_{n-N})_0| \le e^{-\pi^2/4}\left(\frac{4e^{\delta}}{\pi^2} \right)^N e^{W(2n-2, C_1)}\frac{e^{2nC}}{q_3^{2n}} \frac{e^{2(n - N)\delta}}{(\pi/2)^{2(n - N)}} \sum_{j = 0}^{+\infty} \frac{(\pi/2)^{2j}}{j!} \le \\ \le 
		e^{W(2n-2, C_1)} \times \frac{e^{2nC}}{q_3^{2n}} \times e^{2n\delta} (\pi/2)^{-2n} \times e^{-N\delta}.
		\label{eq592}
	\end{align}
	
	Since we can guarantee that $n^{-3/4} + e^{-N\delta} \le 1$, we can prove \eqref{eq578} by adding up \eqref{eq589} and \eqref{eq592}, thus proving the inductive bound for boxes. 
	
	If $q_3 = 0$, we divide by $0$ so we cannot really prove, but we immediately get the result because nothing comes from $\hat{R}$.

	\subsection{Circle vertices}
	
	We have proven the inductive statement of Lemma \ref{lema51} for the boxes. Now, we need to do the same for the circles. We note that the bounds for circles shouldn't be as strict, as those for the boxes, since the principal trees involve recurrence in only boxes.
	
	Similarly to boxes, we have to consider two types of trees of depth one: small children and large child trees. The proof is similar to the box case, but there are some differences. For example, L is different and labels in large child trees can be large. We start by considering small children trees.  To avoid confusion, we will prove the statement for $\varphi_{j_0, k_0}$. 
	
	Once again, exponentials cancel and give only $2s_0$ power; B1, B3, B6, C5 turn into a constant, as well as $\cos \alpha$ from L. The cosine difference factors from L and $G$ of the root vertex cancel out. We remove all the $\cos(b_0/2)$ from C4, $G_{j, k}$, $G_{j_0, k_0}$ and L in the same way as for the boxes. B2 is bounded by $j_0+k_0+1$ from $G_{j_0, k_0}$. We bound B4 by $|l_0|/m!$ and B5 by $l^m/m!$, where $l_0 = j_0 - k_0$.  The denominator in $G_{j, k}$ is bounded by $D(|l|+1)$. 
	
	We sum over C3 giving us $\binom{C_0}{\frac{C_0+\delta}{2}}$, and C1, C2, C6 change into \eqref{eq58}. We also do the same trick for semicircles:
	
	\begin{equation}
		(\sum_{i}|l_i| + |l_0| + 1)^{\mu}(2D)^{2\omega_b}.
		\label{eq594}
	\end{equation} 

	Note that since B4 is different from the box case, we get $|l_0|$ inside of the power. We bound the sum of labels by $4(j_0+k_0)+1$.
	
	Now we need to sum over labels of the circle children. Here, we differ from the box case, since if we use \eqref{eq513} and $j_0$ or $k_0$ would be small (and $k_0$ or $j_0$ large respectively), we wouldn't be able to cancel the resulting power of $2$ using the binomial in $G_{j_0, k_0}$. Hence, we first bound $D(|l|+1)$ by $2D(j+k)$ for every vertex. Then, we can sum over all possible labels of circle children, while keeping the sum of the labels (and hence $\delta$ fixed). This only changes $\binom{j+k}{j}$ in $G_{j, k}$. If the total sum of $j$-s ($k$-s) over all circle children is $j_{\infty}$ ($k_{\infty}$), we get:
	
	\begin{equation}
		\sum_{j_1+\ldots+j_{s_0} = j_{\infty}}\prod_{i = 1}^{s_0} \binom{\omega_i}{j_i} = \binom{\sum_i \omega_i}{j_\infty} = \binom{\omega_\varphi + s_0}{(\omega_{\varphi} + s_0 + l_0 - \delta)/2 }.
	\end{equation}

	Here, $\omega_i$ is the order of $i$-th circle, that stays fixed. $\omega_\varphi$ comes from the small child box case and we have also recalled the definition of $\delta$. We can now sum over $\delta$:
	
	\begin{align}
		\sum_{\delta} \binom{C_0}{(C_0+\delta)/2} \binom{\omega_\varphi + s_0}{(\omega_{\varphi} + s_0 + l_0 - \delta)/2 } = \binom{\omega_{\varphi} + s_0 + C_0}{(\omega_\varphi + s_0 + C_0 + l_0)/2} = \\ = \binom{\omega_\varphi + \omega_0 - 1}{(\omega_\varphi + \omega_0 - 1 + l_0)/2} = \binom{j_0 + k_0 - 2\omega_b}{j_0 - \omega_b} \le \binom{j_0+k_0}{j_0}. 
	\end{align}

	We cancel it with the same binomial from $G_{j_0, k_0}$.
	
	The rest of the proof for this case is identical to the low children box case.
	
	The final case is the large child case for the circle root. We approach it in the same way:
	
	We have two cases: $\circledast$ and $\boxed{*}$. For the large child, we use the inductive statement, while for the small ones we use \eqref{eq5444} and \eqref{eq545}. B1, B3, B6 and C5 are bounded, as well as $\cos \alpha$ from L. $(j_0+k_0+1)$ from $G_{j_0, k_0}$ is bounded $\boxed{*}$ by B2 or $\circledast$ by $(j+k+1)$ from $G_{j,k}$, giving us $(j_0+k_0)^{1/50000}$ from B2. The exponents leave $2s_0$ scaling. Cosine difference terms from $G_{j_0, k_0}$ and L cancel out, as well as $\cos(b_0/2)$ factors. B4 and B5 ar bounded the same, as in small children case. In $\circledast$, we bound the denominator of $G_{j, k}$ in the same way. We sum over $x$, giving us $\binom{C_0}{(C_0+\delta)/2}$ in $
	\circledast$ we further bound it with $2^{C_0}$.
	
	We also get \eqref{eq58} and \eqref{eq594} in the same way. Now we bound the sum of labels in \eqref{eq594} by $4(j_0+k_0)$. Then, we sum over labels of circled vertices. For small circles, the contribution doesn't depend on the label and hence we get $j+k+1$, that cancels with $j+k+1$ in \eqref{eq5444} ($-1$ in exponential). For the big circle in $\circledast$ there are at most $n^{1/40000}$ possible labels (its label should be close to the label of the root). For all of them, the binomial in $G_{j, k}$ is bounded by the binomial in $G_{j_0, k_0}$ ($j\le j_0$ and $k \le k_0$), so we cancel them. We also bound $D(|j-k|+1)$ by $D(j_0+k_0)$. We then sum over $\delta$, giving us $n^{1/40000}$ factor.
	
	In $\boxed{*}$ we should also sum over $\delta$: $\binom{C_0}{(C_0+\delta)/2}$ should always be bounded by the binomial in $G_{j_0, k_0}$, so we cancel them, and we also get at most $n^{1/40000}$ from the choice of $\delta$.
	
	We cancel $W$ of the root with $W$ of the large child. Identically to the box case, we sum over small circle orders, giving us
	
	\begin{equation}
		e^{70D(\omega_{s, \varphi} + s_{s, 0})}(\omega_{s, \varphi} - 2s_{s, 0} +3)^{90(\omega_{s, \varphi} - 2s_{s, 0}+3)},
	\end{equation}

	and we also remove exponential factors. We also get \eqref{eq552} and \eqref{eq553}, summing over polygonal orders. In $\circledast$ this allows us to get rid of $2^{C_0}$ by increasing $D$. 
	
	Hence, in $\boxed{*}$ we get:
	
	\begin{equation}
		2(j_0+k_0)^{1/10000} \frac{\omega_0!}{(j_0+k_0+1)!}\binom{s_0+C_0}{s_0}\binom{C_0}{y_0}e^{70D(\omega_{s, \varphi} + s_{0})} (\omega_{s, \varphi} - 2s_{0} + 3)^{90(\omega_{s, \varphi} - 2s_{0} + 3)}(4(j_0+k_0))^{\mu}(2D)^{2\omega_b},
		\label{eq599}
	\end{equation}
	
	while in $\circledast$, we get:
	
	\begin{align}
		2(j_0+k_0)^{1/10000}  D(j_0+k_0)\omega_0^{90\omega_0} \frac{(\omega_{l, \varphi} + 1)!}{(j_0+k_0+1)!}\binom{s_0+C_0}{s_0}\binom{C_0}{y_0}\times \\ \times e^{70D(\omega_0 + \omega_{s, \varphi} + s_{0} - 1)} (\omega_{s, \varphi} - 2s_{0} + 5)^{90(\omega_{s, \varphi} - 2s_{0} + 5)}(4(j_0+k_0))^{\mu}(2D)^{2\omega_b}.
		\label{eq5100}
	\end{align}

	In both cases increasing $\mu$ by $1$ up until $y_0+\mu = \omega_b$ and then decreasing $y_0$ by $1$ and increasing $\mu$ by $1$ increases the contribution by a factor of at least $4$. Hence, we may assume that $y_0 = 0$ and $\mu = \omega_b$. We also note that in $\circledast$ when $\omega_b - \mu < 3$ both of those changes will increase the contribution by an order of $(j_0+k_0)^{39999/40000}$. We also perform the same bounds on $\omega$ as we did before \eqref{eq556}:
	
		\begin{equation}
		8(j_0+k_0)^{1/10000} \frac{\omega_0!}{(j_0+k_0+1)!}\binom{s_0+C_0}{s_0}e^{70D(\omega_{s, \varphi} + s_{0})} (j_0+k_0)^{(\omega_{s, \varphi} - 2s_{0} + 3)/500}(4(j_0+k_0))^{\omega_b}(2D)^{2\omega_b},
		\label{eq102}
	\end{equation}
	
	and
	
	\begin{equation}
		8(j_0+k_0)^{1/10000} D(j_0+k_0) \frac{(\omega_{l, \varphi} + 1)!}{(j_0+k_0+1)!}e^{70D(\omega_0 + \omega_{s, \varphi} + s_{0} - 1)} (j_0+k_0)^{(\omega_0 + \omega_{s, \varphi} - s_{0} + 5)/500}(4(j_0+k_0))^{\omega_b}(2D)^{2\omega_b}.
		\label{eq5103}
	\end{equation}

	We deal with $\circledast$ and \eqref{eq5103}. We bound $\omega_{s, \varphi} + s_0 - 1$ by $3(\omega_{s, \varphi} - s_0 + 1)$ and then introduce $\omega_1 = \omega_0 + \omega_{s, \varphi} - s_0$ with
	
	\begin{equation}
		\omega_1 + 2\omega_b = j_0+k_0+1 - \omega_{l, \varphi}.
	\end{equation}
	
	Then, $\eqref{eq5103}$ turns into:
	
	\begin{equation}
		8(j_0+k_0)^{1/90} D(j_0+k_0) \frac{(\omega_{l, \varphi} + 1)!}{(j_0+k_0+1)!}e^{210D(\omega_1 + \omega_b)} (j_0+k_0)^{\omega_1/500}(4(j_0+k_0))^{\omega_b}.
		\label{eq5104}
	\end{equation}

	Then, we have that decreasing $\omega_1$ by $1$ and increasing $\omega_{l, \varphi}$ by $1$ increases the contribution by a factor of $n^{99/100}$. Decreasing $\omega_b$ by $1$ and increasing $\omega_{l,  \varphi}$ by $2$ also does that. Hence, since $\omega_1 \ge 1$, the maximum will happen when $\omega_1 = 1$, $\omega_b = 0$ and $\omega_{l, \varphi} = j_0+k_0$. It will be bounded by
	\begin{equation}
		(j_0+k_0)^{11/10}.
	\end{equation}
	
	Hence, if we take terms where we did an action that increased the size of the deformation by $n^{99/100}$ at least $2$ times, the total contribution of that terms would be bounded by $n^{-4/5}$. We have to do an action at least once, otherwise we have a tree that expresses $\varphi_{j_0, k_0}$ in terms of itself. Hence, either $\omega_b = 1$ and $\omega_1 = 1$ or $\omega_{b} = 0$ and $\omega_1 = 2$. This results in:
	
	\begin{itemize}
		\item $\omega_q = 2$ and $\omega_c = 0$ or $\omega_c = 2$ and $\omega_q = 0$, $s_0 = 1$, $\mu = 1$, $y_0 = 0$, $\omega_b = 1$. 
		\item $\omega_q+\omega_c = 3$, $s_0 = 1$, $\mu = y_0 = \omega_b = 0$.
	\end{itemize}

	We will deal with those after finishing $\boxed{*}$. In that case, we also see that we can decrease $\omega_b$ by $1$ and increase $\omega_0$ by $2$ and this will increase the contribution by an order of $n$. After setting $\omega_b$ to $0$, we see that the rest of the proof would be identical to $\boxed{*}$ for the box case.
	
	Hence, we see that only the following trees of depth one need to be considered: there should be one box (or more generally, square) child and the rest of the children would consist of $k$ circle children of order $2$ ($k \ge 0$). 
	
	For those trees we can compute their contribution explicitly. However, we don't need a strict bound, since circles are not that important to the main recurrence. We can estimate these trees in the following way. First, we estimate the tree that has just one child, that is a square. $G$ was chosen in a way, so that the contribution of this tree was bounded by $1/C_7$. For other trees, we can respectively bound them by $e^{C_{13}\sigma}/(C_7\sigma!)$, where $\sigma$ is the number of circle children for some $C_{13}$, dependent on $C$ and $D$. Since there are at most $3^{\sigma}$ of such trees for every $\sigma$, we can bound the total contribution by $e^{3e^{C_{13}}}/C_7 \le 1/\sqrt{C_7}$. 
	
	Next, we go to the two classes of trees we left out, while discussing $\circledast$. For those classes, we can directly compute their contributions, since we can draw all the trees. We will call the large child vertex $A$ and the root vertex $B$. If $A$'s Diophantine denominator wasn't small (of order $1/n$), then we can use a better bound for $G_{j, k}$ for $A$, and so both classes will go to the error term. If we consider the second class, then $B$ cannot have a small denominator, since $A$ and $B$'s labels will differ by $1$. Then, since this cannot be “chained” ($B$ cannot serve in principal trees, as well as in these $2$ classes), we can deal with it by increasing $C_7$.

	\begin{figure}
		\centering
		\begin{tikzpicture}[scale=0.6]
			\node [circ,label=center:$l$, label=below:$\varphi_{j, k}$] (0) at (-4, 0) {};
			\node [semi,label=center:$2$] (1) at (-2, 2) {};
			\node [pent,label=center:$2$] (2) at (-2, 0) {};
			\node [circ,label=center:$l$, label=below:$\varphi_{j-1, k-1}$] (3) at (-2, -2) {};
			\draw (0) -- (1);
			\draw (0) -- (2);
			\draw (0) -- (3);
			\node [circ,label=center:$l$, label=below:$\varphi_{j, k}$] (10) at (-8, 0) {};
			\node [semi,label=center:$2$] (11) at (-6, 2) {};
			\node [squa,label=center:$2$] (12) at (-6, 0) {};
			\node [circ,label=center:$l$, label=below:$\varphi_{j-1, k-1}$] (13) at (-6, -2) {};
			\draw (10) -- (11);
			\draw (10) -- (12);
			\draw (10) -- (13);
			\node [circ,label=center:$l$, label=below:$\varphi_{j, k}$] (20) at (4, 0) {};
			\node [semi,label=center:$2$] (21) at (6, -2) {};
			\node [pent,label=center:$2$] (22) at (6, 2) {};
			\node [circ,label=center:$l$, label=below: $\varphi_{j-1, k-1}$] (23) at (6, 0) {};
			\draw (20) to[out=-60,in=130] (21);
			\draw (20) -- (22);
			\draw (20) -- (23);
			\node [circ,label=center:$l$, label=below:$\varphi_{j, k}$] (30) at (0, 0) {};
			\node [semi,label=center:$2$] (31) at (2, -2) {};
			\node [squa,label=center:$2$] (32) at (2, 2) {};
			\node [circ,label=center:$l$, label=below:$\varphi_{j-1, k-1}$] (33) at (2, 0) {};
			\draw (30) to[out=-60,in=130] (31);
			\draw (30) -- (32);
			\draw (30) -- (33);
			\node [circ,label=center:$l$, label=below:$\varphi_{j, k}$] (40) at (8, 0) {};
			\node [pent,label=center:$1$] (41) at (10, 2) {};
			\node [squa,label=center:$2$] (42) at (10, 0) {};
			\node [circ,label=center:\scriptsize$l+1$, label=below:$\varphi_{j, k-1}$] (43) at (10, -2) {};
			\draw (40) -- (41);
			\draw (40) -- (42);
			\draw (40) -- (43);
		\end{tikzpicture}
		
		\caption{\label{fig5}All $4$ first type trees and an example of the second type.}
	\end{figure}
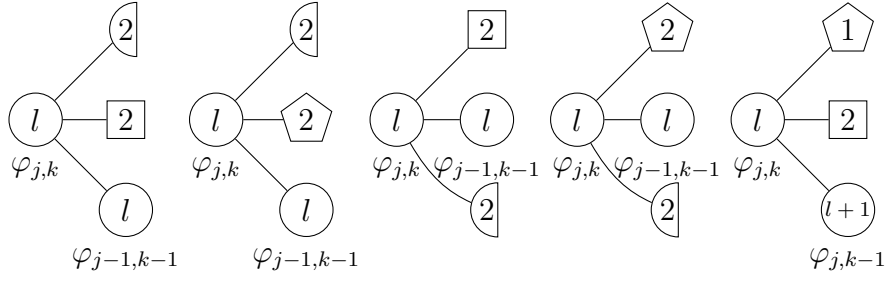

	Finally, we deal with the first class. Since $A$ and $B$ should share the same label, they both should have large denominators. In this case, there are no cancellations and $B$ can be chained, so we have to use the condition on $b_2$ from Lemma \ref{lema51} to deal with this problem. Under this assumption, it can be verified that the inductive bound holds by directly computing the tree contribution, if we increase $C_7$.
	
	\section{Dealing with initial conditions}
	\label{sec6}
	Having proven the upper bound on the $q$ and $\varphi$, we come to the lower bound. Specifically, we want to prove the following:
	
	\begin{lemma}
		For any $C$, $D$, for any large enough $n_0$, there exist positive $C_1$, $C_7$ and $C_8$, such that the following holds. Lets assume we are given $\left\{q_{2n+1}\right\}$ and $\left\{b_{2n}\right\}$ that satisfy the conditions of Lemma \ref{lema51}. Moreover, assume that
		\begin{itemize}
			\item $e^{C - 0.01} \le |q_3|$.
		\end{itemize}
		Then, if one increases $q_{2n_0+1}$ by $C_8$, the solution will satisfy:
		\begin{equation}
			\left\|\mathbf{q}_{n} - T\mathbf{q}_{n-1} + S_{2n_0+1}^{-1}\frac{C_8}{(2n-2n_0-1)!} \mathbf{e}_0 \right\| < \frac{1}{n^{4/5}} \left\| \mathbf{q}_{n} \right\|, \; \; n > n_0,
			\label{eq61}
		\end{equation}
	
	Moreover, if we define 
	
	\begin{equation}
		V(n) = \begin{cases}
			e^{W(n-2, C_1)}n!e^{n(C+F(\lambda) + 0.01)}, \; \; n < 2n_0+1\\
			2C_8, \; \; n = 2n_0+1\\
			\left\| \mathbf{q}_{n/2}\right\| S_{n}, \; \; n \ge 2n_0+2, \; \text{even}\\
			\sqrt{V(n-1)V(n+1)}, \; \; n \ge 2n_0+2, \; \text{odd}
		\end{cases},
	\end{equation} 

	$V(n)$ will be a supermultiplicative function in a sense of Definition \ref{def61}, and the following bounds will be satisfied:
	
	\begin{equation}
		|q_{n}| \le V(n), \; \; \; |\varphi_{j, k}| \le G_{j, k} V(j+k+1), \; n > 1, \; j+k>1.
		\label{eq63}  
	\end{equation}
	The bounds on $C_8$ are given by \eqref{eq624} and \eqref{eq627}.
	\label{lema61}
	
	\end{lemma}

	Here, supermultiplicity means the following:
	
	\begin{definition}
		A function $V(n)$ is called supermultiplicative if:
		\begin{itemize}
			\item $V(a)(a+1)(a+2)V(b)(b+1)(b+2) \le 20 C_{10}^{\min(a, b) - 3}V(a+b-3)(a+b-2)(a+b-1)V(3)$ for $a, b \ge 3$. 
			\item There exists $N_{11}$, dependent only on $C$ and $D$, such that the norm of $\mathbf{q}_{n}$, restricted to its first $N_{11}$ elements has at least half the norm of $\mathbf{q}_n$ for any $n > n_0$. 
		\end{itemize}
		\label{def61}
	\end{definition}
	
	Note that the second assertion of supermultiplicity gives that $V(n)$ is increasing in $n$: for $n \le 2n_0$ we can guarantee it by picking $C_1$ to be large enough, and for $n > 2n_0$ we can use
	
	\begin{equation}
		\left\|\mathbf{q}_{k+1}\right\| \ge 	\left\|\mathbf{q}_{k+1, [1, N_{11}]}\right\| \ge \frac{1}{N_{11}}\left\|\mathbf{q}_{k, [0, N_{11}-1]}\right\| \ge \frac{1}{2N_{11}} \left\|\mathbf{q}_{k}\right\|.
	\end{equation} 
	
	Here, we have used the second part of \eqref{eq585}: the next index of $\mathbf{q}$ shifts and scales elements of $\mathbf{q}_m$.
	
	So, the norm of $\mathbf{q}$ decreases by a constant, while the scaling (and hence $V(2k+2)$) increased by an order of $n^2$, compared to $V(2k)$. This suffices to prove an increasing by a factor of $n$ for $V(n)$, compared to $V(n-1)$, unless $n = 2n_0+1$ or $n = 2n_0+2$. For $2n_0+1$ it is trivial, based on bounds on $C_8$. For $2n_0+2$ it follows, since $q_{2n_0+1}$ has a non-trivial effect on $q_{2n_0+2}$. 
	
	We start the proof of Lemma \ref{lema61} by first applying Lemma \ref{lema51} for our values of $C$ and $D$ and $\delta = 0.01$. This gives us $C_1$ and bounds low order $q$ and $\varphi$. We are yet to choose $n_0$, but we may only use the results of Lemma \ref{lema51} for orders, lower than $2n_0+2$ ($2n_0+1$ for circles), since increasing $q_{2n_0+1}$ by $C_8$ will make it fail the requirements of the lemma. For higher orders we have to do a separate proof, although it would be similar to the on for Lemma \ref{lema51}. Also, we demand $n_0$ to be large enough, so that $W(n, C_1)$ becomes constant before it. So, $n_0$ is chosen after $C_1$. 
	
	We note that \eqref{eq63} holds trivially for squares due to Lemma \ref{lema51}, bounds on $q_{2n+1}$ and since the norm of the vector always bounds its first element. It also holds for circles with $j+k<2n_0+1$. 
	
	One important ingredient of the proof of Lemma \ref{lema51} was Lemma \ref{lema54}. After increasing $q_{2n_0+1}$ by $C_8$, its conditions are no longer met. We remark that we can still use this lemma, provided $C_8 \le (2n_0)^{4n_0}$, since the bound there was so inaccurate. Particularly, the factor $15m$ there can be turned into $12m$ since there are at most $12m$ non-leaf vertices in a tree. We also note that the total order of these bad diamond leaves of order $2n_0+1$ cannot exceed $3m$, so their total weight can be at most $(3m)^{6m}$. So, the lemma will still hold. 
	
	The plan of the proof of Lemma \ref{lema61} is once again induction in $n$. For every $n$, we first use \eqref{eq63} the same way we did the inductive bounds in Lemma \ref{lema51} and study non-principal trees. This allows us to justify the size of the error term in \eqref{eq61} and says that $T$ describes the behavior accurately. We then use $T$ to prove that $V$ continues to be supermultiplicative for our $n$, as well as to prove the bond on the circles. 
	
 	We start be considering small children trees of the box case ($q_{2n}$ for $n > n_0$). We want to prove that the total contribution of these trees to $q_{2n}$ is bounded by $n^{-5/6}V(2n)$. We apply exactly the same method we did for Lemma \ref{lema51}. We follow that proof up until Lemma \ref{lema52} (except we don't get the exponential with $2s_0$ power). In our case Lemma \ref{lema52} and Lemma \ref{lema53} take the following form:
 	
 	\begin{lemma}
 		For any $m \ge 2$, there is the following bound:
 		\begin{equation}
 			\sum_{a, b \ge 1}^{a+b = m} V(a+2)(a+3)(a+4)V(b+2)(b+3)(b+4) \le \frac{40}{1 - C_{10}} V(m+1)(m+2)(m+3) V(3).
  		\end{equation}
 	\end{lemma}  
	
	\begin{lemma}
		Let $k > \frac{2m}{3}$ and let $2 \le s \le m$. Then,
		\begin{equation}
			\sum_{1\le a_1,\ldots, a_s \le k}^{a_1+\ldots+a_s = m}\prod_{i = 1}^{s}V(a_i+2)(a_i+3)(a_i+4) \le \left(\frac{80}{1-C_{10}}\right)^sV(k+2)V(m - k - s + 4)V(3)^{s-2}(m+4)^4
		\end{equation}
		if $k < m - s + 2$ and 
		\begin{equation}
			\sum_{1\le a_1,\ldots, a_s \le k}^{a_1+\ldots+a_s = m}\prod_{i = 1}^{s}V(a_i+2)(a_i+3)(a_i+4) \le \left(\frac{80}{1-C_{10}}\right)^s V(m-s+3)V(3)^{s-1}(m+4)^2
		\end{equation}
		otherwise.
	\end{lemma}

	This allows us to mostly follow the proof of Lemma \ref{lema51} in this case with some minor adjustments. First, instead of factorials and $W$ in the first factor of \eqref{eq528}, as well as in \eqref{eq531}, we get $V(x)(x+1)(x+2)$. Hence, we don't have to use $99999/100000$ powers, so we use $1$ powers. Next, in our case $C_4$ will also depend on $C_{10}$, and hence on $C_1$, but we only use these computations for $n \ge 2n_0$ in our case, so we can just make $n_0$ large enough. We also move $V(3)$ from the previous lemma to $C_4$.
	
	The next notable remark comes after we start decreasing $\omega_0$ or $\omega_{\varphi}$ after \eqref{eq533}: there it did increase the contribution, since we had a regular $(\omega_0-1)!$ or $(\omega_\varphi - H(2n) - s_0+7)!$. However, in the case of $V$, this may not be the case (for example $V(2n_0+1)/V(2n_0)$ can be quite large). However, we can still set them to their lowest values, since we have that
	
	\begin{equation}
		V(n) \le e^{W(n-2, C_1)}n^{90n}e^{70Dn},
		\label{eq68}
	\end{equation}

	so decreasing it will not overpower the factor, coming from $\omega_b$. Similarly, we need to change $n/10$ and $n/5$ after \eqref{eq534} into $n/1000$ and $n/500$, respectively, so that we can use \eqref{eq68} to have the bound for low $\omega_0$ and $\omega_\varphi$. 
	
	The next difference comes when we start to increase $s_0$, when $\omega_0$ is close to $H(2n)$, but we can once again use \eqref{eq68} and directly set $s_0$ to $\omega_\varphi$. 
	
	The final difference comes, when we are considering the case $\omega_b = 0$. In the context of Lemma \ref{lema61}, the total contribution has the following formula (we remove fixed powers of $n$, since the bound will be stronger then them):
	
	\begin{equation}
		\frac{V(\omega_0)}{V(2n)} C_4^{s_0}\frac{(\omega_0-1)!}{s_0!(\omega_0-s_0-1)!}V(\omega_\varphi-s_0+3)
	\end{equation}
	
	with $\omega_0+\omega_\varphi = 2n$. Then, we can increase $\omega_0$ by $1$ and decrease $s_0$ and $\omega_\varphi$ by $1$. This will increase the contribution, unless $s_0 \le 2C_4C_{12}$, where $C_{12}$ is independent of $n_0$. We can do this until either $s_0$ becomes too small or $\omega_{0}$ becomes $H(2n)$. We have already considered the latter, so we focus on the former. In that case,
	
	\begin{equation}
		C_4^{s_0}\frac{(\omega_0-1)!}{s_0!(\omega_0-s_0-1)!}
	\end{equation} 

	is bounded by $n$ to the power of constant, independent of $n_0$, and we bound the rest with $20V(3)C_{10}^{2n-H(2n)}$ using supermultiplicity of $V$, thus finishing the low child case.
	
	Next, we come to the case of large child trees. We also want to bound their contribution by $n^{-5/6}V(2n)$. For this we use the same technique, as we did for the lower bound: we bound small children with Lemma \ref{lema54} and we bound the large child, using the inductive bound. The only difference would be that in the inductive bound earlier we had $n!$, whereas now we have $V(n)$. However, the only property of factorial we have used in that proof was that it was decreased by an order of $n$, when $n$ is decreased by $1$. Since the same property is true for $V$, the same proof works as well.
	
	The only remark we have to make is that we didn't consider trees with large diamond child in that case, since the contribution of diamonds was just analytic. In this case however, we have to consider trees with a diamond $q_{2n_0+1}$. We once again claim that trees should be “principal” (only have circles of order $2$ as children, apart from the diamond), otherwise their contribution will be bounded by $n^{-5/6}V(2n)$. 
	
	\subsection{Lower bound for principal trees}
	
	So, at least when the root of the tree is a square, we have reduced the study to just principal trees and these “principal” trees with a diamond child. Our next step would is to reduce the coefficients $R_{\sigma, n}$ to $\hat{R}_{\sigma, n}$. Using the same bounds, like in Lemma \ref{lema55}, we bound the difference contribution by $n^{-5/6}V(2n)$. We note that the same Lemma works for the diamond trees, their version of $\hat{R}$ is
	
	\begin{equation}
		 - \frac{S_{2n}S_{2n_0+1}^{-1}}{(2n-2n_0-1)!}.
	\end{equation}

	Similarly, we also consider $\hat{R}$ for the small children principal trees. Thus, we have proven 
	
	\begin{proposition}
		For the current value of $n$, the following holds:
		\begin{equation}
			|\mathbf{q}_{n, 0} - (T\mathbf{q}_{n-1})_0 + S_{2n_0+1}^{-1}\frac{C_8}{(2n-2n_0-1)!}| < n^{-4/5}.
		\end{equation}
	\end{proposition}

	Hence, the first assertion of Lemma \ref{lema61} for the current value of $n$ is proven (by definition of $T$ the error can only happen along $\mathbf{e}_0$). Using it, we will be able to prove that supermultiplicity of $V$ persists, if we extend it until $V(2n)$. 

	To do this, we will first require more information about an operator $T$. We want to prove that as we will continue iterating $T$ for a long time with some errors the solution will get and remain close to the eigenspace of the largest eigenvalue $\lambda = -4/\pi^2$. Specifically, we will introduce an invariant stable cone around this eigenspace, and claim that the solution will get inside of it at some point. Then, it will remain there and its norm will be scaled by a regular amount for every iteration. We construct this cone.
	
	First, denote $L_1$ to be a $-4/\pi^2$ eigenspace of $T$, $\dim L_1 = 1$ by Lemma \ref{lema56}. We also denote $L_2 = \text{ran} (T + 4/\pi^2)$. Since $T$ is compact, $L_2$ is a closed linear subspace, invariant under $T$. $T$, restricted to $L_2$, is a compact operator and its spectrum is bounded by $4/(9\pi^2)$. Since the codimension of $L_2$ is $1$ (the dimension of the kernel is $1$), we have that $\ell^2$ is a direct sum of $L_1$ and $L_2$. 
	
	Since the spectrum of $T$ on $L_2$ is bounded by $4/(9\pi^2)$, there exists $N_2$, such that:
	
	\begin{equation}
		\left\| T_{L_2}^{N_2}\right\| < \pi^{-2N_2}. 
		\label{eq613}
	\end{equation}

	We can then introduce the cone as:
	
	\begin{equation}
		\mathcal{\tilde{C}}_\theta=\left\{\mathbf{x} \in \ell^2 | \mathbf{x} = \mathbf{x}_1 + \mathbf{x}_2, \mathbf{x}_1 \in L_1, \mathbf{x}_2 \in L_2, ||\mathbf{x}_1|| > \theta ||\mathbf{x}_2|| \right\}.
	\end{equation}

	We have the invariance due to \eqref{eq613}: $T^{N_2}\mathcal{\tilde{C}}_\theta \subset \mathcal{\tilde{C}}_\theta$. Moreover, since we consider $T$ with an error, we have $T^{N_2}\mathcal{\tilde{C}}_\theta \subset \mathcal{\tilde{C}}_{2\theta}$. Since we want the cone to be invariant under $T$, we set:
	
	\begin{equation}
		\mathcal C_\theta = \left\{\mathbf{x} \in \ell^2 | \forall j = 0, \ldots, N_2 - 1: T^j\mathbf{x} \in \mathcal{\tilde{C}}_{2^{j/N_2}\theta} \right\}.
	\end{equation}

	\begin{lemma}
		For a fixed $\theta$, if $n$ is large enough, $\mathbf{x} \in \mathcal C_\theta$ and 
		\begin{equation}
			||\mathbf{y} - T\mathbf{x}|| < \frac{1}{n^{4/5}} ||\mathbf{y}||,
		\end{equation}
		
		then $\mathbf{y} \in \mathcal C_\theta$.
		\label{lema64}
	\end{lemma}
	
	\begin{proof}
		Denote $\mathbf{z} = \mathbf{y} - T\mathbf{x}$. We need to show $T^{j}\mathbf{y} \in \mathcal{\tilde{C}}_{2^{j/N_2}\theta}$ for every $j$. We have $T^{j}\mathbf{y} = T^{j+1} \mathbf{x} + T^{j}\mathbf{z}$. We decompose $\mathbf{z} = \mathbf{z}_1 + \mathbf{z}_2$. Since projectors are bounded operators, the norms of both of those vectors (and hence of their images under $T^j$) are bounded by
		
		\begin{equation}
			\frac{O_{N_2}(1)}{n^{4/5}}||\mathbf y||.
		\end{equation}
	
		However, we also know that $T^{j+1}\mathbf{x} \in \mathcal{\tilde{C}}_{2^{(j+1)/N_2}\theta}$. Hence, by applying triangle inequality, we are only left to prove:
		
		\begin{equation}
			||T^{j+1}\mathbf{x}_1|| - \frac{O_{N_2}(1)}{n^{4/5}}||\mathbf y|| > 2^{j/N_2}\theta \left(2^{-(j+1)/N_2}\theta^{-1}||T^{j+1}\mathbf{x}_1|| +   \frac{O_{N_2}(1)}{n^{4/5}}||\mathbf y||\right).
		\end{equation}
	
		This reduces to 
		
		\begin{equation}
			||T^{j+1}\mathbf{x}_1|| > \frac{O_{N_2, \theta}(1)}{n^{4/5}} ||\mathbf{y}||,
		\end{equation}
	
		which is trivial. 
	
	\end{proof}
	
	Next, we need to prove that the solution to our problem gets into the cone and remains there. First, we will prove it for the system without errors. Hence, we consider the sequence:
	
	\begin{equation}
	\mathbf{\hat{q}}_{n_0} = 0, \; \; \mathbf{\hat{q}}_{n} = T \mathbf{\hat{q}}_{n-1} - S_{2n_0+1}^{-1} \frac{C_8}{(2n-2n_0-1)!}, \; n > n_0. 
	\end{equation}

	The exact values of $C_8$ and $S_{2n_0+1}$ don't matter (as long as they are non-zero), since they just scale the solution. If we write down the generating function $\mathbf{\hat{Q}}(z) = \sum_{n > n_0} \mathbf{\hat{q}}_{n, 0} z^{2n-2n_0 - 1}$, then if we had only deformed with $-S_{2n_0+1}^{-1}C_8$ at $n = n_0+1$, we would have $\mathbf{\hat{Q}}(z) = -S_{2n_0+1}^{-1}C_8z/\cosh(z)$, as we have seen in \eqref{eq418}. But in our case, we deform on every step, and since the formula is linear, we can express the solution in the following sum:
	
	\begin{equation}
		\mathbf{\hat{Q}}(z) = -S_{2n_0+1}^{-1}C_8\sum_{j = 1}^{\infty} \frac{z^{2j+1}}{(2j+1)!}/\cosh(z) \Rightarrow \mathbf{\hat{Q}}(z) = -S_{2n_0+1}^{-1}C_8 \tanh(z).
	\end{equation}
	
	Since $\tanh(z)$ has poles at $\pm i\pi /2$, we see that $\mathbf{\hat{q}}_{n, 0}$ has a subsequence, growing faster than $(4/\pi^2 - 0.1)^{n-n_0}$. If we treat $\mathbf{\hat{q}}_n$ as being close to $T \mathbf{\hat{q}}_{n-1}$ when $n-n_0$ is large, we see that the $L_1$ component of some $\mathbf{\hat{q}}_n$ will become arbitrary larger than their $L_2$ component, hence this $\mathbf{\hat{q}}_n \in \mathcal{C}_{\theta}$. Analogously to Lemma \ref{lema64}, we deduce that all the following elements will also lie in $\mathcal{C}_{\theta}$, if $n-n_0$ is large enough. Denote this difference as $N_7$. Note that this quantity doesn't depend on the original choice of $n_0$.
	
	Since $\mathcal{C}_{\theta}$ is an open set, we see that we can choose $n_0$ to be large enough, so that even if we deform by $n^{-4/5}$ on every step, $\mathbf{q}_n$ will still get into (and hence remain) in the cone $\mathcal{C}_\theta$. 
	
	Using this information, we an now verify that $V$ is supermultiplicative. First, we can preemptively choose $\theta$, independently of other parameters, such that if $\mathbf{x}\in \mathcal{C}_\theta$, then:
	
	\begin{equation}
		\frac{3.99}{\pi^2} < \frac{|(T\mathbf{x})_0|}{|\mathbf{x}_0|} < \frac{4.01}{\pi^2}.
		\label{eq622}
	\end{equation} 

	So, once $\mathbf{q}_n$ gets into the cone, we can write the same inequality for the neighboring elements. The second assertion of supermultiplicity follows immediately from it: once we get into the cone, this property is true for the eigenvector, and hence it should be true for vectors, close to it. Before we get into the cone, there are only a constant number of iterations, dependent only on $C$ and $D$. Hence, we can just choose a constant $N_{11}$ to cover those cases. 
	
	Now we need to verify the first assertion of supermultiplicity. Note that when $a+b-3 < 2n_0+1$ it is verified by Lemma \ref{lema52}, hence we should only consider the case $a+b-3 > 2n_0$. Then, from \eqref{eq622} for $\mathbf{q}$, we can get the following bound for $n > 2n_0$:
	
	\begin{equation}
	\frac{C_8}{C_{12}} \times \frac{n!}{(2n_0+1)!} e^{(n - 2n_0 - 1) (C+F(\lambda) -0.03)} <	V(n) < C_8C_{12} \times \frac{n!}{(2n_0+1)!} e^{(n - 2n_0 - 1) (C+F(\lambda) +0.03)} ,
	\label{eq623}
	\end{equation}

	where $C_{12}$ is a constant, dependent only on $C$ and $D$, included to make the statement true before the solution gets to a cone. Hence we see that provided
	
	\begin{equation}
		C_8 > (2n_0+1)! e^{W(2n_0-1, C_1)} e^{(2n_0+1)(C+F(\lambda) - 0.03)} e^{0.2n_0},
		\label{eq624}
	\end{equation}

	we will have that
	
	\begin{equation}
		V(n) > e^{W(n-2, C_1)}n!e^{n(C+F(\lambda) +0.01)}
	\end{equation}

	for $n > 2n_0$. So, if both $a$ and $b$ are less than $2n_0+1$, the condition still holds. Now assume that $a < 2n_0+1 \le b$. When $b$ already corresponds to a vector in a cone, we can increase it by $1$ and decrease $a$ by $1$, and it will increase the product by a fixed factor. If $b$ is to small to correspond to a cone vector, then if $a < \sqrt{n_0}$ we can still use the same argument, and when $a\ge \sqrt{n_0}$ then doing this change $\sqrt{n_0}$ times will compensate for a factor of type $C_{12}$ we get before we get into the cone. 
	
	Finally, we have to consider the case $2n_0+1\le a\le b$. Since $a+b-3$ corresponds to a vector in a cone, we have:
	
	\begin{equation}
		V(a+b-3)/V(b) \ge \frac{1}{C_{12}}\frac{(a+b-3)!}{b!} e^{(a-3) (C + F(\lambda) - 0.03)} \ge \frac{1.9^a}{C_{12}} a! e^{(a-3)(C+F(\lambda) - 0.03)}. 
	\end{equation} 
	
	To bound $V(a)$ from above, we can use \eqref{eq623}, and we get that the needed inequality, provided
	
	\begin{equation}
		C_8 < (2n_0+1)! 1.7^{2n_0} \frac{1}{C_{12}^2} e^{(2n_0+1)(C+F(\lambda))},
		\label{eq627}
 	\end{equation}
 
 	that is compatible with \eqref{eq624}, thus proving supermultiplicity. We are only left to check that $\varphi_{j, k}$ satisfies the bounds as well. The proof of this in similar to the circle case for Lemma \ref{lema51}, so we will not put it here.
		
	\appendix
	
	\section{Determining $\varphi$ up to order $3$ \label{ap1}}
	
	In this section, we will look at error of \eqref{maineq} in $z^2w^0$, $z^1w^1$ and $z^0w^2$ and find $\varphi_{2, 0}$, $\varphi_{1, 1}$ and $\varphi_{0, 2}$ respectively. During this section, we can approximate $b$ with $b_0$, since higher orders of $b$ only appear in the third order of the conjugate map. First, we will consider error terms, coming from linear approximation of $\varphi$:
	
	\begin{align}
		q'\left(\frac{(\lambda^{-1}+1) z + (\lambda + 1)w}{2}\right)\cos\left(\frac{-(\lambda^{-1} - 1)z - (\lambda - 1)w}{2} - \alpha\right) + \\ + q\left(\frac{(\lambda^{-1} + 1)z + (\lambda+1)w}{2}\right)\cos'\left(\frac{-(\lambda^{-1} - 1)z - (\lambda - 1)w}{2} - \alpha\right) + \\ + q'\left(\frac{(\lambda+1)z + (\lambda^{-1}+1)w}{2}\right)\cos\left(\frac{(\lambda - 1)z + (\lambda^{-1} - 1)w}{2} - \alpha\right) - \\ -q\left(\frac{(\lambda+1)z + (\lambda^{-1} + 1)w}{2}\right)\cos'\left(\frac{(\lambda - 1)z + (\lambda^{-1} - 1)w}{2} - \alpha\right).
	\end{align} 
	
	To efficiently compute the coefficient, we will separately study terms coming with $q_0$, $q_2$ and $q_3$. Starting with $q_0$, we have
	
	\begin{align}
		\cos'\left(\frac{-(\lambda^{-1} - 1)z - (\lambda - 1)w}{2} - \alpha\right)  - \cos'\left(\frac{(\lambda - 1)z + (\lambda^{-1} - 1)w}{2} - \alpha\right).
	\end{align} 

	We must take the quadratic term of $\cos'$ at $-\alpha$ to achieve the second order term:
	
	\begin{align}
		-\frac{\sin \alpha}{8}\left((\lambda^{-1} - 1)z + (\lambda - 1)w\right)^2  + \frac{\sin \alpha}{8}\left((\lambda - 1)z + (\lambda^{-1} - 1)w\right)^2.
	\end{align} 

	So, we get $0$ contribution to $z^1w^1$ and the contribution to $z^2w^0$ is ($z^0w^2$ has minus this)
	
	\begin{align}
		\frac{\sin \alpha}{8}\left(\lambda - 2 + \lambda^{-1}\right)\left(\lambda - \lambda^{-1}\right).
	\end{align}
	 
	 Next, we study terms with $q_2$:
	 
	 \begin{align}
	 	2q_2\left(\frac{(\lambda^{-1}+1) z + (\lambda + 1)w}{2}\right)\left(\frac{-(\lambda^{-1} - 1)z - (\lambda - 1)w}{2}\right)\sin \alpha - \\ - q_2\left(\frac{(\lambda^{-1} + 1)z + (\lambda+1)w}{2}\right)^2\sin \alpha + \\ + 2q_2\left(\frac{(\lambda+1)z + (\lambda^{-1}+1)w}{2}\right)\left(\frac{(\lambda - 1)z + (\lambda^{-1} - 1)w}{2}\right)\sin \alpha + \\ +q_2\left(\frac{(\lambda+1)z + (\lambda^{-1} + 1)w}{2}\right)^2\sin \alpha.
	 \end{align} 
	
	The symmetry between $z$ and $w$ is evident, so we will get $0$ in $z^1w^1$ contribution. In $z^2w^0$ we will get ($z^0w^2$ has minus this):

	\begin{align}
		\frac{3\lambda^2 + 2\lambda - 2\lambda^{-1} - 3\lambda^{-2}}{4}q_2\sin \alpha.
	\end{align} 
	
	Finally, we study terms with $q_3$:
	
	\begin{align}
		3q_3\left(\frac{(\lambda^{-1}+1) z + (\lambda + 1)w}{2}\right)^2\cos \alpha  + 3q_3\left(\frac{(\lambda+1)z + (\lambda^{-1}+1)w}{2}\right)^2\cos \alpha.
	\end{align} 
	
	Particularly, in $z^2w^0$ and in $z^0w^2$ we get:
	
	\begin{align}
		\frac{3}{4}q_3\left((\lambda+1)^2 + (\lambda^{-1}+1)^2\right)\cos \alpha.
	\end{align}

	And in $z^1w^1$ we have: 
	\begin{align}
		3q_3(\lambda + 1)(\lambda^{-1} + 1)\cos \alpha.
	\end{align} 

	The last step is to consider the dependence of the error on the quadratic coefficients of $\varphi$. Since we require the error to be $0$, we will find these coefficients by solving a linear equation. Because of Remark \ref{rem1}, only $\varphi_{2, 0}$ can contribute to $z^2w^0$, and it is the same for other $2$ pairs. If we require the usage of quadratic term of $\varphi$, up to degree $3$ the left side of \eqref{maineq} reduces to:
	
	\begin{align}
		q'\left(\frac{t_- + t}{2}\right)\cos \alpha + \cos'\left(\frac{t - t_-}{2} - \alpha\right)  + q'\left(\frac{t_+ + t}{2}\right)\cos \alpha - \cos'\left(\frac{t_+ - t}{2} - \alpha\right).
	\end{align} 

	Then, the contribution of $\varphi_{2, 0}$ to $z^2w^0$ is:
	
	\begin{align}
		\varphi_{2, 0} \left( q_2\left(\lambda^2 + 2 + \lambda^{-2}\right) - \frac{\lambda^2 - 2 + \lambda^{-2}}{2} \right) \cos \alpha.
	\end{align} 

	To get a contribution of $\varphi_{0, 2}$ to $z^0w^2$ one has to swap it with $\varphi_{2, 0}$ in the formula and the contribution of $\varphi_{1, 1}$ to $z^1w^1$ is $4\varphi_{1, 1}q_2\cos\alpha$.
	
	Thus, we can now find the quadratic coefficients of $\varphi$:
	
	\begin{align}
		\varphi_{2, 0} = \overline{\varphi_{0, 2}} =\\=  -\frac{\frac{\tan \alpha}{8}\left(\lambda - 2 + \lambda^{-1}\right)\left(\lambda - \lambda^{-1}\right) + \frac{3\lambda^2 + 2\lambda - 2\lambda^{-1} - 3\lambda^{-2}}{4}q_2\tan \alpha + \frac{3}{4}q_3\left((\lambda+1)^2 + (\lambda^{-1}+1)^2\right) }{q_2\left(\lambda^2 + 2 + \lambda^{-2}\right) - \frac{\lambda^2 - 2 + \lambda^{-2}}{2} }.
	\end{align}  

	\begin{equation}
		\varphi_{1, 1} = -\frac{3q_3(\lambda + 1)(\lambda^{-1} + 1)}{4q_2}.
	\end{equation}
	
	The expression $\varphi_{2, 0} \cos b_0 + \varphi_{1, 1} + \varphi_{0, 2}\cos b_0$ will arise in Section \ref{sec4} and it will be important to verify that it is non-zero. It reduces to:
	
	\begin{equation}
		-\frac{3q_3\left((\lambda+1)^2 + (\lambda^{-1}+1)^2\right) (\lambda + \lambda^{-1})}{q_2\left(\lambda^2 + 2 + \lambda^{-2}\right) - \frac{\lambda^2 - 2 + \lambda^{-2}}{2} } - \frac{3q_3(\lambda + 1)(\lambda^{-1} + 1)}{4q_2}.
	\end{equation}

	Factoring out, we get:
	
	\begin{equation}
		-3q_3(\lambda+1)(\lambda^{-1}+1)\left(\frac{1}{q_2- \frac{\lambda^2 - 2 + \lambda^{-2}}{2\left(\lambda^2 + 2 + \lambda^{-2}\right) } } + \frac{1}{4q_2}\right).
	\end{equation}

	Both denominators are positive, so this value is non-zero, provided $q_3 \ne 0$.
	
	\section{Comments about the original Treschev problem}
	\label{apb}
	
	As stated earlier, Theorem \ref{th2} only works for the case $q_3\ne0$, so we don't have a lower bound for the original Treschev problem. However, one can still try to estimate or guess the rate of growth in that case. To do that one can use more accurate bounds in the estimates, simplify the problem and do some numerics. After doing this, it lead us to consider Conjecture \ref{conj1}. 
	
	Firstly, since in the main order the recurrence cancels out, as we have seen in Section \ref{sec4}, we may want to consider the next order, corresponding to half the original Gevrey order growth ($n \log n /2 $). This involves applying better bounds in formulas like \eqref{eq49} to get the next order term out of the error and considering more trees. After various rescalings and identities, the second order also seems to reduce to \eqref{eq521}, so we still get the cancellations, hence Gevrey order can be lower than $n \log n /2$. Potentially, one may be able to iterate this scheme, getting lower and lower Gevrey orders. Unfortunately, we were unable to rigorously do that, since the number of estimates one should do grows very fast.
	
	Another interesting idea is to simplify the problem, while capturing the essence of it. Since we know that principal trees can control the growth, we may only consider them. So, we set $\varphi(z, w) = z + w + z^3 + w^3$, since in $q_3 = 0$ case only $\varphi_{3, 0}$ and $\varphi_{0, 3}$ appear in principal trees. Moreover, we simplify the problem further, by removing parameters like $\lambda$, $c_i$, etc. This results in the following:
	
	\begin{question}
		Let $\varphi(z, w) = z + w + z^3 + w^3$. Assume $q(t)$ is an odd power series, such that
		\begin{equation}
			\frac{\partial^{2n-1}}{\partial z^n w^{n-1}} q(\varphi(z, w)) = \begin{cases}
				1, \; \; n = 2 \\0, \; \; n\ne 2
			\end{cases}.
		\end{equation}
	What is the behavior of $q_{2n-1}$ as $n \rightarrow \infty$?
	\label{qu2}
	\end{question}
	
	This question is similar to an inverse function problem discussed in Section \ref{sec1}, but $\varphi$ depends on $2$ variables and we are interested in the diagonal coefficients of the composition. We note that we consider $q$ to be odd, whereas in the original problem $q$ was even, since we are taking the derivative of it in \eqref{maineq}. Also, we are considering $n=2$ as opposed to $n=1$ since otherwise $q(t) = t$ would give the solution. 
	
	\begin{figure}
		\includegraphics[width=15cm]{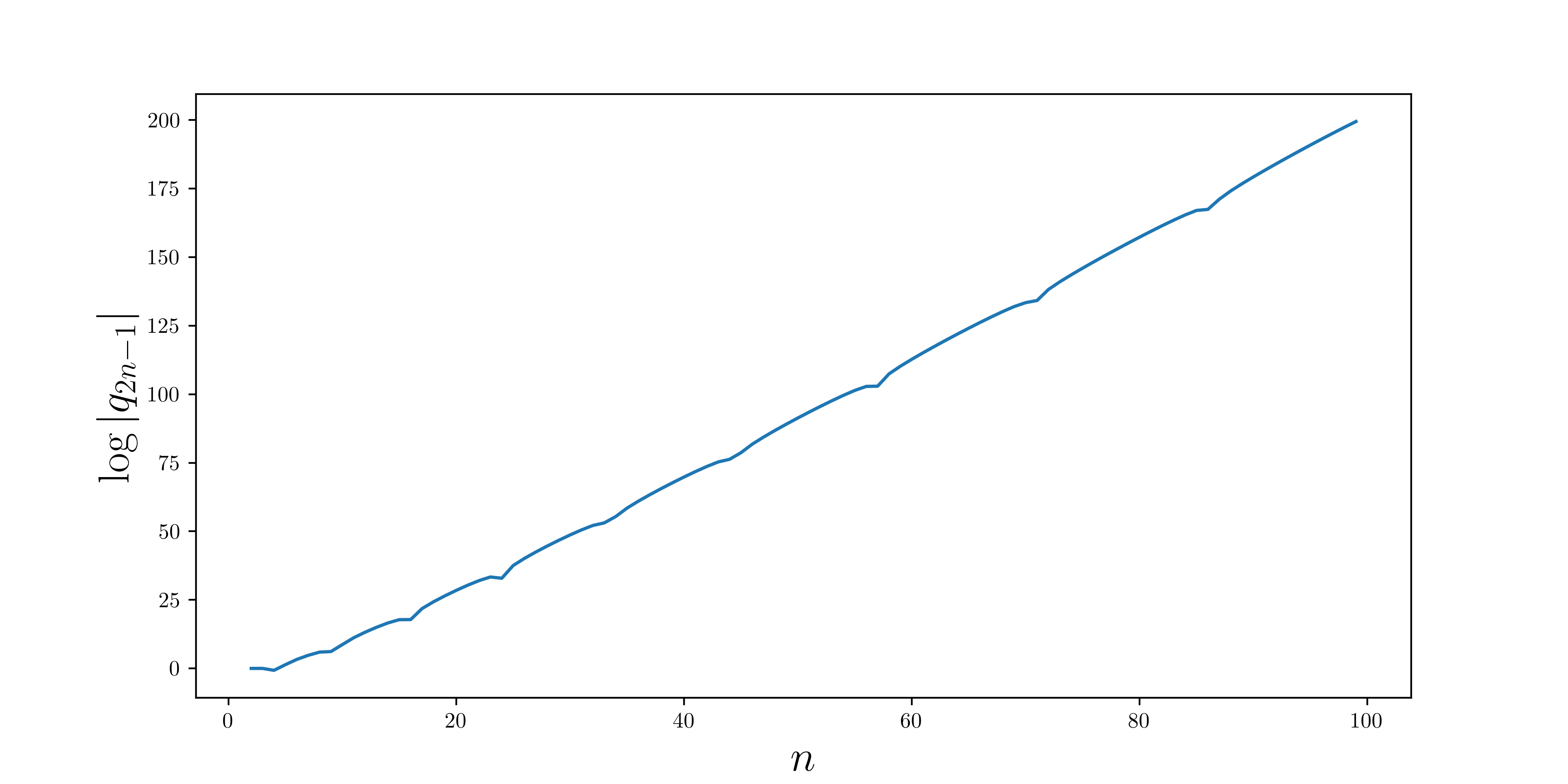}
		\centering
		\caption{Starting terms of $q$ in Question \ref{q2} in logarithmic scale. The graph is slightly curved up. The data up to $n = 1000$ shows the same picture, here $n \le 100$ to make dips visible.}
		\label{djdjd}
	\end{figure}
	
	Question \ref{qu2} has many similarities to the Treschev problem. For example, the trivial upper bound gives Gevrey $n \log n$ growth for $q_{2n-1}$, but the recurrent cancels out. As a benefit, it is much less messy, since one doesn't have to deal with many trees and Diophantine problems. Moreover, it is much easier to do numerics for it. 
	
	Still, it is unclear whether $q_{2n-1}$ grow exponentially or faster. The graph in Figure \ref{djdjd} is slightly curved to the top and this may indicate faster than exponential growth. However, if the solution was Gevrey of some significant order, one would expect the graph to be more curved. This shows that there are a lot of cancellations in thus problem. So, the answer may be somewhere in the middle between exponential and Gevrey growth. To guess this, one can apply various transforms to the graph to better visualize the growth rate.  
	
	We have also found an interesting pattern in this problem. If we compute the ratios $r_n = q_{2n+1}/q_{2n-1}$ for $n$ up to a $1000$, we get the following. Usually, $r_n$ is negative. As $n$ increases, $r_n$ increases as well. This continues up until $r_n$ becomes positive. When this happens, the next $r_n$ jumps and becomes a big negative number. After this, the new cycle of increasing begins. These sudden jumps manifest themselves as dips in the graph. If $q_{2n-1}$ corresponds to a dip, then the distance to the next dip is approximately $C\sqrt{n}$. The behavior of $q_{2n-1}$ within one cycle is pretty regular: the increasing slows down in the middle of the cycle and speeds up around dips. Several consecutive truncated values of $r_n$ near a dip are given below. Perhaps there exists some limiting regime that can explain this pattern and shed some light onto Question \ref{qu2}. 
	\begin{equation*}
		-8.8 \;\; -8.5 \;\; -8.1\;\;-7.6\;\;-6.7\;\;-5.1\;\;-0.7\;\;\;108.5\;\;-20.2\;\;-14.6\;\;-12.8\;\;-11.9\;\;-11.3
	\end{equation*}
	
	One can also play around with different versions of $\varphi(z, w)$ in Question \ref{qu2}. It turns out that if one sets $\varphi(z, w) = z + w + z^2w + zw^2$ or $\varphi(z, w) = z + w + (z+w)^3$, then the problem can be reduced to being $1$-dimensional: inverting the function $x + x^3$. As such, in those cases $q(t)$ is analytic and we get an exponential growth. Particularly, we can obtain an explicit formula for $q_{2n-1}$ and we don't have any dips. We believe this to be rather rare among $\varphi$-s. Adding even terms to $\varphi(z, w)$ is expected to have the same effect as adding $q_{odd}$ into the Treschev problem, potentially resulting in Gevrey growth.    
	
	We have tried several methods to tackle Question \ref{qu2}. For example, we have tried to guess the solution by imagining $q$ to be a hyper-geometric function and to use some complex analytic techniques. Still, we cannot get anything concrete about this problem.  
	
	\section{Diophantine cancellations}
	\label{apc}
	In Sections \ref{sec5} and \ref{sec6} we have heavily used the Diophantine property of $\lambda$ to bound the contributions of many trees. The proof wouldn't work in its current form, if that condition is relaxed. In KAM-theory one can get a much better condition, due to inherent cancellations. Here, we will show that the same cancellations also appear in this method, so $\lambda$ doesn't have to have restricted partial quotients. We will assume that $b$ is trivial, since it is vital for cancellations. This should also explain why the choice of $\varphi_{k+1, k}$ doesn't matter (for any $b$). 
	
	We can represent the possible monomials in $z$ and $w$ as a $2$-dimensional grid. The $x$-coordinate of a point will represent the power of $z$, and $y$-coordinate -- that of $w$. Every cell in the grid corresponds to the element of the series for $q$ or $\varphi$ ($z^3w$ corresponds to $\varphi_{3, 1}$). If we visualize the cells of the grid, where $\lambda$ creates the small denominator, we will get a family of diagonals, representing $\varphi$. These diagonals will come in pairs, but otherwise they will be spread out, since small denominators are not frequent. 
	
	The circle vertex can directly affect (be a direct on in a tree) only elements of the grid with greater or equal $x$ and $y$-coordinates (boxes just need the sum of coordinates to be greater or equal). Hence, it is not beneficial to move from one small denominator diagonal to another (from another pair): the order has to increase at least by the distance between the diagonals, while you only get a factor of $n^\tau$ for doing this. One can also move from one diagonal to another, stopping at the box, but that involves first going to the main diagonal and hence wastes order. 
	
	The only other (and the most intuitive way) of using these diagonals to get a fast growth is to travel within the diagonal pair using small steps. Theoretically, on every step the value will get divided by a small denominator, and that will give a fast growth. However, this won't happen due to cancellations.
	
	Specifically, if we will only consider the linear contribution of $\varphi_{j, k}$ on the small denominator diagonal to other elements of the grid, we will get the similar grid to the original one, it will be just scaled and shifted with an error. The reason for this is simple: in every tree that features $\varphi_{j, k}$ once one can do a transformation by changing the $\varphi_{j, k}$ vertex into $\varphi_{1, 0}$ or $\varphi_{0, 1}$, depending on which diagonal in a pair $\varphi_{j, k}$ was, and shifting all the circle labels accordingly.
		\begin{figure}[H]
		\centering
		\begin{tikzpicture}[scale = 0.6, box/.style={rectangle,draw=black,thin, minimum size=0.6cm}]
			\draw[step=1.0,black,thin] (0,0) grid (25.5,14.5);
			\node[box,fill=lightgray] at (1.5,0.5){};
			\node[box,fill=lightgray] at (2.5,1.5){};  
			\node[box,fill=lightgray] at (3.5,2.5){};  
			\node[box,fill=lightgray] at (4.5,3.5){};  
			\node[box,fill=lightgray] at (5.5,4.5){};  
			\node[box,fill=lightgray] at (6.5,5.5){};  
			\node[box,fill=lightgray] at (7.5,6.5){};  
			\node[box,fill=lightgray] at (8.5,7.5){};  
			\node[box,fill=lightgray] at (9.5,8.5){};  
			\node[box,fill=lightgray] at (10.5,9.5){};
			\node[box,fill=lightgray] at (11.5,10.5){};  
			\node[box,fill=lightgray] at (12.5,11.5){};  
			\node[box,fill=lightgray] at (13.5,12.5){};  
			\node[box,fill=lightgray] at (14.5,13.5){};  
			\node[box,fill=lightgray] at (0.5,1.5){};
			\node[box,fill=lightgray] at (1.5,2.5){};
			\node[box,fill=lightgray] at (2.5,3.5){};  
			\node[box,fill=lightgray] at (3.5,4.5){};  
			\node[box,fill=lightgray] at (4.5,5.5){};  
			\node[box,fill=lightgray] at (5.5,6.5){};  
			\node[box,fill=lightgray] at (6.5,7.5){};  
			\node[box,fill=lightgray] at (7.5,8.5){};  
			\node[box,fill=lightgray] at (8.5,9.5){};  
			\node[box,fill=lightgray] at (9.5,10.5){};  
			\node[box,fill=lightgray] at (10.5,11.5){};
			\node[box,fill=lightgray] at (11.5,12.5){};  
			\node[box,fill=lightgray] at (12.5,13.5){}; 
			\node[box,fill=yellow] at (8.5,0.5){};
			\node[box,fill=yellow] at (9.5,1.5){};
			\node[box,fill=yellow] at (10.5,2.5){};
			\node[box,fill=yellow] at (11.5,3.5){};
			\node[box,fill=yellow] at (12.5,4.5){};
			\node[box,fill=yellow] at (13.5,5.5){};
			\node[box,fill=yellow] at (14.5,6.5){};
			\node[box,fill=yellow] at (15.5,7.5){};
			\node[box,fill=yellow] at (16.5,8.5){};
			\node[box,fill=yellow] at (17.5,9.5){};
			\node[box,fill=yellow] at (18.5,10.5){};
			\node[box,fill=yellow] at (19.5,11.5){};
			\node[box,fill=yellow] at (20.5,12.5){};
			\node[box,fill=yellow] at (21.5,13.5){};
			\node[box,fill=yellow] at (10.5,0.5){};
			\node[box,fill=yellow] at (11.5,1.5){};
			\node[box,fill=yellow] at (12.5,2.5){};
			\node[box,fill=yellow] at (13.5,3.5){};
			\node[box,fill=yellow] at (14.5,4.5){};
			\node[box,fill=yellow] at (15.5,5.5){};
			\node[box,fill=yellow] at (16.5,6.5){};
			\node[box,fill=yellow] at (17.5,7.5){};
			\node[box,fill=yellow] at (18.5,8.5){};
			\node[box,fill=yellow] at (19.5,9.5){};
			\node[box,fill=yellow] at (20.5,10.5){};
			\node[box,fill=yellow] at (21.5,11.5){};
			\node[box,fill=yellow] at (22.5,12.5){};
			\node[box,fill=yellow] at (23.5,13.5){};
			\node[box,fill=yellow] at (0.5,8.5){};
			\node[box,fill=yellow] at (1.5,9.5){};
			\node[box,fill=yellow] at (2.5,10.5){};
			\node[box,fill=yellow] at (3.5,11.5){};
			\node[box,fill=yellow] at (4.5,12.5){};
			\node[box,fill=yellow] at (5.5,13.5){};
			\node[box,fill=yellow] at (0.5,10.5){};
			\node[box,fill=yellow] at (1.5,11.5){};
			\node[box,fill=yellow] at (2.5,12.5){};
			\node[box,fill=yellow] at (3.5,13.5){};
			\node[box,fill=yellow] at (22.5,0.5){};
			\node[box,fill=yellow] at (23.5,1.5){};
			\node[box,fill=yellow] at (24.5,2.5){};
			\node[box,fill=yellow] at (24.5,0.5){};
			\node[text width=3cm] at (2.67,-0.5){$z^0$};
			\node[text width=3cm] at (3.67,-0.5){$z^1$};
			\node[text width=3cm] at (4.67,-0.5){$z^2$};
			\node[text width=3cm] at (5.67,-0.5){$z^3$};
			\node[text width=3cm] at (6.67,-0.5){$z^4$};
			\node[text width=3cm] at (7.67,-0.5){$z^5$};
			\node[text width=3cm] at (8.67,-0.6){$\ldots$};
			\node[text width=3cm] at (1.57,0.5){$w^0$};
			\node[text width=3cm] at (1.57,1.5){$w^1$};
			\node[text width=3cm] at (1.57,2.5){$w^2$};
			\node[text width=3cm] at (1.57,3.5){$w^3$};
			\node[text width=3cm] at (1.57,4.5){$w^4$};
			\node[text width=3cm] at (1.57,5.5){$w^5$};
			\node[text width=3cm] at (1.73,6.5){$\vdots$};
			\node[text width=3cm] at (1.37,10.5){$w^{10}$};
			\node [connection] (1) at (3.5, 2.5) {};
			\node [connection] (2) at (4.5, 3.5) {};
			\node [connection] (3) at (7.5, 6.5) {};
			\node [connection] (4) at (9.5, 8.5) {};
			\node [connection] (5) at (10.5, 9.5) {};
			\node [connection] (6) at (12.5, 11.5) {};
			\node [connection] (7) at (14.5, 13.5) {};
			\node [connection] (8) at (15.5, 14.5) {};
			\node [connection] (30) at (10.5, 3.5) {};
			\node [connection] (31) at (8.5, 11.5) {};
			\node[text width=3cm] at (6.51,2.65){\scriptsize  $\mathbf{\times n^2}$};
			\node[text width=3cm] at (9.51,5.65){\scriptsize  $\mathbf{\times n^6}$};
			\node[text width=3cm] at (11.51,7.65){\scriptsize  $\mathbf{\times n^4}$};
			\node[text width=3cm] at (12.51,8.65){\scriptsize  $\mathbf{\times n^2}$};
			\node[text width=3cm] at (14.51,10.65){\scriptsize  $\mathbf{\times n^4}$};
			\node[text width=3cm] at (16.51,12.65){\scriptsize  $\mathbf{\times n^4}$};
			\node [connection] (11) at (15.5, 5.5) {};
			\node [connection] (12) at (16.5, 6.5) {};
			\node [connection] (13) at (17.5, 7.5) {};
			\node [connection] (14) at (18.5, 8.5) {};
			\node [connection] (15) at (19.5, 9.5) {};
			\node [connection] (16) at (20.5, 10.5) {};
			\node [connection] (17) at (21.5, 11.5) {};
			\node [connection] (18) at (22.5, 12.5) {};
			\node [connection] (19) at (23.5, 13.5) {};
			\node [connection] (99) at (24.5, 13.5) {};
			\node [connection] (10) at (24.5, 14.5) {};
			\node [connection] (32) at (16.5, 10.5) {};
			\node [connection] (33) at (19.5, 12.5) {};
			\draw [->, thick] (1)--(2);
			\draw [->, thick] (2)--(3);
			\draw [->, thick] (3)--(4);
			\draw [->, thick] (4)--(5);
			\draw [->, thick] (5)--(6);
			\draw [->, thick] (6)--(7);
			\draw [->,  thick, dotted] (3)--(30);
			\draw [->,  thick, dotted] (5)--(31);
			\draw [thick] (7)--(8);
			\draw [->, dashed, thick, blue] (11)--(12);
			\draw [->, dashed, thick, blue] (12)--(13);
			\draw [->, dashed, thick, blue] (13)--(14);
			\draw [->, dashed, thick, blue] (14)--(15);
			\draw [->, dashed, thick, blue] (15)--(16);
			\draw [->, dashed, thick, blue] (16)--(17);
			\draw [->, dashed, thick, blue] (17)--(18);
			\draw [->, dashed, thick, blue] (18)--(19);
			\draw [->, dotted, thick, blue] (12)--(32);
			\draw [->, dotted, thick, blue] (15)--(33);
			\draw [->, dotted, thick, blue] (19)--(99);
			\draw [dashed, thick, blue] (19)--(10);
			\node[text width=3cm, blue] at (18.51,5.65){\scriptsize  $\mathbf{\times n^\tau}$};
			\node[text width=3cm, blue] at (19.51,6.65){\scriptsize  $\mathbf{\times n^\tau}$};
			\node[text width=3cm, blue] at (20.51,7.65){\scriptsize  $\mathbf{\times n^\tau}$};
			\node[text width=3cm, blue] at (21.51,8.65){\scriptsize  $\mathbf{\times n^\tau}$};
			\node[text width=3cm, blue] at (22.51,9.65){\scriptsize  $\mathbf{\times n^\tau}$};
			\node[text width=3cm, blue] at (23.51,10.65){\scriptsize  $\mathbf{\times n^\tau}$};
			\node[text width=3cm, blue] at (24.51,11.65){\scriptsize  $\mathbf{\times n^\tau}$};
			\node[text width=3cm, blue] at (25.51,12.65){\scriptsize  $\mathbf{\times n^\tau}$};
			\node [connection] (40) at (1.0, 10.0) {};
			\node [connection] (41) at (6.0, 10.0) {};
			\node [connection] (42) at (1.0, 14.5) {};
			\node [connection] (43) at (2.5, 10.5) {};
			\node [connection] (44) at (3.5, 10.5) {};
			\node [connection] (45) at (2.5, 11.5) {};
			\node [connection] (46) at (3.1, 11.5) {};
			\node [connection] (47) at (3.5, 11.1) {};
			\node [connection] (48) at (3.1, 11.1) {};
			\draw [ultra thick] (40)--(41);
			\draw [ultra thick] (40)--(42);
			\draw [->, thick, blue] (43)--(44);
			\draw [->, thick, blue] (44)--(47);
			\draw [->, thick, blue] (43)--(45);
			\draw [->, thick, blue] (45)--(46);
			\draw [->, thick, blue] (43)--(48);
			\draw [blue] (3.5,11.5) node {\huge $\mathbf{\times}$};
		\end{tikzpicture}
		\caption{\label{fig4}A visualization of a grid. The gray cells are associated to $q$, while all the others are to $\varphi$. The yellow cells have small denominators. Black arrows show an example of contribution growth from principal trees. Blue dashed arrows show the trees with $b_2$, competing with principal trees. On the top-left the cancellation is shown: if one shifts the thick lines to the origin and associates $z^2w^{10}$ to $z^1w^0$, its contributions will form a scaled version of the original grid, resulting in cancellations.}
	\end{figure}
	This transformation will be valid for almost every tree, and the multiplier of the tree will be changed regularly (if $\lambda$ had been rational, it would've just gotten scaled). The only exception for this principle are trees, that have roots at one of the diagonals in a pair: this transformation will turn them invalid, since the root would get label $\pm 1$. This just shows that the whole grid will get scaled and shifted, but these diagonal elements will be almost zero. It follows that the growth along the diagonal pair is not possible.

	\medskip
	
	\printbibliography

\end{document}